\documentclass[12pt,reqno]{amsart}

\usepackage[a4paper,top=2.75cm,bottom=2.75cm,left=3.cm,right=2.5cm]{geometry} %toglierlo

\usepackage{amssymb}
\usepackage{mathtools}
\usepackage{empheq}
\usepackage{graphicx}
\usepackage{times}
\usepackage{rsfso}
\usepackage{enumerate}
\usepackage{subeqnarray} 
\usepackage{subcaption}
\usepackage{xspace}
\usepackage{tikz}
\usepackage{dsfont}

\graphicspath{{../images/}}

\usepackage[colorlinks,citecolor=marin,linkcolor=rouge,
    bookmarksopen,
    bookmarksnumbered
]{hyperref}

\usepackage{color}
\definecolor{marin}  {rgb}{0.,   0.,  0.6}
\definecolor{rouge}  {rgb}{0.8,  0.,  0.}

\newtheorem{theorem}{Theorem}[section]
\newtheorem{lemma}[theorem]{Lemma}
\newtheorem{proposition}[theorem]{Proposition}

\makeatletter
\renewcommand\thetheorem{\thesection.\@Alph\c@theorem}
\makeatother

\theoremstyle{definition}
\newtheorem{definition}{Definition}[section]
\newtheorem{notation}[definition]{Notation}
\newtheorem{assumption}[definition]{Assumption}

\theoremstyle{remark}
\newtheorem{remark}[definition]{Remark}
\newtheorem{example}[definition]{Example}

\numberwithin{equation}{section}
\allowdisplaybreaks

\newcommand{\ee}{\hskip0.15ex}
\newcommand{\pv}{\ee;\ee}
\newcommand{\me}{\hskip-0.15ex}
\newcommand{\on}[1]{\raise-.5ex\hbox{\big|}_{#1}}
\newcommand\resid{\underline{\varepsilon}}

\renewcommand{\Re}{\operatorname{\mathrm{Re}}}
\renewcommand{\Im}{\operatorname{\mathrm{Im}}}
\newcommand{\artanh}{\operatorname{\mathrm{artanh}}}

\renewcommand{\div}{\operatorname{\rm div}}

\newcommand{\comp}{\mathrm{comp}}

\newcommand{\supp}{\mathrm{supp}}
\newcommand{\loc}{\mathsf{loc}}

\newcommand\C{{\mathbb C}}
\newcommand\R{{\mathbb R}}
\newcommand\N{{\mathbb N}}
\renewcommand\P{{\mathbb P}\ee}

\newcommand\Z{{\mathbb Z}}

\newcommand{\Ac}{\mathcal{A}}

\newcommand{\Cc}{\mathcal{C}}

\newcommand{\Ec}{\mathcal{E}}

\newcommand{\Hc}{\mathcal{H}}

\newcommand{\Lc}{\mathcal{L}}
\newcommand{\Nc}{\mathcal{N}}
\newcommand{\Oc}{\mathcal{O}}

\newcommand{\Rc}{\mathcal{R}}
\newcommand{\Sc}{\mathcal{S}}

\newcommand{\Zc}{\mathcal{Z}}

\newcommand{\as}{\mathsf{a}}
\newcommand{\As}{\mathsf{A}}
\newcommand{\Hs}{\mathsf{H}}
\newcommand{\Js}{\mathsf{J}}

\newcommand{\bA}{\boldsymbol{\mathsf{A}}}
\newcommand{\bR}{\boldsymbol{\mathsf{R}}}

\newcommand{\Hr}{\mathrm{H}}
\newcommand{\Lr}{\mathrm{L}\me}

\newcommand{\Fg}{{\mathfrak F}}
\newcommand{\Hg}{{\mathfrak H}}
\newcommand{\Kg}{{\mathfrak K}}

\newcommand{\Ug}{{\mathfrak U}}
\newcommand{\Xg}{{\mathfrak X}}

\newcommand{\wt}{\widetilde}
\newcommand{\wh}{\widehat}
\newcommand{\dps}{\displaystyle}
\newcommand{\te}{\textstyle}

\newcommand{\texte}[1]{\hbox{\ #1\ }}
\newcommand{\e}{\ensuremath{\mathrm{e}}}
\newcommand{\ic}{\ensuremath{\mathrm{i}}\xspace} % le i complexe
\newcommand{\dd}{\mathrm{\, d}}

\newcommand{\lp}{\left(}
\newcommand{\rp}{\right)}
\newcommand{\lc}{\left\lbrack}
\newcommand{\rc}{\right\rbrack}

\newcommand{\lo}{\left\lVert}
\newcommand{\ro}{\right\rVert}
\newcommand{\lv}{\left\lvert}
\newcommand{\rv}{\right\rvert}
\DeclareMathOperator{\Oo}{\mathcal{O}}

\newcommand{\ve}{\varepsilon}
\newcommand{\vp}{\varphi}

\newcommand{\ku}{\underline{k}}
\newcommand{\uu}{\underline{u}}
\newcommand{\vu}{\underline{v}}
\newcommand{\wu}{\underline{w}}
\newcommand{\lu}{\underline{\lambda}}
\newcommand{\Lu}{\underline{\Lambda}}
\newcommand{\Vu}{\underline{V\!}\,}
\newcommand{\Lcu}{\underline{\Lc\!}\,}
\newcommand{\II}{\mathds{1}}

\newcommand{\Vs}{\textsc w}

\newcommand{\kb}{\breve\kappa}
\newcommand{\mb}{\breve\mu}
\newcommand{\kr}{\mathrm k}

\newcommand{\Kr}{\mathrm K}

\newcommand{\TM}{\scriptscriptstyle\sf TM}
\newcommand{\TE}{\scriptscriptstyle\sf TE}
\newcommand{\GH}{\scriptscriptstyle\sf GH}

\newcommand{\no}{n}
\newcommand{\nx}{\tilde{n}}
\newcommand{\pj}{\relax}
\DeclareMathOperator{\Vect}{span}

\newcommand{\wb}{\overline}

\DeclareMathOperator{\Card}{Card}

\begin{document}

\title[Asymptotics for whispering gallery modes in optical micro-disks]
{Asymptotics for 2D whispering gallery modes in optical micro-disks with radially varying index}

\author{St\'ephane Balac}
\address{Univ. Rennes, CNRS, IRMAR - UMR 6625, F-35000 Rennes, France}
\email{stephane.balac@univ-rennes1.fr}

\author{Monique Dauge}
\address{Univ. Rennes, CNRS, IRMAR - UMR 6625, F-35000 Rennes, France}
\email{monique.dauge@univ-rennes1.fr}

\author{Zo\"{i}s Moitier}
\address{Karlsruhe Institute of Technology, Institute for Analysis, Englerstraße 2, D-76131 Karlsruhe, Germany}
\email{zois.moitier@kit.edu}

\thanks{The authors acknowledge support of the Centre Henri Lebesgue ANR-11-LABX-0020-01.}



\begin{abstract}
    Whispering gallery modes [WGM] are resonant modes displaying special features: They concentrate along the boundary of the optical cavity at high polar frequencies and they are associated with complex scattering resonances very close to the real axis.
    As a classical simplification of the full Maxwell system, we consider two-dimensional Helmholtz equations governing transverse electric [TE] or magnetic [TM] modes. Even in this 2D framework, very few results provide asymptotic expansion of WGM resonances at high polar frequency $m\to\infty$ for cavities with radially varying optical index.
    
    In this work, using a direct Schr\"odinger analogy we highlight three typical behaviors in such optical micro-disks, depending on the sign of an \emph{effective curvature} that takes into account the radius of the disk and the values of the optical index and its derivative.  Accordingly, this corresponds to abruptly varying effective potentials (step linear or step harmonic) or more classical harmonic potentials, leading  to three distinct asymptotic expansions for ground state energies.
     Using multiscale expansions, we design a unified procedure to construct families of quasi-resonances and associate quasi-modes that have the WGM structure and satisfy eigenequations modulo a super-algebraically small residual $\Oc(m^{-\infty})$.
     We show using the black box scattering approach that quasi-resonances are $\Oc(m^{-\infty})$ close to true resonances.
\end{abstract}
\maketitle

\noindent\rule{\linewidth}{0.5pt}
\par
{\small
    \setcounter{tocdepth}{1}
    \thispagestyle{empty}
    \tableofcontents
}
\vskip-3em
\noindent\rule{\linewidth}{0.5pt}

%============================
\section*{Introduction}
%============================

%= = = = = = = = = = = = = = = = = = = = = = = = = = = = = 
\subsection*{Motivation}
%= = = = = = = = = = = = = = = = = = = = = = = = = = = = = 
Our work is motivated by the study of light-wave propagation in optical micro-resonators and has its origin in a long-run collaboration with a physics laboratory specialized in optics \cite{BalEtAl20}.
These optical devices, with micrometric size, came to be important components in the photonic toolbox. They are basically composed of a dielectric cavity coupled to waveguides or fibers for light input and output \cite{HeeGro08}.

The knowledge of the {\em scattering resonances} of the dielectric cavity alone, considered as an open system, is a corner stone to find adequate conditions for the confinement of light-waves in the cavity when the full resonator is operated. This allows to access a wide range of optical phenomena. In order to dissipate possible misunderstanding due to various uses of the same words in different scientific groups, let us clarify the following terms and notations:
\begin{enumerate}
\item A scattering resonance $k$ (or simply resonance for short) of the cavity is an exact eigen-frequency of the open system formed by the cavity alone with outgoing radiation condition. With the $e^{-ikt}$ convention for time harmonic fields, this implies that $k$ has a negative imaginary part: $\Im k<0$, {\em cf.} \cite{PopVod99a}. Corresponding eigenvectors $u$ are called modes.
\item A quasi-resonance $\ku$ in association with a quasi-mode $\uu$, is an approximate eigen-pair of the same system, with a small residue. The term {\em quasi-mode} with this meaning is widely used in literature, {\em cf.} \cite{Ste99,TanZwo98} in our context. The pair $(\ku,\uu)$ will be called a quasi-pair.
\end{enumerate}
The study of scattering resonances is the subject of this paper. More specifically, we focus on ``\emph{Whispering Gallery Modes}'' for which modes are essentially localized inside the cavity and concentrated in a (boundary) layer. Such modes $u$ are associated with resonances $k$ close to the real axis ($|\Im k|$ is small).

\medskip

This work is motivated by the following observation found in the literature in optics:
Whispering Gallery Mode (WGM) resonators are in most cases formed of dielectric materials with constant optical index $n$ but this constitutes a potential limit in their performance and in the range of their applications. Resonators with spatially varying optical index, that fall under the category of ``graded index''
structures \cite{GomPer02}, offer new opportunities to improve and enlarge the field of applications of these devices and start to be investigated in optics.

Among properties that can be sensitive to a modification of the optical index $n$, let us mention
\begin{itemize}
\item[a)] The Q factor. This factor is the radiation quality factor of a mode, defined as the quotient $\Re k\,|\Im k|^{-1}$ of the real and imaginary parts of the corresponding resonance $k$. This factor accounts for radiative losses.
Larger Q factors correspond to longer life times.
\item[b)] The localization of modes. One is looking for modes that are concentrated inside the cavity, close to its boundary, with sufficient radiation outside to transmit light in another part of the device. 
\item[c)] The lattice structure of resonances (repartition of modes). In the WGM regime, modes can be classified by integer indices (e.g. the polar mode index $m$ and the radial mode index $j$ in two dimensions, characterizing the resonance $k_{m,j}$). A regular lattice structure will be obtained if differences between adjacent mode resonances $k_{m+1,j}-k_{m,j}$ and $k_{m,j+1}-k_{m,j}$ are almost constant in some regions. 
Such a lattice structure of resonances is sought, for instance, in the design of frequency comb generators. 
\end{itemize}

\medskip
The examples of graded index structures that we found in the literature in optics all relate to \emph{radially symmetric cavities}: In 3 dimensions of space, cylinders and spheres are concerned, whereas in 2 dimensions one considers disks,  for which the 2-dimensional model can be obtained as an approximation of the 3-dimensional one by using an approach referred as the effective index method \cite{SmoNos05}.

    A first theoretical example is provided by \cite{Ilc03} in which the radial profile of the optical index depends on a  parameter $\delta$
\begin{equation}
\label{eq:Ilc}
   n(r) = n_0 \sqrt{1+\delta R\Big(1-\frac{r}{R}\Big)}\quad\mbox{with}\quad n_0>1 \ \ \mbox{and}\ \ \delta \ge0.
\end{equation}
Here $R$ denote the cavity radius and it is understood that $n(r)$ is $1$ outside the cavity.
As shown in \cite{Ilc03}, depending on the values of the parameter $\delta$, the dispersion of modes changes, becoming more regular when $\delta$ tends to $\frac{2}{R}$ from below.

Among other examples of graded index structures 
we can quote a modified form of the ``Maxwell's fish eye'', that can be implemented using dielectric material, where the optical index varies with the radial position $r$ in the micro-disk resonator as \cite{DadKur14,NajVah16}
\begin{equation*}
    n(r) = \alpha \lp 1+\frac{r^2}{R^2}\rp^{-1}
\end{equation*}
where $\alpha>2$ and $R$ is the disk radius.
Another relevant work is \cite{StrCon14} where is considered a micro-cavity made of a quadratic-index glass doped with dye molecules and the refractive index is written as
$
    n(r) = \alpha -  \frac{1}{2} \beta r^2 ,
$
where $\alpha, \beta >0$.
In \cite{ZhuZho12}, an analysis of hollow cylindrical whispering
gallery mode resonator is carried out where the refractive index of the cladding varies according to $n(r) = \beta/r$, $\beta>R$.

%= = = = = = = = = = = = = = = = = = = = = = = = = = = = = 
\subsection*{Investigation framework}
%= = = = = = = = = = = = = = = = = = = = = = = = = = = = = 
Motivated by the above mentioned examples,
we found interesting to investigate in this paper the case of micro-disk cavities with radially varying optical index. The case of micro-spheres can be partly deduced from the investigation of micro-disks, see \cite[Appendix C]{BalEtAl20}.

In this configuration of micro-disks, Maxwell equations reduce to two scalar Helmholtz equations, corresponding to  \emph{Transverse Magnetic} (TM) or \emph{Transverse Electric} (TE) modes, for which eigen-equations are
\begin{equation}
\label{eq:TETM}
   \Delta u + k^2n^2u = 0 \quad\mbox{[TM]}\quad
   \qquad\mbox{or}\qquad
   \div \left( n^{-2}\,\nabla u\right) +  k^2 u = 0 \quad\mbox{[TE]}
\end{equation}
These Helmholtz equations are posed in the whole plane and completed by a suitable outgoing radiation condition at infinity. The resonances $k$ have a negative imaginary part.

In this preliminary discussion, let us concentrate on the TM formulation. By separation of variables using polar coordinates, the Helmholtz equation becomes a family of scalar equations indexed by the (integer) polar mode index $m$
\begin{equation}
\label{eq:TMm}
   u'' + \frac{1}{r} u' - \frac{m^2}{r^2} u + k^2n^2 u = 0.
\end{equation} 
Here appears the optical potential, sum of the centrifugal part $\frac{m^2}{r^2}$ and an attractive part driven by $n^2$. Setting 
\begin{equation}
\label{eq:}
   h = \frac{1}{m}\,,\quad  W(r) = \frac{1}{r^2n^2(r)}\,, \quad\mbox{and}\quad E = \frac{k^2}{m^2}
\end{equation}
the analogy with the (time-independent) Schr\"{o}dinger equation becomes clear, yielding the equation
\begin{equation}
\label{eq:Schr}
   - h^2 \, \frac{1}{n^2}\lp u'' + \frac{1}{r} u'\rp + Wu = E u .
\end{equation}
The behavior of the \emph{effective potential} $W$ near its minimum is the key for the structure of WGM resonances when $m$ tends to infinity. As an illustrative example, let us observe the optical index $n$ and the potential $W$ in example \eqref{eq:Ilc} for the values $0$, $\frac2R$, and $\frac4R$ of parameter $\delta$, see Fig. \ref{fig:Ilc}.

\begin{figure}[!htbp]
    \centering
    \begin{subfigure}[h]{0.33\linewidth}
        \hglue-1.4em
        \includegraphics[width=1.12\linewidth]{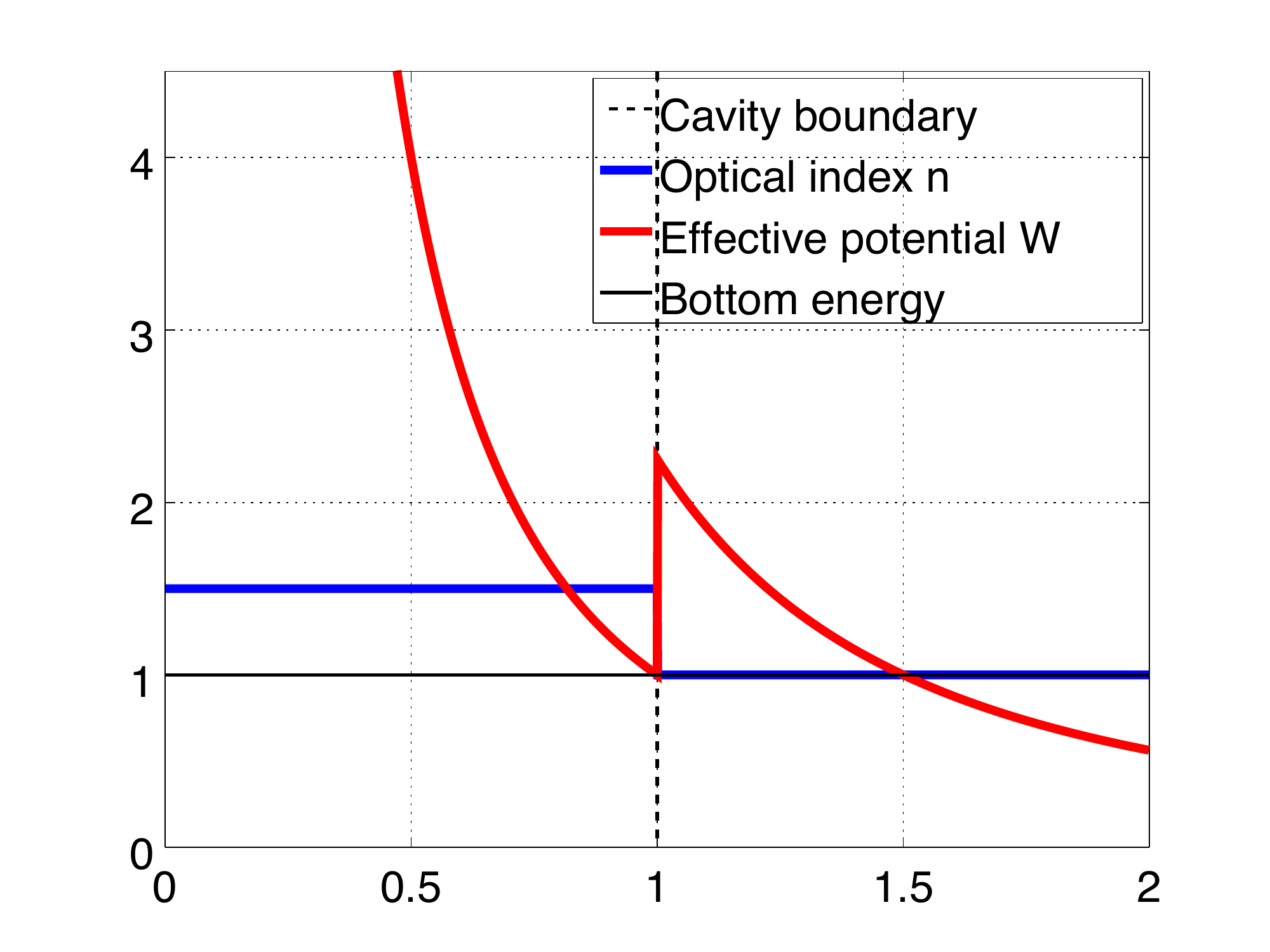}
        \caption{\ $\delta=0$}
    \end{subfigure}%
    \begin{subfigure}[h]{0.33\linewidth}
        \hglue-0.7em
        \includegraphics[width=1.12\linewidth]{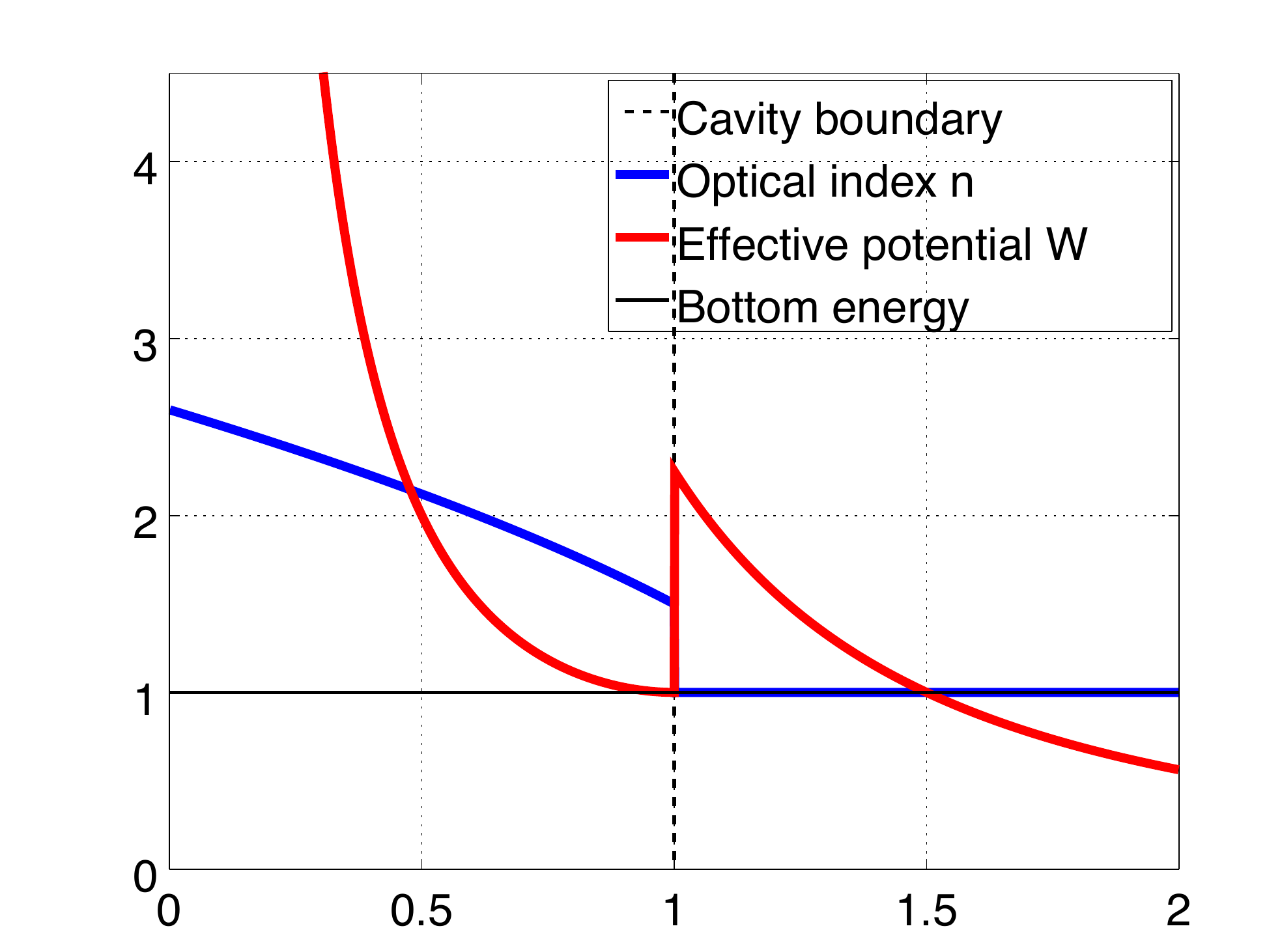}
        \caption{\ $\delta=2$}
    \end{subfigure}%
    \begin{subfigure}[h]{0.33\linewidth}
        \centering
        \includegraphics[width=1.12\linewidth]{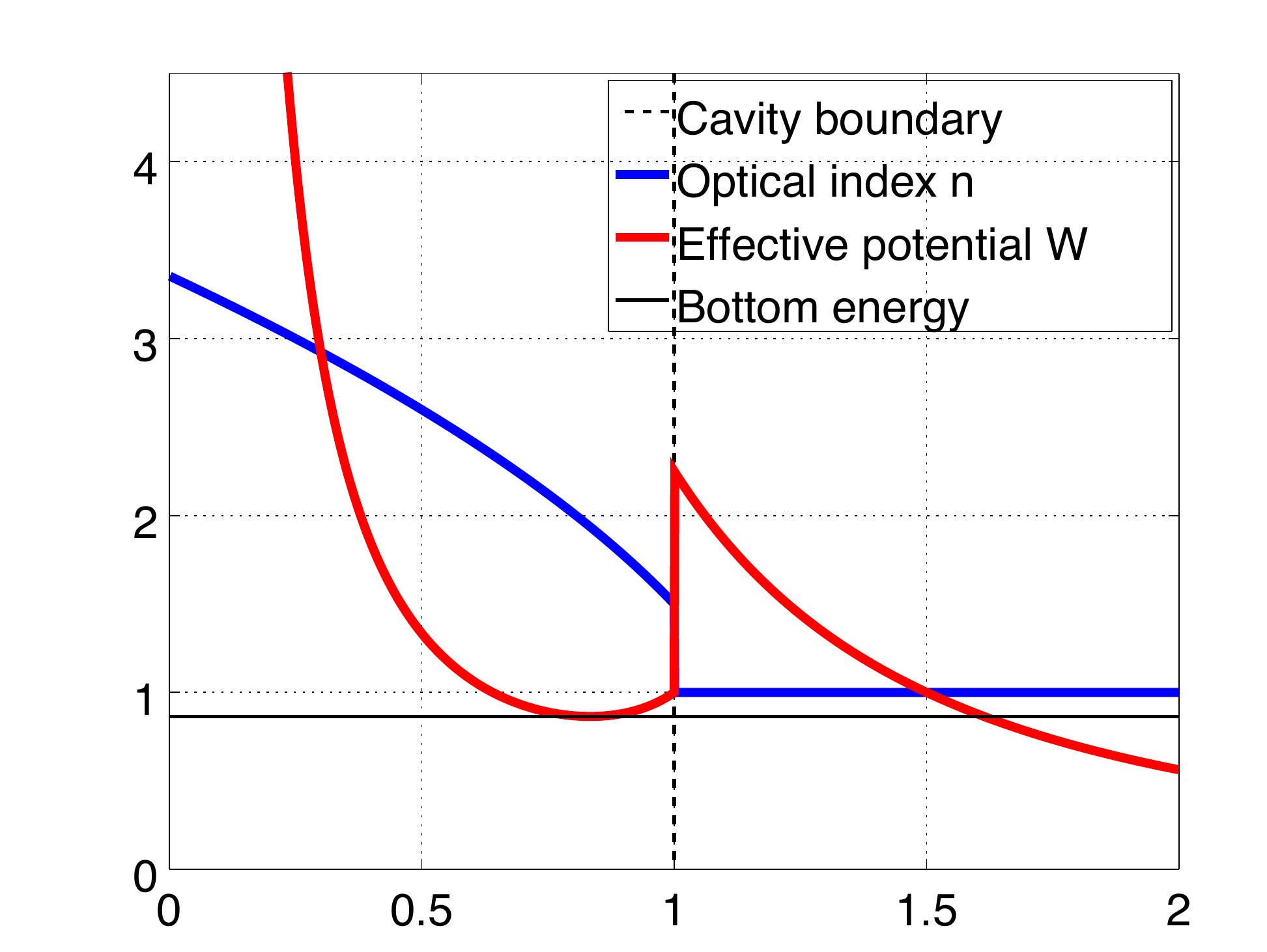}
        \caption{\ $\delta=4$}
    \end{subfigure}

    \caption{Optical index $n$ and potential $W$ in example \eqref{eq:Ilc} with $R=1$, $n_0=1.5$ ($W$ in units $R^{-2}n_0^{-2}$).}
    \label{fig:Ilc}
\end{figure}

Cases {\textsc{(a)}}, \textsc{(b)}, and \textsc{(c)}, display well bottoms with distinct shapes, driven by the values of the potential $W$ and its derivatives. The formula
\[
   W'(r) = -\frac{2}{r^2n^2(r)}\, \lp \frac{1}{r} + \frac{n'(r)}{n(r)} \rp,\quad 0<r\le R,
\]
leads to the introduction of the \emph{effective curvature} $\kappa_{\sf eff}$
\begin{equation}
\label{eq:keff}
   \kappa_{\sf eff} = \frac{1}{R} + \frac{n'(R)}{n(R)}
\end{equation}
the sign of which discriminates these three situations. Note that in example \eqref{eq:Ilc} we simply have
\begin{equation}
\label{eq:Ilck}
   \kappa_{\sf eff} = \frac{1}{R} - \frac{\delta}{2}\,.
\end{equation}
This leads to the classification of the three cases that we investigate in detail in this paper:
\begin{itemize}
\item[\textsc{(a)}] $\kappa_{\sf eff}>0$: The potential $W$ has a well bottom with half-triangular (aka step-linear) shape at $r=R$.
\item[\textsc{(b)}] $\kappa_{\sf eff}=0$: The potential $W$ has a well bottom with half-quadratic (aka step-harmonic) shape at $r=R$.
\item[\textsc{(c)}] $\kappa_{\sf eff}<0$: There exist $R_0<R$ at which the potential $W$ has a well bottom with quadratic shape at $r=R$.
\end{itemize}
Going from \textsc{(a)} to \textsc{(c)} we observe that WGM spread more and more inside the cavity, and that the corresponding resonances have smaller and smaller imaginary parts, see Fig. \ref{fig:IlcWGM1}.
\begin{figure}[!htbp]
    \centering
    \begin{subfigure}[h]{0.33\linewidth}
        \hglue-1.4em
        \includegraphics[width=1.12\linewidth]{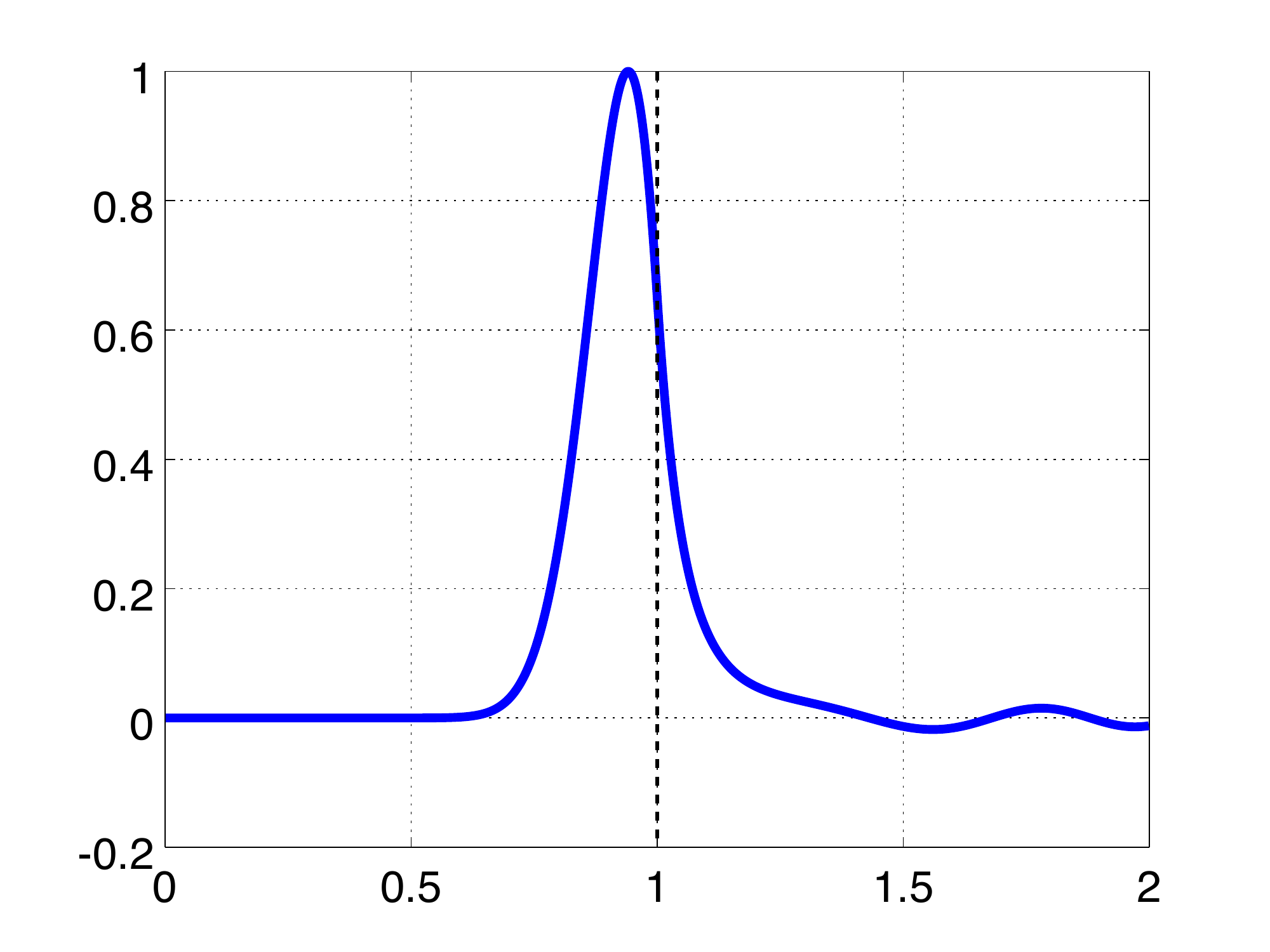}\\
        \centerline{\small$k = 23.04 - 5.2\,10^{-4}\ee\ic$}
        \caption{\ $\delta=0$}
    \end{subfigure}%
    \begin{subfigure}[h]{0.33\linewidth}
        \hglue-0.7em
        \includegraphics[width=1.12\linewidth]{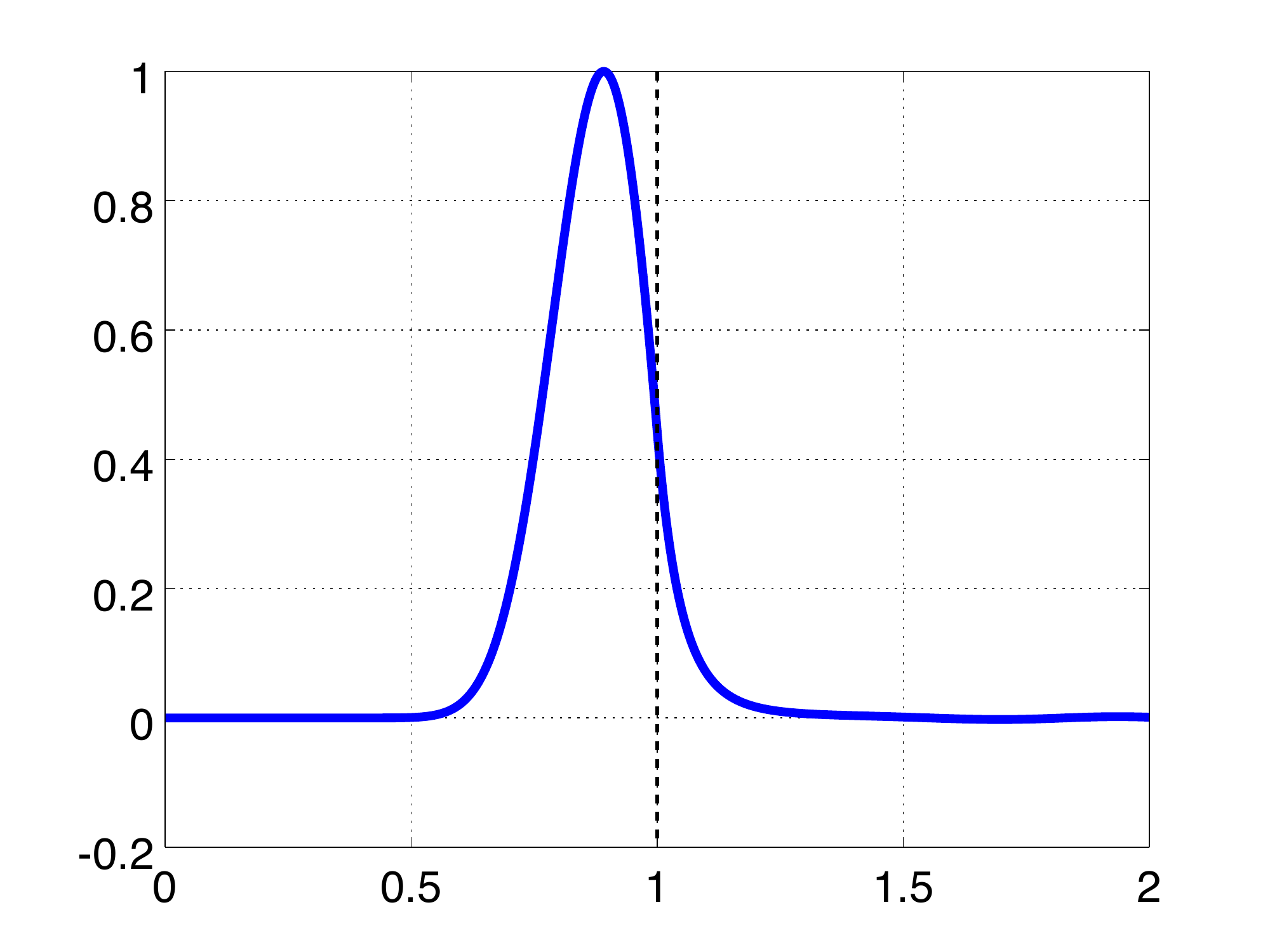}\\
        \centerline{\small$k = 21.20 - 6.4\,10^{-6}\,\ic$}
        \caption{\ $\delta=2$}
    \end{subfigure}%
    \begin{subfigure}[h]{0.33\linewidth}
        \centering
        \includegraphics[width=1.12\linewidth]{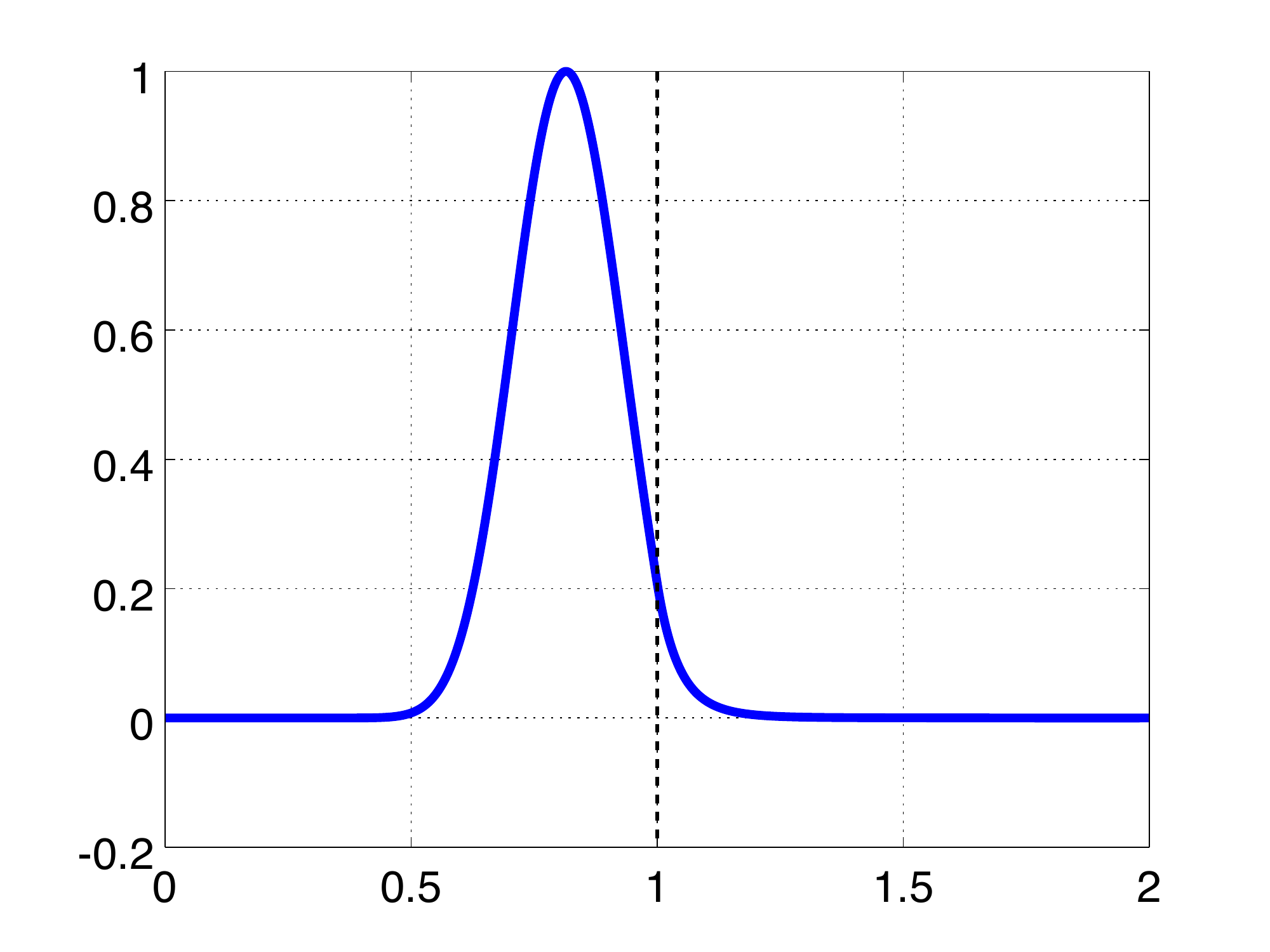}\\
        \centerline{\small$k = 19.16 - 1.2\,10^{-8}\,\ic$}
        \caption{\ $\delta=4$}
    \end{subfigure}

    \caption{Real part of radial profiles of first WGM in example \eqref{eq:Ilc} with $R=1$, $n_0=1.5$, and $m=30$. Computed by solving \eqref{eq:TETM} by finite differences with Perfectly Matched Layers.}
    \label{fig:IlcWGM1}
\end{figure}

In this paper we provide asymptotic expansions for WGM resonances as the polar mode index $m$ tends to infinity. This means two things:
\begin{enumerate}
\item Construct sequences of quasi-resonances and quasi-modes $m\mapsto (\ku(m),\uu(m))$ that satisfy approximately equation \eqref{eq:TMm} or its TE counterpart with super-algebraically small residues in $\Oc(m^{-\infty})$. 
\item Apply the theory of black box scattering \cite{TanZwo98,Ste99} to deduce the presence of true resonances $k(m)$ super-algebraically close to quasi-resonances $\ku(m)$.
\end{enumerate}
This will be done in a general framework of variable optical indices entering in either case \textsc{(a)}, \textsc{(b)} or \textsc{(c)}, see Theorems \ref{th:A}--\ref{th:C} later on. The construction in case \textsc{(a)} and case \textsc{(c)} leads to a computable algorithm, i.e.\ involving at each level matrix products and matrix inversions in finite dimension.

%= = = = = = = = = = = = = = = = = = = = = = = = = = = = = 
\subsection*{Connection with literature}
%= = = = = = = = = = = = = = = = = = = = = = = = = = = = = 
Our results in case \textsc{(a)} are reminiscent of those by \textsc{Babi\v{c}} and \textsc{Buldyrev} \cite[Sec.\ 7.4]{BabBul61}, where the Helmholtz equation 
$
   \Delta u + \frac{k^2}{c^2}u = 0
$
is considered in a closed cavity with Dirichlet boundary conditions. 
Asymptotic expansions of eigen-frequencies corresponding to whispering gallery mode are determined there. The
same condition on the effective curvature appears in \cite[eq.~(7.2.1)]{BabBul61}. 
As a matter of fact, both problems are close: Asymptotics of \cite{BabBul61} and ours start with the same two terms and differ by the third one. Nevertheless there are also striking differences: In contrast with \cite{BabBul61}, we have to couple the interior and the exterior of the cavity and need distinct scales for that. Moreover, the scattering resonances appear as eigen-frequencies of a \emph{non selfadjoint problem}, which prevents from using the spectral theorem as in \cite[Thm.~7.1]{BabBul61}.

Let us also mention that our results in case \textsc{(a)} generalize to \emph{variable optical indices} the expansion obtained by \textsc{Lam} et al. \cite{LamLeu92} for constant indices by a quite specific method (expansion of Bessel functions entering in the so called modal equation \eqref{eq:modal} see later on).

Our case \textsc{(c)} appears to essentially enter the framework of the ``puits ponctuel dans une isle'' of  \textsc{Helffer} and \textsc{Sj\"ostrand} \cite[Chap. 10]{HelSjo86}.
Namely, asymptotic expansions of resonances and modes are given there in Theorem 10.7 and  Theorem 10.8, respectively.  These very sharp results are proved for analytic potentials and make use of semi-classical analytic pseudo-differential calculus. Stricto sensu, our effective potentials do not satisfy the assumptions of \cite[Chap. 10]{HelSjo86}. Our (one-dimensional) multi-scale expansions are more explicit and use simpler tools, which is our motivation to include them in our analysis.

%= = = = = = = = = = = = = = = = = = = = = = = = = = = = = 
\subsection*{Real vs imaginary parts}    
%= = = = = = = = = = = = = = = = = = = = = = = = = = = = = 
In the three cases \textsc{(a)}, \textsc{(b)}, \textsc{(c)}, the quasi-resonances $\ku(m)$ that we construct  are real. The associated quasi-modes $\uu(m)$ are localized in the radial variable (close to the interface $r=R$ in the first two cases and close to the internal circle $r=R_0$ in the third one).
    The couples $(\ku(m), \uu(m))$ are in fact quasi-pairs for a transmission problem in a larger bounded domain containing the disk of radius $R$. This latter problem can be viewed as self-adjoint.
    This explains why the quasi-resonances $\ku(m)$ are real. This also explains why our analysis does not access the imaginary part of resonances apart from the rough information that $|\Im k(m)|=\Oc(m^{-\infty})$ for resonances $k(m)$ that are close to quasi-resonances $\ku(m)$.
    
The issue of imaginary parts calls for a couple of comments.
Information about the imaginary parts of resonances needs a WKB approximation based on solution of eikonal and transport equations. The imaginary part comes from the tunneling effect in the classically forbidden region where $W>E$. In the semi-classical analytic framework of \cite[Chap. 10]{HelSjo86}, an expression is proven (Theorem 10.12), displaying the exponentially decreasing behavior of $|\Im k|$ as $h$ tends to $0$ (i.e.\ $m$ tends to $\infty$ for us). The WKB approximation can be conducted in a more or less formal way to obtain the first terms of such an expansion, see \cite[Chap.~15]{Hall2013} for confining potentials and \cite{FernandezGarcia2011} for radial barriers. The investigation of discontinuous potentials of the same kind as ours in \cite{Amthong2014} shows that the matter is far from being obvious. Calculations done in the same spirit as \cite{Hall2013,FernandezGarcia2011,Amthong2014} in cases \textsc{(a)} and \textsc{(b)} lead to the plausible estimate $|\Im k(m)| = \Oo ( \e^{- 2 S_0 m} )$ at the ground state energy $E=\frac{1}{R^2\, n(R)^2}$ with $S_0$ the integral of the {\em kinetic parameter} $\sqrt{W(r) - E}$ between the point of discontinuity $r=R$ and the turning point $r=Rn(R)$.
We compute:
\begin{equation}
\label{eq:WKB2}
   S_0 = \int_{R}^{R\, n(R)} \sqrt{W(r) - E} \dd{r} = \artanh(T) - T \quad\mbox{with}\quad T = \sqrt{1-\tfrac{1}{n(R)^2}} .
\end{equation}
Even if correct, these formulas have to be taken with care. One can notice that the term $\Oo$ hides a power of $m$ and the dependency of resonances on the radius $R$.
When the potential is analytic, the $\Oo$ is encoded precisely in an analytic symbol. But in our case, we could simply learn that we have an exponential decay of the imaginary part. 

We have lead numerical computations combining perfectly matched layers (PML), finite differences and extrapolation, that provided an evaluation of the imaginary part, up to $10^{-10}$ relative precision at least. The exponential decay of $|\Im k(m)|$ as $m\to\infty$ is clearly identifiable, but the decay rate does not converge rapidly in $m$.

By contrast, the expansions of the real parts involves integer powers of $m^{-\frac13}$ in case \textsc{(a)} and of $m^{-\frac12}$ in cases \textsc{(b)}-\textsc{(c)}. Their coefficients depend explicitly on TE/TM type and on another index $j$ that is, roughly the number of oscillations of modes inside the cavity. We provide explicit formulas for the first 5 terms of expansions in case \textsc{(a)} (and 3 in cases \textsc{(b)}-\textsc{(c)}). This displays the influence of the first derivatives of the optical index on these various terms. According to what was first guessed in \cite{Ilc03}, one can also observe the regular asymptotic repartition of resonances in cases \textsc{(b)}-\textsc{(c)}.

%= = = = = = = = = = = = = = = = = = = = = = = = = = = = = 
\subsection*{Intertwining of asymptotics, simulations, and physics}
%= = = = = = = = = = = = = = = = = = = = = = = = = = = = = 
The values of the polar frequency $m$ that are interesting in practice for micro-resonators correspond to micrometric wave length on the boundary of the cavity. According to the size of the cavity (with $R$ ranging from {\em ca}.\ $10\mu$m to a few hundreds $\mu$m) this corresponds to values of $m$ that go from {\em ca}.\  $50$ to $1000$. This means that we are either in a middle or in a high frequency regime. In any case, asymptotic formulas give insight into the repartition of WGM resonances and the localization of modes. So, prior to any experimentation, one can forecast what can be expected if a certain variable index is used.

In the middle/high frequency regime, asymptotic formulas provide a priori repartition of WGM resonances with at most $1\%$ relative error. This information can be combined with great advantage to numerical simulations including PML's. Since the computed spectrum spreads in the complex plane and contains non physical parts due to the PML angle and the cut-off, an a priori localization is of utmost importance to dig out the correct eigen-frequencies of discretized matrices. In this range of frequencies, as already mentioned, it is possible to obtain an evaluation of the radiation Q-factor, up to $10^{10}$ at least.

In very high frequency regime, numerical computations loose precision due to high stiffness and, by contrast, asymptotics become very precise with at most {\em ca}.\ $10^{-6}$ relative error. The imaginary part of resonance tends exponentially rapidly to $0$ with larger $m$, becoming undetectable in double precision arithmetics. This means that at such frequencies, $\Im k$ is not calculable. Nevertheless, it is important to note that, in this situation, the radiation $Q$ factor becomes irrevelant from a physical point of view because the radiation losses are much lower than other loss sources (such as edge roughness, material absorption or material defects) that contribute the more, {\em cf.} \cite{borselli2005}.

To conclude, we can say that asymptotics and numerical computations help each other to gain a very precise insight into WGM. The building of a finite difference code with perfectly matched layers (PML) is the object of a work in progress by the authors. Let us finally mention that the third author \cite{MoitierPhD} handles two-dimensional non-axisymmetric cavities, theoretically by a semi-WKB method, and numerically by finite element method via the library \href{https://uma.ensta-paris.fr/soft/XLiFE++/}{XLiFE++}.

%= = = = = = = = = = = = = = = = = = = = = = = = = = = = = 
\subsection*{Organization of the paper}
%= = = = = = = = = = = = = = = = = = = = = = = = = = = = = 
Section~\ref{sec:main} presents the mathematical setting of the problem and the main results. We also revisit there the case of constant optical index to illustrate what are the WGM resonances in the general landscape of complex resonances.
In section~\ref{sec:quasi}, we make precise the notion of quasi-pairs, quasi-resonances, and quasi-modes.
In section~\ref{sec:3types}, we come back to the Schr\"odinger analogy and the classification \textsc{(a)}\textsc{(b)}\textsc{(c)} of our cases of study.
In sections~\ref{sec:(a)}--\ref{sec:(c)}, we construct quasi-pairs associated with localized resonances in these three cases.
Finally, in section~\ref{sec:proxy}, we show that the quasi-resonances just constructed are asymptotically close to true resonances of the cavity, hence ending the proof of  Theorems~\ref{th:A}--\ref{th:C}. 

%= = = = = = = = = = = = = = = = = = = = = = = = = = = = = 
\subsection*{Notation}
%= = = = = = = = = = = = = = = = = = = = = = = = = = = = = 
The set of non-negative integers is denoted by $\N$ and the set of positive integer by $\N^*$.
We denote by $\Lr^2(\Omega)$ the space of square-integrable functions on the open set $\Omega$, and by $\Hr^\ell(\Omega)$ the Sobolev space of functions in $\Lr^2(\Omega)$ such that their derivatives up to order $\ell$ belong to $\Lr^2(\Omega)$. Finally, $\Sc(I)$ denotes the space of Schwartz functions on the unbounded interval $I$.

\section{Problem setting and main results}
\label{sec:main}

\subsection{Helmholtz equations for optical micro-disks}
%= = = = = = = = = = = = = = = = = = = = = = = = = = = = = 

In the 2-dimensional model that we investigate from now on, the optical cavity is a disk of radius $R$ centered at the origin, that we denote by $\Omega$. This cavity is associated with an optical index $n>1$ which is a regular radial function of the position on the closed interval $[0,R]$. Outside $\overline\Omega$, i.e. for $r>R$, the index $n$ is equal to $1$. There are two relevant scalar Helmholtz equations associated with such a configuration, corresponding to \emph{Transverse Electric} (TE) modes or \emph{Transverse Magnetic} (TM) modes. In order to have a common formulation, we introduce a parameter $p\in\{\pm1\}$ defined as
\begin{equation}
\label{eq:p}
   p=1 \quad\mbox{in TM case}\quad\mbox{and}\quad p=-1  \quad\mbox{in TE case}.
\end{equation}
With this convention, the resonance problems are: \\[1ex]
{\sl Find $k\in\C$ and nonzero $u$ with $u\on{\Omega}\in \Hr^2(\Omega)$ and
     $u\on{\R^2\setminus\Omega} \in \Hr^2_\loc(\R^2\setminus\Omega)$ such that}
\begin{subequations}
    \label{eq:Pp}
\begin{empheq}[left={\hspace*{-3ex}\empheqlbrace\quad}]{align}
 &- \div \left( n^{p-1}\,\nabla u\right) -  k^2 n^{p+1} \,  u = 0,
                                   &&\texte{in }\Omega \texte{ and }\R^2\setminus\overline\Omega
   \label{eq:Ppa} \\
 & [u] = 0 \quad\mbox{and}\quad [n^{p-1} \,\partial_\nu u] = 0  &&\texte{across } \partial\Omega 
   \label{eq:Ppb}\\
 & \mbox{Outgoing radiation condition} &&\texte{as }r\to+\infty  
   \label{eq:Ppd}
\end{empheq}
\end{subequations}
Here $[f]$ denotes the jump of $f$ across $\partial\Omega$.
The outgoing radiation condition \eqref{eq:Ppd} is distinct from the usual
Sommerfeld condition that is valid for real k's. It can be referred as the Siegert condition in the literature. 
In our framework, \eqref{eq:Ppd} imposes that 
outside $\overline\Omega$, $u$ has an expansion
in terms of Hankel functions of the first kind $\Hs^{(1)}_m$ 
in polar coordinates (compare with \cite{SteUng17} for instance):
\begin{align}\label{eq:cr}
    u(x,y) = \sum_{m\in\Z} C_m \,\e^{\ic m \theta} \ \Hs_m^{(1)}(k r) 
    && \forall \theta\in[0,2\pi] ,\ r>R_\Omega.
\end{align}
Owing to Rellich theorem (see \cite{McL00} for instance) the radiation condition \eqref{eq:Ppd} implies that the imaginary part of $k$ is negative and the modes, exponentially increasing at infinity.
Such wave-numbers $k$ are the poles of the extension of the resolvent of underlying Helmholtz operators when coming from the upper half complex plane.

We note that, in the case when $n$ is constant inside $\Omega$, problem \eqref{eq:Pp} appears also as a modeling of scattering by a transparent obstacle (see \textsc{Moiola} and \textsc{Spence} \cite[Remark 2.1]{MoiSpe19} for a discussion of the models in acoustics and electromagnetics).
In contrast with impenetrable obstacles (see \textsc{Sj\"ostrand} and \textsc{Zworski} \cite{SjoZwo99}), transparent obstacles or dielectric cavities may have resonances super-algebraically close to the real axis due to almost total internal reflection on a convex boundary (see \textsc{Popov} and \textsc{Vodev} \cite{PopVod99a}, and also \textsc{Galkowski} \cite{Gal19}).
These resonances correspond to the whispering gallery modes we are interested in.

%____________

In polar coordinates $(r,\theta)$ centered at the disk center, $n=n(r)$.
Taking advantage of the invariance by rotation of equations \eqref{eq:Pp}, it is easily proved by separation of variables that any solution $u$ associated with a $p\in\{\pm1\}$ and a resonance $k\in\C$ can be expanded as a Fourier sum
\[
    u (x,y) = \sum_{m\in\Z} w_m(r)\,\e^{\ic m \theta}
\]
and that each term $u_m(x,y) \coloneqq w_m(r)\,\e^{\ic m \theta}$ is a solution of problem \eqref{eq:Pp} associated with the same $p$ and the same $k$. Hence it is sufficient to solve, for any $m\in\Z$, problem \eqref{eq:Pp} with $u$ of the form $w_m(r)\,\e^{\ic m \theta}$. Here $m\in\Z$ is referred as the \emph{polar mode index}.
The radial problem satisfied by $w:r\mapsto w(r)$ when $u=w(r)\,\e^{\ic m \theta}$  is plugged into problem \eqref{eq:Pp} is the following parametric radial problem depending on the integer $m$: \\[1ex]
{\sl Find $k\in\C,\;\; w\on{(0,R)}\!\in \Hr^2((0,R),r\dd r),\;\;
            w\on{(R,\infty)} \!\in \Hr^2_\loc([R,\infty),r\dd r) \texte{ such that}$}
\begin{subequations}
    \label{eq:Pprad}
\begin{empheq}[left={\hspace*{-3ex}\empheqlbrace\quad}]{align}
    &-\frac{1}{r}\partial_r(n^{p-1}r\partial_r w) +
    n^{p-1} \lp \frac{m^2}{r^2} - k^2\, n^2 \rp w = 0 
     && \texte{in } (0,R) \texte{and} (R,+\infty) 
     \label{eq:Pprada}\\[0.5ex]
    &\lc w \rc = 0 \quad\mbox{and}\quad \lc n^{p-1}w' \rc= 0
     && \texte{for } r=R                          
      \label{eq:Ppradb}\\[0.5ex]
    &w(r) = \Hs_m^{(1)}(k r) 
     && \texte{for } r>R
      \label{eq:Ppradd}
\end{empheq}
\end{subequations}
Note that the radiation condition \eqref{eq:Ppradd} implies that $w$ is solution of \eqref{eq:Pprada} in $(R,\infty)$. When $m=0$ a Neumann condition at the origin should be added. When $m$ is nonzero, we have a natural zero Dirichlet condition at $0$ due to the blow up of the centrifugal potential.

%= = = = = = = = = = = = = = = = = = = = = = = = = = = = = 
\subsection{Circular cavity with constant optical index}
\label{sec:0939}
%= = = = = = = = = = = = = = = = = = = = = = = = = = = = = 

As a fundamental illustrative example, let us consider a circular cavity with \emph{constant} optical index $n\equiv n_0 >1$ in $\overline\Omega$. Though apparently simple, this case is indeed already very rich. The use of partly analytic formulas provides a lot of information on the resonance set and the associated modes. Let us sketch this now. For both TM and TE modes, solutions $w$ of \eqref{eq:Pprad} have the form
\begin{equation}
    \label{eq:cmt:2754a}
    w(r) = \left\{
    \begin{array}[h]{ll}
        \Js_{m}(n_0 k r)    & \texte{if} r\leq R \\[2mm]
        \dps  \frac{\Js_{m}(n_0 k R )}{\Hs^{(1)}_{m}( k R )}\,
        \Hs^{(1)}_{m}( k r) & \texte{if} r>R.    \\
    \end{array}\right.
\end{equation}
Here  $\mathsf{J}_{m}$ refers to the order $m$ Bessel function of the first kind
and $\mathsf{H}^{(1)}_{m}$  refers to the order $m$  Hankel function of the first kind.
In both TE and TM cases, the resonance $k$ is obtained as a solution to the following non-linear equation, termed \emph{modal equation} (recall that $p=1$ for TM modes and $p=-1$ for TE modes)
\begin{align}
    \label{eq:modal}
    n_0^p\, \Js^{\ee\prime}_{m}(n_0 R k)\, \Hs^{(1)}_{m}(R k) - \Js_{m}(n_0 R k)\, {\Hs^{(1)}_{m}}'(R k) = 0 .
\end{align}
For each value of the polar mode index $m$, the modal equation \eqref{eq:modal} has infinitely many solutions $k\in\C$.
We denote by $\Rc_p[n_0,R](m)$ this set.
Because $\Js_{-m}(\rho)=(-1)^m\Js_{m}(\rho)$ and the same for $\Hs^{(1)}_{m}$, the two integers $\pm m$ provide the same resonance values,
which reflects a degeneracy of resonances for disk-cavities. Thus, we can restrict the discussion to nonnegative integer values of $m$.

\begin{figure}[!htbp]
    \centering
    \begin{subfigure}[h]{0.49\linewidth}
        \centering
        \includegraphics{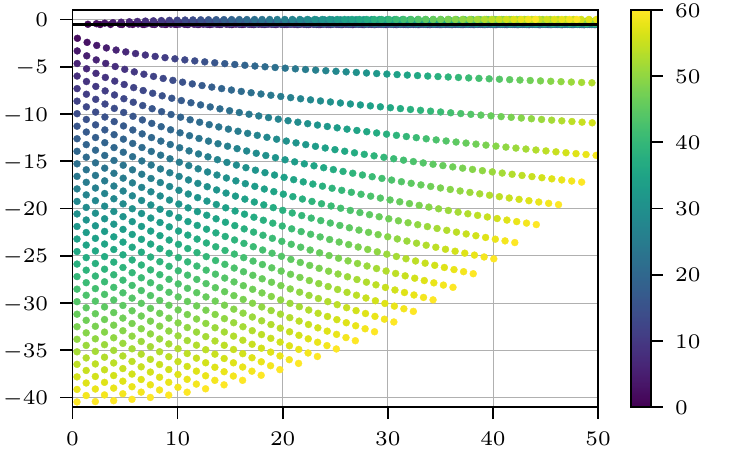}
        \caption{$p = +1$}
    \end{subfigure}
    \begin{subfigure}[h]{0.49\linewidth}
        \centering
        \includegraphics{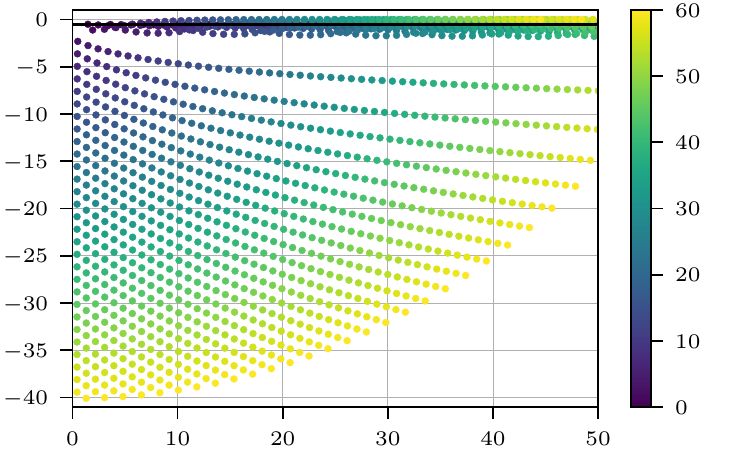}
        \caption{$p = -1$}
    \end{subfigure}

    \bigskip\medskip

    \begin{subfigure}[h]{0.49\linewidth}
        \centering
        \includegraphics{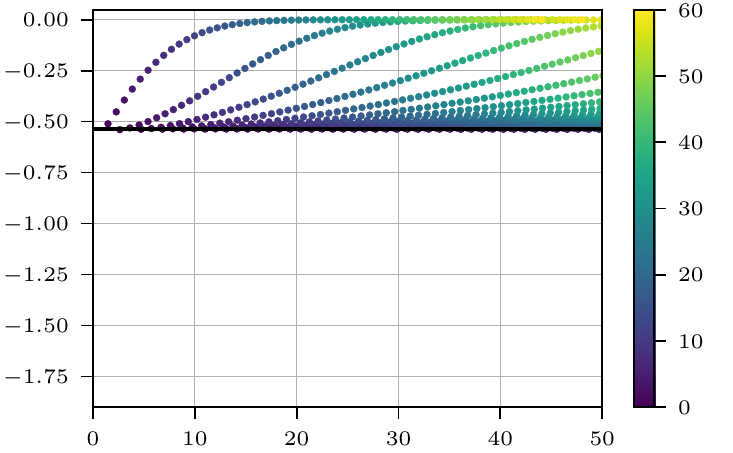}
        \caption{$p = +1$}
    \end{subfigure}
    \begin{subfigure}[h]{0.49\linewidth}
        \centering
        \includegraphics{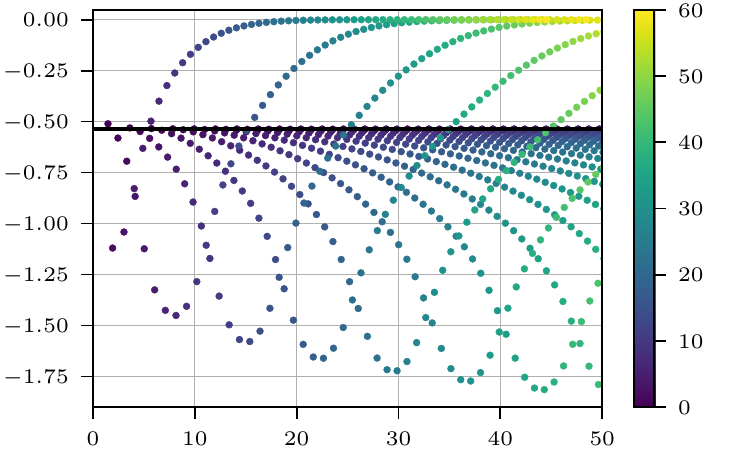}
        \caption{$p = -1$}
    \end{subfigure}
    \caption{Roots of \eqref{eq:modal} when $p\in\{\pm 1\}$, $n_0 = 1.5$, $R=1$, and $0\le m \le 60$.
        The first row gives a global view, whereas the second row provides a zoom on inner resonances.
            The color of dots displays the value of index $m$ according to the color bar.}
    \label{fig:Bes_all_m}
\end{figure}

In Figure \ref{fig:Bes_all_m}, we show the complex roots of equation \eqref{eq:modal} when $n_0=1.5$, $R=1$, and $0\le m \le 60$, the values of $m$ being distinguished by a color scale. This figure displays clearly the general features of the set of resonances. For each $m$ the set of resonances $\Rc_p[n_0,R](m)$ can be split into two parts, see \cite{ChoKim10},
\begin{itemize}
    \item an infinite part $\Rc_{p,\, \sf inner}[n_0,R](m)$ made of \emph{inner resonances} for which the modes are essentially supported inside the disk $\Omega$
    \item a finite part $\Rc_{p,\, \sf outer}[n_0,R](m)$ made of \emph{outer resonances} for which the modes are essentially supported outside the disk $\Omega$.
\end{itemize}
The sets of inner and outer resonances are given by
\begin{align*}
    \Rc_{p,\,\mathsf{inner}}[n_0,R] = \bigcup_{m\in\N} \Rc_{p,\,\mathsf{inner}}[n_0,R](m) &&
    \text{and}                                                                            &&
    \Rc_{p,\,\mathsf{outer}}[n_0,R] = \bigcup_{m\in\N} \Rc_{p,\,\mathsf{outer}}[n_0,R](m)
\end{align*}
respectively.
It appears that for TM modes ($p=1$) there exists a negative threshold $\tau$ such that the outer resonances satisfy $\Im k < \tau$, and the inner ones, $\Im k \ge \tau$. We can clearly see on  Fig\ \ref{fig:Bes_all_m} some organization in sub-families, not indexed by $m$ (i.e., $m$ varies along these families). Observation of the associated modes shows that these families depend on another parameter, $j$, which can be called a \emph{radial mode index}: This is the number of sign changes (or nodal points) of the real part of an associated mode. For inner resonances, the sign changes occur inside the disk.
For outer resonances, there is no such interpretation in term of sign changes but they can be linked to the resonances of the exterior Dirichlet problem: One can see in \cite[eq.\ (49)]{BogDub08E} that when $m \to +\infty$, outer resonances tend to the zeros of  Hankel function $k \mapsto \Hs^{(1)}_{m}(R k)$.

Inner resonances will be denoted by $k_{p\pv j}(m)$, with $p=1$ and $p=-1$ according to the TM and TE cases respectively, and with $m$ and $j$ the polar and radial mode indices, respectively. Then there exist two distinct asymptotics for $k_{p\pv j}(m)$ according to $j\to\infty$ or $m\to\infty$.

On the one hand, direct calculations yield, \cite[Section 3.3.1]{MoitierPhD},
when $p=\pm1$ and $j \to +\infty$
\begin{align*}
    k_{p\pv j}(m) & \sim \frac{j\pi}{R n_0} + \frac{(2m+2-p)\pi}{4R n_0} + \frac{\ic}{2Rn_0}\ln\lp\frac{n_0-1}{n_0+1}\rp.
\end{align*}
Thus, as $j\to\infty$, the imaginary part of $k_{p\pv j}(m)$ tends to the negative value $\frac{1}{2n_0}\ln(\frac{n_0-1}{n_0+1})$.
For the example displayed in Figure \ref{fig:Bes_all_m}, this value is $-0.53648$.
This same value can also be found in the physical literature, see \cite[eq.\ (13)]{BogDub08A} for example.

On the other hand, using asymptotic expansions of Bessel's functions involved in the modal equation \eqref{eq:modal},
one obtains that resonances $k_{p\pv j}(m)$ for a given radial index $j$ satisfy the following asymptotic expansion when $m \to +\infty$:
\begin{multline}\label{eq:1421}
    k_{p\pv j}(m) =
    {\frac{m}{R{n_0}}} \Bigg[ 1 + \frac{\as_j}{2} {\lp\frac{2}{m}\rp^{\frac{2}{3}}} -
    \frac{{n_0^p}}{2(n_0^2-1)^{\frac{1}{2}}} {\lp\frac{2}{m}\rp} +
    \frac{3{\as_j^2}}{40} {\lp\frac{2}{m}\rp^{\frac{4}{3}}}
    \\
    -\as_j\frac{{n_0^p} {(3n_0^2 - 2n_0^{2p})}}{12(n_0^2-1)^{\frac{3}{2}}}
    {\lp\frac{2}{m}\rp^{\frac{5}{3}}} + \Oo\lp {m^{-2}} \rp \Bigg].
\end{multline}
For details, we refer to \cite{LamLeu92} where computations were carried out for a sphere and therefore spherical Bessel's functions appears in this latter case in the modal equation instead of cylindrical Bessel's functions.
In the asymptotic expansion \eqref{eq:1421}, $0<\as_0<\as_1<\as_2<\cdots$ are the successive roots of the flipped Airy function $\As : z\in\C\mapsto \mathsf{Ai}(-z)$ where $\mathsf{Ai}$ denotes the Airy function.
It is important to note that the terms in the asymptotic expansions are \emph{real}: Hence the imaginary part of $k_{p\pv j}(m)$ is contained in the remainder. This part of the resonance set correspond to typical whispering gallery modes.

\begin{figure}[!htbp]
    \begin{tabular}{cccc}
        $m=5$                                                                &
        $m=10$                                                               &
        $m=20$                                                               &
        $m=40$                                                                 \\
        \includegraphics[width=0.2\linewidth]{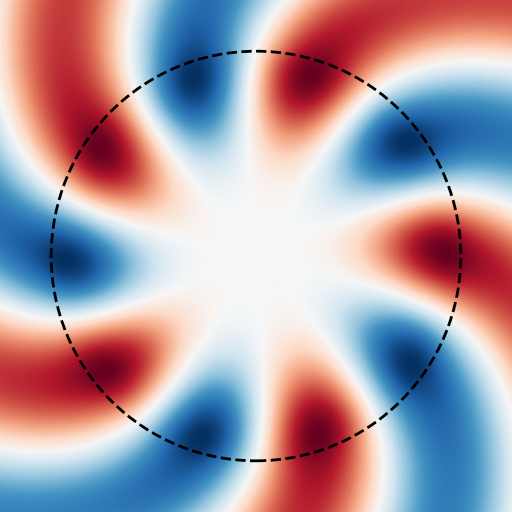} &
        \includegraphics[width=0.2\linewidth]{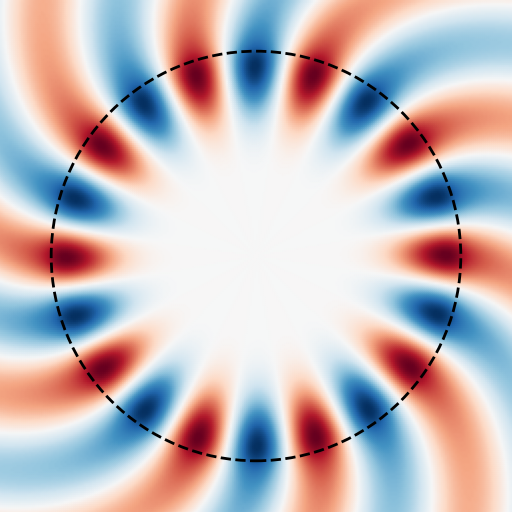} &
        \includegraphics[width=0.2\linewidth]{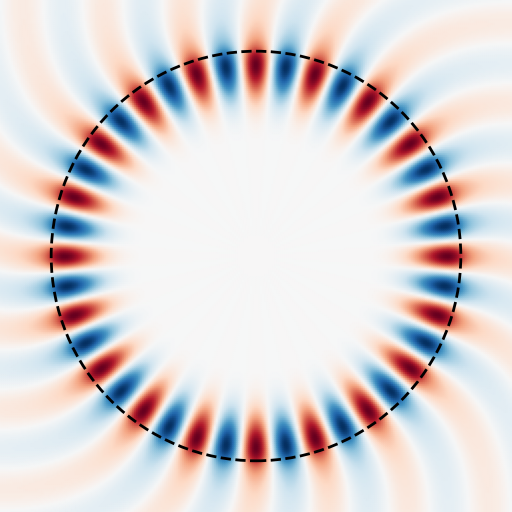} &
        \includegraphics[width=0.2\linewidth]{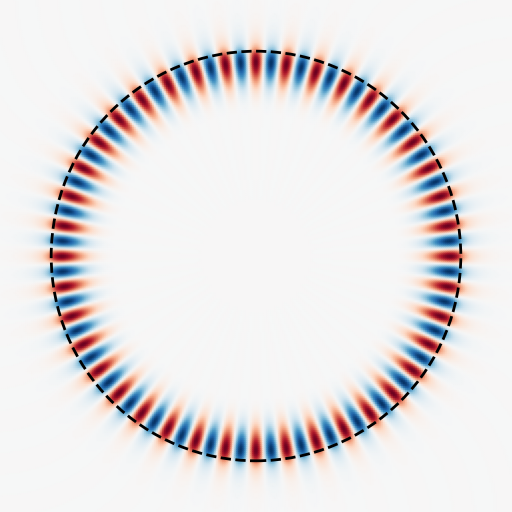}   \\[-1ex]
        \footnotesize $4.64 - \ic\, 10^{-0.54}$                              &
        \footnotesize $8.46 - \ic\, 10^{-0.92}$                              &
        \footnotesize $15.9 - \ic\, 10^{-1.96}$                              &
        \footnotesize $30.1 - \ic\, 10^{-4.74}$                                \\[1ex]
        \includegraphics[width=0.2\linewidth]{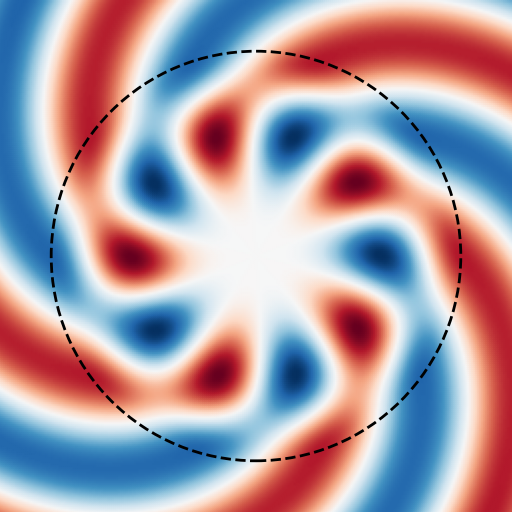} &
        \includegraphics[width=0.2\linewidth]{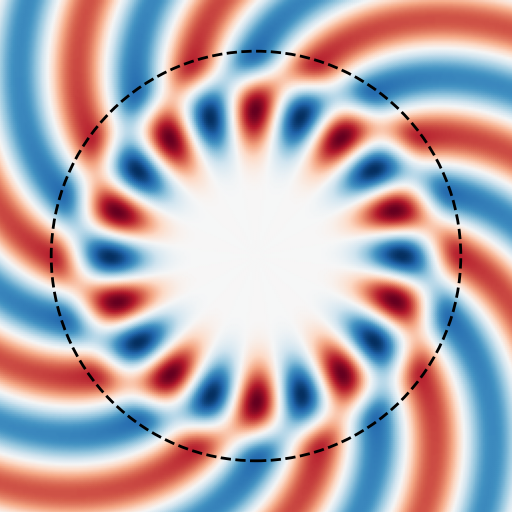} &
        \includegraphics[width=0.2\linewidth]{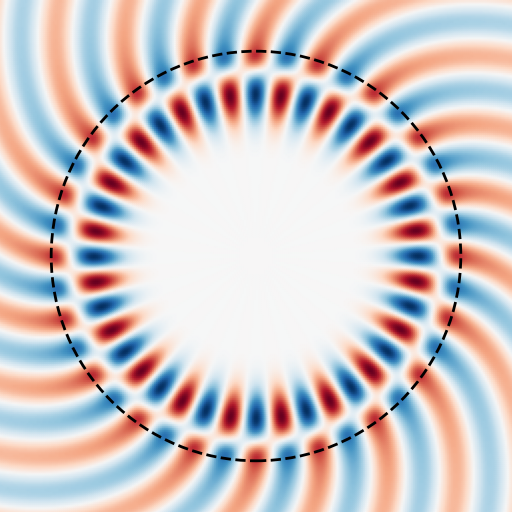} &
        \includegraphics[width=0.2\linewidth]{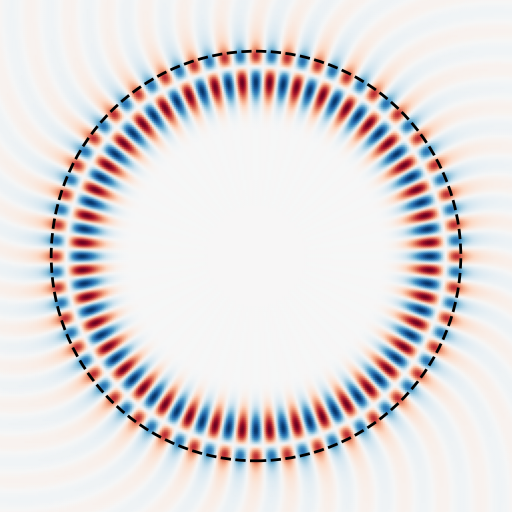}   \\[-1ex]
        \footnotesize $7.08 - \ic\, 10^{-0.34}$                              &
        \footnotesize $11.1 - \ic\, 10^{-0.45}$                              &
        \footnotesize $18.7 - \ic\, 10^{-0.86}$                              &
        \footnotesize $33.6 - \ic\, 10^{-2.59}$                                \\[1ex]
        \includegraphics[width=0.2\linewidth]{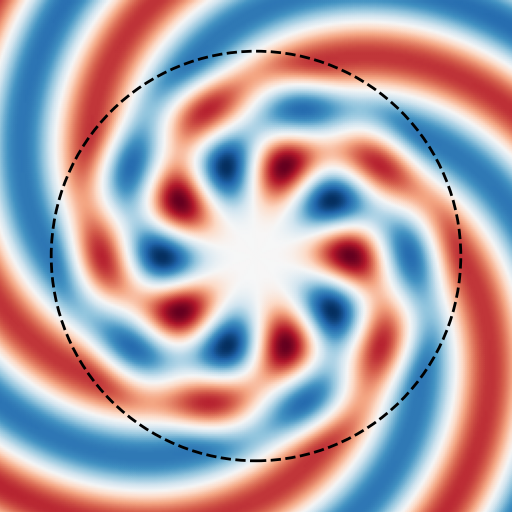} &
        \includegraphics[width=0.2\linewidth]{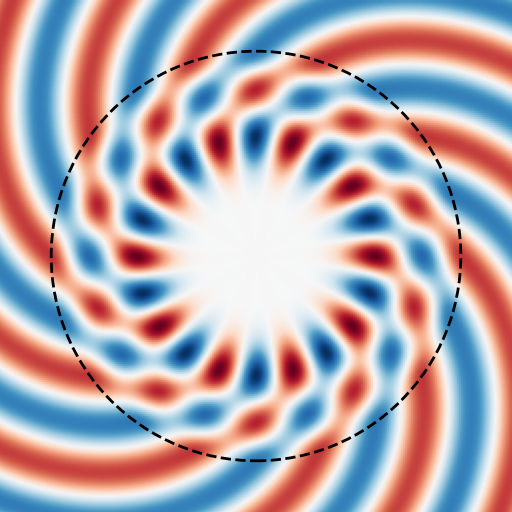} &
        \includegraphics[width=0.2\linewidth]{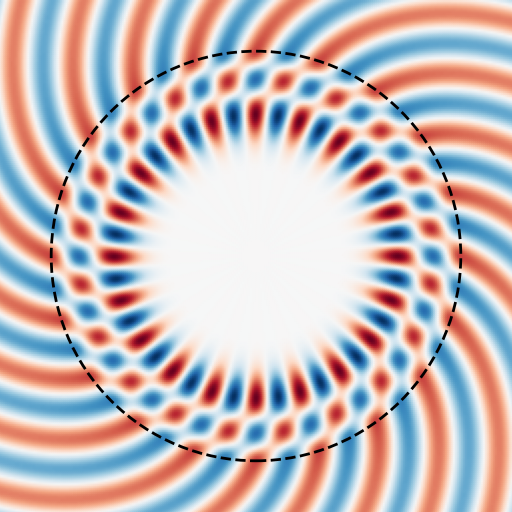} &
        \includegraphics[width=0.2\linewidth]{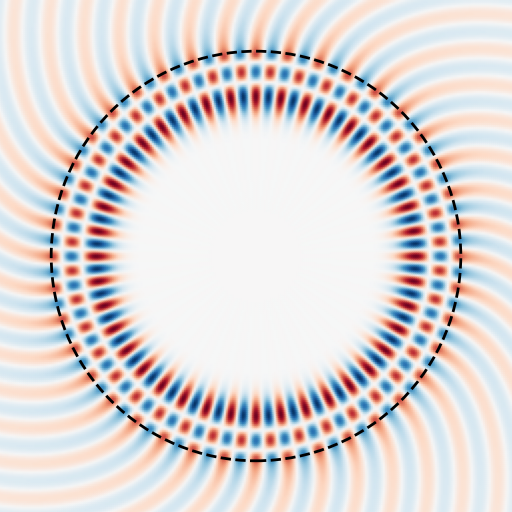}   \\[-1ex]
        \footnotesize $9.36 - \ic\, 10^{-0.30}$                              &
        \footnotesize $13.5 - \ic\, 10^{-0.35}$                              &
        \footnotesize $21.4 - \ic\, 10^{-0.52}$                              &
        \footnotesize $36.6 - \ic\, 10^{-1.39}$
    \end{tabular}
    \caption{Plots of real parts of TM modes $u$ in a circular cavity of radius $R=1$ and index $n=1.5$ (computed by solving the modal equation \eqref{eq:modal} using complex integration \cite{cxroots}).
        Below each plot, we give values of computed resonances $k\in\C$. Each row corresponds to a distinct value of $j$, from $0$ to $2$.}
    \label{fig:modeRes}
\end{figure}

We observe the following whispering gallery mode features when $j$ is chosen and $m$ gets large, see Fig.\ \ref{fig:modeRes}:
\begin{enumerate}
    \item The resonances obtained by solving the modal equation \eqref{eq:modal} have a \emph{negative imaginary part} which tends to zero rapidly (exponentially) when $m\to\infty$.
    \item The modes $u=w_m(r)\,\e^{\ic m \theta}$ with $w_m$ given by \eqref{eq:cmt:2754a} for $k$ solution of  \eqref{eq:modal} concentrate around the interface between the disk and the exterior medium.
\end{enumerate}

%= = = = = = = = = = = = = = = = = = = = = = = = = = = = = 
\subsection{Radially varying index: Main results}
%= = = = = = = = = = = = = = = = = = = = = = = = = = = = = 

The proof of the formula \eqref{eq:1421} given in \cite{LamLeu92} relies on the modal equation \eqref{eq:modal}
and makes use of asymptotic formulas for Bessel functions \cite{Olv97, AbrSte64}.
Such an approach is specific to disks with constant optical index, and the number of terms in the expansion is limited by the asymptotics as $m\to\infty$ of Bessel functions available in the literature.

In this paper we develop a more versatile approach, based on multiscale expansions and semiclassical analysis.
The idea is to consider $h=\frac{1}{m}$ as small parameter and to take advantage of the factor $m^2$ in front of $\frac{1}{r^2}$ in equation \eqref{eq:Pprada} to transform this equation into a semiclassical 1-dimensional Schr\"odinger operator with a singular potential $W$. Generically, $W$ will have a potential well at $r=R$. We perform an asymptotic construction of quasi-resonances and quasi-modes in the vicinity of $r=R$, in such a way that we can rely on the general arguments of the black box scattering to deduce the existence of true resonances close to quasi-resonances modulo $\Oc(m^{-\infty})$, i.e., more rapidly than any polynomial in $\frac{1}{m}$. Unless explicitly mentioned, we suppose the following.

\begin{assumption}
    \label{as:n}
    The radial function $n:r\mapsto n(r)$ satisfies the following properties
    \begin{enumerate}
        \item $n(r)=1$ if $r>R$;
        \item The function $r\mapsto n(r)$ belongs to $\Cc^\infty([0,R])$ and $n(r)>1$ for all $r\le R$.
    \end{enumerate}
\end{assumption}

This assumption motivates the following notations.

\begin{notation}
    \label{no:n}
We denote by $n_i(R)$ the limit values of $\partial^i_r n(r)$ as $r$ tends to $R$ from inside the cavity, namely 
\begin{equation}
    \label{eq:n0}
    n_0 = \lim_{r\nearrow R} n(r), \quad n_1 = \lim_{r\nearrow R} n'(r),\quad n_2 = \lim_{r\nearrow R} n''(r).
\end{equation}
In the sequel we also handle the \emph{effective adimensional curvature} $\kb=R\kappa_{\sf eff}$
\begin{equation}
    \label{eq:kappa}
    \kb \coloneqq  R\lp\frac{1}{R} + \frac{n_1}{n_0}\rp .
\end{equation}
The next quantity arising in asymptotics is the \emph{adimensional hessian}
\begin{equation}
    \label{eq:mu}
    \mb \coloneqq  R^2\lp \frac{2}{R^2} - \frac{n_2}{n_0} \rp .
\end{equation}
\end{notation}

In this paper we prove expansions of resonances as $m\to\infty$ in three distinct cases discriminated by the sign of $\kb$. Such expansions are ``modulo $\Oc(m^{-\infty})$'' in a sense defined below.

\begin{notation}
    \label{Om-infty}
    Let $(a_m)_{m\in\N}$ be a sequence of numbers:
    \begin{align*}
        a_m = \Oc(m^{-\infty}) &  & \text{means that } \forall N\in\N,\ \exists C_N \text{ such that } |a_m| \le C_N\, m^{-N}, &  & \forall m\in\N.
    \end{align*}
\end{notation}

\begin{theorem}
    \label{th:A}
    Assume that the radial function $n$ satisfies Assumption \emph{\ref{as:n}} and
    \begin{equation}
        \label{eq:A}
        \kb >0.
    \end{equation}
    Choose $p\in\{\pm1\}$ and denote by $\Rc_{p}[n,R]$ the resonance set solution to problem \eqref{eq:Pp}.
    Then for any $j\in\N$, there exists a smooth real function $\Kr_{p\pv j}\in \Cc^\infty([0,1]):t\mapsto \Kr_{p\pv j}(t)$ defining distinct sequences
    \begin{equation}
        \label{eq:contkA}
        \ku_{p\pv j}(m) = m \,\Kr_{p\pv j}\lp m^{-\frac{1}{3}}\rp,\quad \forall m\ge1
    \end{equation}
    that are close modulo $\Oc(m^{-\infty})$ to the resonance set $\Rc_{p}[n,R]$, i.e.\ for each $m$, there exists $k_m\in \Rc_{p}[n,R]$ such that $\ku_{p\pv j}(m)-k_m = \Oc(m^{-\infty})$. Let $\Kr^\ell_{p\pv j}$ be the coefficients of the Taylor expansion of $\Kr_{p\pv j}$ at $t=0$. We have, with numbers $\as_j$ being the successive roots of the flipped Airy function,
    \begin{equation}
        \label{eq:thAKl}
        \Kr^0_{p\pv j} = \frac{1}{Rn_0} \,,\quad \Kr^1_{p\pv j} = 0,\quad
        \Kr^2_{p\pv j} = \frac{1}{Rn_0} \,\frac{\as_j}{2} (2\kb)^{\frac{2}{3}} .
    \end{equation}
    All coefficients $\Kr^\ell_{p\pv j}$ are calculable, being the solution of an explicit algorithm involving matrix products and matrix inversions in finite dimensions.
\end{theorem}

We refer to Section \ref{ss:statA} for  expressions of quasi-modes and more terms in the resonance expansions.
As a consequence of Theorem \ref{th:A}, for each chosen $p$ and $j$, there exists a sequence of true resonances $m\mapsto k_{p\pv j}(m)\in \Rc_{p}[n,R]$ such that
\begin{equation}
    \label{eq:th1asy}
    k_{p\pv j}(m) = m \lc\sum_{\ell=0}^{N-1} \Kr^\ell_{p\pv j} \lp\frac{1}{m}\rp^{\frac{\ell}{3}}
    + \Oc\lp\frac{1}{m}\rp^{\frac{N}{3}} \rc \quad \forall N\ge1.
\end{equation}
This clearly generalizes \eqref{eq:1421}.

\medskip

When $\kb$ is zero and if the hessian $\mb$ is positive, we have a similar statement, the powers of $m^{-\frac{1}{3}}$ being replaced by powers of $m^{-\frac{1}{2}}$:

\begin{theorem}
    \label{th:B}
    Assume that the radial function $n$ satisfies Assumption \emph{\ref{as:n}} and 
    \begin{equation}\label{eq:B}
        \kb = 0 \quad \text{with} \quad \mb >0.
    \end{equation}
    Then for any $j\in\N$, there exists a smooth real function $\Kr_{p\pv j}\in \Cc^\infty([0,1]):t\mapsto \Kr_{p\pv j}(t)$ defining distinct sequences
    \begin{equation}
        \label{eq:contkB}
        \ku_{p\pv j}(m) = m \,\Kr_{p\pv j}\lp m^{-\frac{1}{2}}\rp ,\quad \forall m\ge1
    \end{equation}
    that are close modulo $\Oc(m^{-\infty})$ to the resonance set $\Rc_{p}[n,R]$. The first coefficients of the Taylor expansion of $\Kr_{p\pv j}$ at $t=0$ are
    \begin{equation}
        \label{eq:thBKl}
        \Kr^0_{p\pv j} = \frac{1}{Rn_0} \,,\quad \Kr^1_{p\pv j} = 0,\quad
        \Kr^2_{p\pv j} = \frac{1}{Rn_0} \,\frac{(4j+3)\sqrt{\mb}}{2}.
    \end{equation}
\end{theorem}

We refer to Section \ref{ss:statB} for more details.
Note that, in contrast to the case $\kb>0$ , the coefficients $\Kr^\ell_{p\pv j}$ are not calculable (except the first four of them) in the sense that their determination needs the inversion of infinite dimensional matrices.

\medskip

Unlike the two previous cases for which the quasi-modes are localized near the interface $r=R$, in the third case the quasi-modes are localized near an internal circle $r=R_0$ with some $R_0<R$.

\begin{theorem}
    \label{th:C}
    Assume that the radial function $n$ satisfies Assumption \emph{\ref{as:n}} and that $\kb < 0$. Let $R_0\in(0,R)$ such that $1 + \frac{R_0 n'(R_0)}{n(R_0)}=0$ and assume further that
    \begin{equation}
        \label{eq:C}
        \mb_0 \coloneqq R_0^2 \lp \frac{2}{R_0^2} - \frac{n''(R_0)}{n(R_0)} \rp >0.
    \end{equation}
    Then a similar statement as in Theorem \emph{\ref{th:B}} holds.  The first coefficients of the Taylor expansion of $\Kr_{p\pv j}$ at $t=0$ are now, instead of \eqref{eq:thBKl},
    \begin{equation}
        \label{eq:thCKl}
        \Kr^0_{p\pv j} = \frac{1}{R_0n(R_0)} \,,\quad \Kr^1_{p\pv j} = 0,\quad
        \Kr^2_{p\pv j} = \frac{1}{R_0n(R_0)} \,\frac{(2j+3)\sqrt{\mb_0}}{2}.
    \end{equation}
\end{theorem}

We refer to Section \ref{ss:statC} for more details.
Note that in this case the coefficients $\Kr^\ell_{p\pv j}$ are all calculable in the sense introduced in Theorem \ref{th:A}.

\begin{example}
    The parametric profile \eqref{eq:Ilc} illustrates the three above cases:
    Depending on the parameter $\delta$ the profile fits the assumptions of the three theorems \ref{th:A}, \ref{th:B}, and \ref{th:C}.
    Here $\kb = 1-\frac{\delta R}{2}$.
    \begin{itemize}
        \item[\textsc{(a)}] If $\delta < \frac{2}{R}$ then $\kb > 0$, hence Theorem \ref{th:A} applies.
        \item[\textsc{(b)}] If $\delta = \frac{2}{R}$ then $\kb = 0$ and $\mb = 3$, hence Theorem \ref{th:B} applies.
        \item[\textsc{(c)}] If $\delta > \frac{2}{R}$ then $\kb < 0$, $R_0 = \frac{2(1 + \delta R)}{3\tau}<R$, and $\mb_0 = 3$, hence Theorem \ref{th:C} applies.
    \end{itemize}
\end{example}

\begin{remark}
\label{eq:multwell}
Theorems \ref{th:A}, \ref{th:B}, \ref{th:C} state that there exist true resonances that are super-algebraically close to our constructed quasi-resonances. This does not mean that all resonances are described by this construction. On the one hand, the optical index $n$ may generate an effective potential with multiple wells: For example the same index $n$ can generate quasi-resonances of type \textsc{(a)} and of type \textsc{(c)}. Likewise our analysis still applies if the index $n$ is piecewise smooth with a finite number of jumps. Then we may have several half-triangular wells of type \textsc{(a)} and the related quasi-resonances. On the other hand, the whole landscape of resonances is wide: It contains families of outer resonances and inner resonances that are not of WGM type, see Fig.\ \ref{fig:Bes_all_m} and also \cite{Ste06}.
\end{remark}

%============================
\section{Families of resonance quasi-pairs}
\label{sec:quasi}
%============================

Inspired by quasi-pair constructions used to investigate ground states in semiclassical analysis of Schr\"odinger operators (see \textsc{Simon} \cite{Sim83} for instance), we are going to construct families of resonance quasi-pairs for problem \eqref{eq:Pp}.
Here appears a fundamental difference:
The quasi-pair construction in semiclassical analysis consists in building approximate eigenpairs $(\lambda_h,u_h)$ that solve $A_hu_h=\lambda_hu_h$ with increasingly small error as $h \to 0$ where $A_h$ is for instance the operator $-h^2\Delta + V$. In our case, we do not have any given semiclassical parameter $h$.
However, the term $\frac{m^2}{r^2}$ in equation of \eqref{eq:Pprada} may play the role of a confining potential in a semiclassical framework if we set
\[
    h = \frac{1}{|m|}\,.
\]
This means that an internal frequency parameter $m$ can be viewed as a driving parameter for an asymptotic study. This leads to the next definition for quasi-pairs, adapted to our problem.

\begin{definition}
    \label{def:quasi}
    Choose $p\in\{\pm1\}$.
    A \emph{family of resonance quasi-pairs} $\Fg_p$ for problem \eqref{eq:Pp} is formed by a sequence $\Kg_p = (\ku(m))_{m\ge1}$ of real numbers called quasi-resonances and a sequence $\Ug_p=(\uu(m))_{m\ge1}$ of complex valued functions called quasi-modes, where for each $m\ge1$, the couple $(\ku(m),\uu(m))$ is a quasi-pair for problem \eqref{eq:Pp} with an error in $\Oo\lp m^{-\infty}\rp$ when $m\to\infty$.
    More precisely, we mean that
    \begin{enumerate}
        \item For any $m\ge 1$, the functions $\uu(m)$ belong to the domain of the operator and are normalized,
              \begin{align*}
                  \uu(m)\in\Hr_p^{2}(\R^2,\Omega) \quad \text{and} \quad
                  \lo \uu(m)\ro_{\Lr^2(\R^2)} = 1
              \end{align*}
              where
              \begin{align}
                  \label{eq:dom}
                  \Hr_p^{2}(\R^2,\Omega) = \bigl\{u\in \Lr^2(\R^2) \ \bigm\vert\
                   & u\on{\Omega}\in \Hr^2(\Omega),\
                  u\on{\R^2\setminus\overline\Omega}\in \Hr^2(\R^2\setminus\overline\Omega),
                  \nonumber                                                                                          \\
                   & [u]_{\partial\Omega} = 0,\ \texte{and}\ [n^{p-1} \,\partial_\nu u]_{\partial\Omega} = 0\bigr\}.
              \end{align}
        \item We have the following quasi-pair estimate as $m \to +\infty$,
              \begin{align}\label{eq:minfty}
                  \lo -\div\lp n^{p-1}\, \nabla \uu(m)\rp - \ku(m)^2\, n^{p+1}\, \uu(m) \ro_{\Lr^2(\R^2)} =
                  \Oo\lp m^{-\infty}\rp .
              \end{align}
        \item\label{it:Xg}
              Uniform localization: There exists a function $\Xg\in \Cc^\infty_0(\R^2)$, $0\le\Xg\le1$, such that
              \begin{align*}
                  \lo \Xg\uu(m)\ro_{\Lr^2(\R^2)}\ge\tfrac12 \quad\mbox{and\quad \eqref{eq:minfty} holds with
                      $\Xg\uu(m)$ replacing $\uu(m)$}.
              \end{align*}
        \item Regularity with respect to $m$: There exist a positive real number $\beta$ and a smooth function $\Kr\in \Cc^\infty([0,1]):t\mapsto\Kr(t)$ such that
              \begin{equation}
                  \label{eq:contk}
                  \frac{\ku(m)}{m} = \Kr(m^{-\beta})\quad \forall m\ge1.
              \end{equation}
    \end{enumerate}
    If the cut-off function $\Xg$ in item \eqref{it:Xg} can be taken as any function that is $\equiv1$ in a neighborhood of $\partial\Omega$ for $m$ large enough, we say that the family $\Fg_p$ is a family of \emph{whispering gallery type}.
\end{definition}

\begin{remark}
    \label{rem:qua1}
    By Taylor expansion of the function $\Kr$ at $t=0$, we obtain that a consequence of \eqref{eq:contk} is the existence of coefficients $\Kr_\ell$, $\ell\in\N$, and constants $C_N$ such that
    \begin{equation}
        \label{eq:contkN}
        \forall N\ge1,\qquad
        \left|\frac{\ku(m)}{m} - \sum_{\ell=0}^{N-1} \,\Kr_\ell\, m^{-\ell\beta} \right|
        \le C_N\, m^{-N\beta}.
    \end{equation}
    Note that the asymptotics \eqref{eq:1421} satisfies such an estimate with $\beta=\frac{1}{3}$ and $N= 6$.
\end{remark}

\begin{remark}
    \label{rem:qua4}
    The estimate \eqref{eq:minfty} implies a bound from below for the resolvent of the underlying operator, compare with \cite[\S6]{MoiSpe19}: At quasi-resonances, we have a \emph{blow up of the resolvent}.
\end{remark}

%============================
\section{Classification of the three typical behaviors by a Schr\"odinger analogy}
\label{sec:3types}
%============================

In our way to prove Theorems \ref{th:A}--\ref{th:C}, we transform the family of problems \eqref{eq:Pprad} when $m$ spans $\N^*$ into a family of 1-dimensional Schr\"odinger operators depending on the semiclassical parameter
$h = \frac{1}{m}$.

Namely, choosing a polar mode index $m\in\N^*$ and coming back to the ODE contained in problem \eqref{eq:Pprad} divided by $n^{p+1}$, we obtain the equation:
\begin{equation}
    \label{eq:1730a}
    -\frac{1}{r n^{p+1}}\partial_r(n^{p-1}r\partial_r w) +
    \frac{m^2}{r^2 n^2} w - k^2\, w = 0
\end{equation}
As a start, we write a quasi-resonance as (compare with \eqref{eq:1421})
\[
    \ku(m)^2 = m^2 E 
\]
where the energy $E$ depends on $h$ and has to be found.
Multiplying \eqref{eq:1730a} by $h^2=1/m^2$, we find that \eqref{eq:1730a} takes the form of a one dimensional semiclassical Schr\"odinger modal equation
\begin{align}\label{eq:schr}
    -h^2\, \Hc w + W w =  E w ,
\end{align}
where $\Hc$ is the second order differential operator
\begin{equation}
    \label{eq:Hc}
    \Hc = \frac{1}{r n^{p+1}}\partial_r(n^{p-1}r\partial_r)
\end{equation}
and $W$ is the potential
\begin{equation}
    \label{eq:W}
    W(r) = \lp\frac{1}{r\,n(r)}\rp^2 \,.
\end{equation}
The operator $-h^2\, \Hc + W$ is self-adjoint on $\Lr^2(\R_+,n^{p+1}r\dd r)$.
We note that
\begin{equation}
    \label{eq:VR}
    \lim_{r\nearrow R} W(r) = \lp\frac{1}{R\,n_0}\rp^2 \quad\mbox{and}\quad
    \lim_{r\searrow R} W(r) = \lp\frac{1}{R}\rp^2\,.
\end{equation}
Since $n_0>1$, we have a potential barrier at $r=R$. The first and second derivatives of $W$ on $(0,R]$ are given by
\begin{subequations}
    \begin{align}
        W'(r)  & = -2\lp\frac{1}{r\,n(r)}\rp^2 \lc \frac{1}{r} + \frac{n'(r)}{n(r)} \rc ,                                                        \\
        W''(r) & = 2\lp\frac{1}{r\,n(r)}\rp^2 \lc 3\lp\frac{1}{r} + \frac{n'(r)}{n(r)}\rp^2 - \frac{2 n'(r)}{r n(r)} - \frac{n''(r)}{n(r)} \rc .
    \end{align}
\end{subequations}
The local minima (potential wells) of $W$ cause the existence of resonances near these energy levels and their asymptotic structure as $h\to0$ is determined by the Taylor expansion of $W$ at its local minima. Let us recall that
$\kb = R\big(\frac{1}{R} + \frac{n_1}{n_0}\big)$ using notation \eqref{eq:n0}.
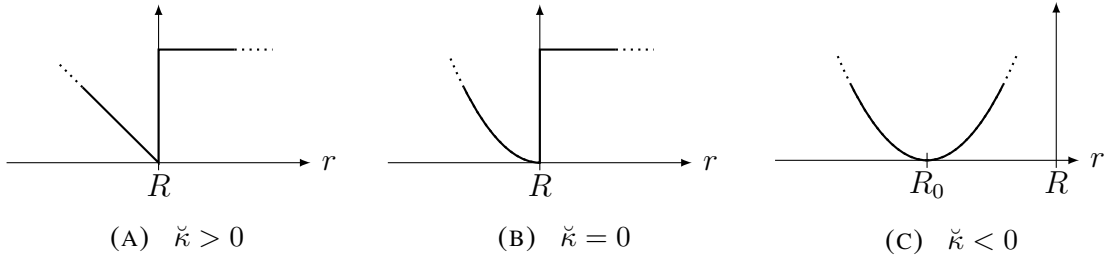
\begin{figure}[t]
    \begin{subfigure}[ht]{0.3\linewidth}
        \centering
        \begin{tikzpicture}
            \draw[>=latex,->] (-2,0) -- (2,0);
            \draw[>=latex,->] (0,-0.1) -- (0,2.1);
            \draw (0,0) node[below] {$R$};
            \draw (2,0) node[right] {$r$};
            \draw[thick] (-1,1) -- (0,0) -- (0,1.5) -- (1,1.5);
            \draw[thick,dotted] (1,1.5) -- (1.5,1.5);
            \draw[thick,dotted] (-1.3,1.3) -- (-1,1);
        \end{tikzpicture}
        \caption{\;\;$\kb > 0$}
    \end{subfigure}
    \;\;
    \begin{subfigure}[ht]{0.3\linewidth}
        \begin{tikzpicture}
            \draw[>=latex,->] (-2,0) -- (2,0);
            \draw[>=latex,->] (0,-0.1) -- (0,2.1);
            \draw (0,0) node[below] {$R$};
            \draw (2,0) node[right] {$r$};
            \draw[thick,dotted] (0,0) parabola (-1.2,1.44);
            \draw[thick] (1,1.5) -- (0,1.5) -- (0,0) parabola (-1,1);
            \draw[thick,dotted] (1,1.5) -- (1.5,1.5);
        \end{tikzpicture}
        \caption{\;\;$\kb = 0$}
    \end{subfigure}
    \;\;
    \begin{subfigure}[ht]{0.3\linewidth}
        \begin{tikzpicture}
            \draw[>=latex,->] (-2,0) -- (2,0);
            \draw[>=latex,->] (1.7,-0.1) -- (1.7,2.1);
            \draw (0,-0.1) -- (0,0.1);
            \draw (1.7,0) node[below] {$R$};
            \draw (0,0) node[below] {$R_0$};
            \draw (2,0) node[right] {$r$};
            \draw[thick,dotted] (0,0) parabola (1.2,1.44);
            \draw[thick,dotted] (0,0) parabola (-1.2,1.44);
            \draw[thick] (0,0) parabola (1,1);
            \draw[thick] (0,0) parabola (-1,1);
        \end{tikzpicture}
        \caption{\;\;$\kb < 0$}
    \end{subfigure}
    \caption{The three typical local behaviors of the potential $W$: half-triangular potential well ($\kb>0$),  half-quadratic potential well ($\kb=0$)
        and quadratic potential well ($\kb<0$).}
    \label{fig:locModPot}
\end{figure}
The sign of $\kb$ (if it is positive, zero, or negative) discriminates three typical behaviors in which case we can construct families of resonance quasi-pairs (see Theorems \ref{th:A}, \ref{th:B}, \ref{th:C}):
\begin{itemize}
  \item[\sc (a)] $\kb>0$. Then $W$ is decreasing on a left neighborhood of $R$ and has a local minimum at $R$. In a two-sided neighborhood of $R$, $W$ is tangent to a \emph{half-triangular potential well}, see Fig.\ \ref{fig:locModPot} \textsc{(a)}.
  \item[\sc (b)] $\kb=0$. In this case, we assume that $W''(R)>0$, which is ensured by the condition
        \begin{equation}
            \label{eq:(b)}
            \frac{2}{R^2} - \frac{n_2}{n_0} >0
            \quad\mbox{while}\quad \frac{1}{R} + \frac{n_1}{n_0} = 0 .
        \end{equation}
        Then $W$ has a local minimum at $R$. In a two-sided neighborhood of $R$, $W$ is tangent to a \emph{half-quadratic potential well}, see Fig.\ \ref{fig:locModPot} \textsc{(b)}.
  \item[\sc (c)] $\kb<0$. Then $W$ has no local minimum at $R$. But, since $\lim_{r \to 0^+} W(r) = +\infty$, it has at least one local interior minimum $R_0$ over $(0,R)$. Now we assume that $W''(R_0)>0$, which is ensured by the condition
        \begin{equation}
            \label{eq:(c)}
            \frac{2}{R_0^2} - \frac{n''(R_0)}{n(R_0)} >0
            \quad\mbox{while}\quad \frac{1}{R_0} + \frac{n'(R_0)}{n(R_0)} = 0 .
        \end{equation}
        Then $W$ has a local non-degenerate minimum at $R_0$ where it is tangent to a \emph{quadratic potential well}, see Fig.\ \ref{fig:locModPot} \textsc{(c)}.
\end{itemize}

%============================
\section{Case \textsc{(a)} \ Half-triangular potential well}
\label{sec:(a)}
%============================
The case $\kb>0$ is in a certain sense the most canonical one, since it includes constant optical indices $n\equiv n_0$ inside $\Omega$.
In this section, after stating the result, we perform the details of construction of families of resonance quasi-pairs.

%= = = = = = = = = = = = = = = = = = = = = = = = = = = = = 
\subsection{Statements}
\label{ss:statA}
%= = = = = = = = = = = = = = = = = = = = = = = = = = = = = 
Recall that Assumption \ref{as:n} is supposed to hold and that $n_0$, $\kb$ and $\mb$ are defined in Notation \ref{no:n}. We give now, in the case when $\kb$ is positive, the complete description of the quasi-pairs that we construct in the rest of this section.
This statement has to be combined with Theorem \ref{th:QuasiRes} to imply Theorem \ref{th:A}.

\begin{theorem}
    \label{th:(a)}
    Choose $p\in\{\pm1\}$.
    If $\kb>0$, there exists for each natural integer $j$, a family of resonance quasi-pairs $\Fg_{p\pv j}=(\Kg_{p\pv j},\Ug_{p\pv j})$ of whispering gallery type ({\em cf.}  Definition \emph{\ref{def:quasi}}) for which the sequence of numbers $\Kg_{p\pv j} = (\ku_{p\pv j}(m))_{m\ge1}$ and the sequence of functions $\Ug_{p\pv j} = (\uu_{p\pv j}(m))_{m\ge1}$ have the following properties: \smallskip

    \medskip
    (i) The regularity property \eqref{eq:contk}--\eqref{eq:contkN} with respect to $m$ holds with $\beta=\frac{1}{3}$: There exist coefficients $\Kr_{p\pv j}^{\,\ell}$ for any $\ell\in\N$, and constants $C_N$ such that
    \begin{equation}
        \label{eq:kNa}
        \forall N\ge1,\qquad
        \left|\frac{\ku_{p\pv j}(m)}{m} -
        \sum_{\ell=0}^{N-1} \,\Kr_{p\pv j}^{\,\ell}\, m^{-\frac{\ell}{3}} \right|
        \le C_N\, m^{-N/3}.
    \end{equation}
    The coefficients $\Kr_{p\pv j}^{\,0}$ (degree $0$) are all equal to $\frac{1}{Rn_0}$, the coefficients of degree $1$ are zero, and the coefficients of degree $2$ are all distinct with $j$, see \eqref{eq:0915a}.

    \medskip
    (ii) The functions $\uu_{p\pv j}(m)$ forming the sequence $\Ug_{p\pv j}$ have the form
    \begin{equation}
        \label{eq:ua}
        \uu_{p\pv j}(m\pv x,y) = \wu_{p\pv j}(m\pv r)\, \e^{\ic m \theta}
    \end{equation}
    where the radial functions $\wu_{p\pv j}(m)$ have a boundary layer structure around $r=R$ with different scaled variables $\sigma$  as $r<R$ and $\rho$ as $r>R$:
    \begin{equation}
        \label{eq:scaA}
        \sigma=m^{\frac{2}{3}}\lp\frac{r}{R}-1\rp \ \ \texte{if}\ \ r<R \quad\mbox{ and }\quad
        \rho=m\lp\frac{r}{R}-1\rp \ \ \texte{if}\ \ r>R.
    \end{equation}
    This means that there exist smooth functions $\Phi_{p\pv j}\in \Cc^\infty([0,1],\Sc(\R_-)): (t,\sigma)\mapsto \Phi_{p\pv j}(t,\sigma)$ and $\Psi_{p\pv j}\in \Cc^\infty([0,1],\Sc(\R_+)): (t,\rho)\mapsto \Psi_{p\pv j}(t,\rho)$ such that
    \begin{equation}
        \label{eq:va}
        \begin{aligned}
            \wu_{p\pv j}(m\pv r) = \Xg(r)\Big( & \II_{r<R}(r)\,\Phi_{p\pv j}(m^{-\frac{1}{3}},\sigma) +
            \II_{r>R}(r)\, \Psi_{p\pv j}(m^{-\frac{1}{3}},\rho ) \Big)
        \end{aligned}
    \end{equation}
    where $\Xg\in\Cc^\infty_0(\R_+)$,  $\Xg\equiv 1$ in a neighborhood of $R$.
\end{theorem}

The first terms of the expansions of the quasi-pairs $(\ku_{p\pv j}(m),\wu_{p\pv j}(m))$ in powers of $m^{\frac{1}{3}}$ are given below.

\subsubsection{Resonances}
\label{sss:resA}
The asymptotics of $\ku_{p\pv j}(m)$ starts as
\begin{multline}\label{eq:0915a}
    \ku_{p\pv j}(m) = {\frac{m}{R{n_0}}} \lc 1 +
    \frac{\as_j}{2} {\lp\frac{2\kb}{m}\rp^{\frac{2}{3}}} -
    \frac{n_0^p}{2\sqrt{n_0^2-1}}{\lp\frac{2\kb}{m}\rp} \right.\\ + \left.
    \kr_{p\pv j}^4 \lp\frac{2\kb}{m}\rp^{\frac{4}{3}} +
    \kr_{p\pv j}^5 \lp\frac{2\kb}{m}\rp^{\frac{5}{3}} +
    \Oo\lp {m^{-2}} \rp \rc
\end{multline}
where, as before, the $\as_j$ are the successive roots of the flipped Airy function and the coefficients $\kr_{p\pv j}^4$ and $\kr_{p\pv j}^5$
are given by
\begin{align*}
    \kr_{p\pv j}^4 & = \frac{\as_j^2}{15}\lp\frac{17}{8} - \frac{3}{\kb} + 
    \frac{\mb}{\kb^2} 
    \rp,                                         \\[1.5ex]
    \kr_{p\pv j}^5 & = -\frac{\as_j\, n_0^p}{12\sqrt{n_0^2-1}}\lp\frac{3n_0^2-2n_0^{2p}}{n_0^2-1}+2-
    \frac{6}{\kb}+\frac{2\mb}{\kb^2} 
    \rp .
\end{align*}

\begin{remark}
    \label{rem:(a)1}
    Note that the second term of \eqref{eq:0915a} separates the families $\Fg_{p\pv j}$, while the third term distinguishes the TM ($p=1$) and TE ($p=-1$) modes. In the case of Dirichlet conditions, the first two terms are the same, whereas the third one is zero, \cite[Sec.\ 7.4]{BabBul61}.
\end{remark}

\subsubsection{Modes}
The asymptotic expansions of the radial part of the quasi-modes $\wu_{p\pv j}(m)$ in~\eqref{eq:ua} starts as
\begin{equation}\label{eq:A_va01}
    \wu_{p\pv j}(m\pv r) =
    \Xg(r) \lp \Vs_{p\pv j}^{\,0}(m;r) +
    \lp\frac{1}{m}\rp^{\frac{1}{3}} \Vs_{p\pv j}^{\,1}(m;r)\rp
    + \Oo\lp m^{-\frac{2}{3}}\rp ,
\end{equation}
where, using the scaled variables $\sigma=m^{\frac{2}{3}}(\frac{r}{R}-1)$ and $\rho=m(\frac{r}{R}-1)$
\begin{equation}\label{eq:A_V0}
    \Vs_{p\pv j}^{\,0}(m;r) = \left\{\begin{array}{ll}
        \As\lp \as_j + (2\kb)^{\frac{1}{3}}\sigma\rp & \texte{if } r < R,   \\
        0                                            & \texte{if } r \ge R,
    \end{array}\right.
\end{equation}
and
\begin{equation}\label{eq:A_V1}
    \Vs_{p\pv j}^{\,1}(m;r) = \frac{-n_0^p\, (2\kb)^{\frac{1}{3}}}{\sqrt{n_0^2-1}}
    \left\{\begin{array}{ll}
        \As'\big( \as_j + (2\kb)^{\frac{1}{3}} \sigma\big)         & \texte{if }r < R,      \\
        \As'(\as_j)\exp\big(-\frac{\sqrt{n_0^2-1}}{n_0}\,\rho\big) & \texte{if } r \ge R\,.
    \end{array}\right.
\end{equation}

\subsubsection{Special case of a constant optical index}

In the constant index case $n\equiv n_0>1$, we have $\kb=1$ and we find the following $8$-term expansion of the resonances for the circular cavity. For an improved readability, we distinguish the TM and TE case and we denote $\ku_{p\pv j}$ by $\ku^{\TM}_{j}$ if $p=1$ and by $\ku^{\TE}_{j}$ if $p=-1$. We have
\begin{align}\label{eq:kjTM}\hskip-2ex
    k^{\TM}_j(m) = \frac{m}{R n_0} \Bigg[ 1
     & + \frac{\as_j}{2}\lp\frac{2}{m}\rp^{\frac{2}{3}}
        - \frac{n_0}{2(n_0^2-1)^{\frac{1}{2}}}\lp\frac{2}{m}\rp
    + \frac{3\,\as_j^2}{40}\lp\frac{2}{m}\rp^{\frac{4}{3}} \nonumber                                                                \\
     & - \frac{\as_j\, n_0^3}{12(n_0^2-1)^{\frac{3}{2}}}\lp\frac{2}{m}\rp^{\frac{5}{3}}
        + \frac{10-\as_j^3}{2800}\lp\frac{2}{m}\rp^2
    + \frac{\as_j^2\, n_0^2\, (n_0^2-4)}{80(n_0^2-1)^{\frac{5}{2}}}\lp\frac{2}{m}\rp^{\frac{7}{3}} \nonumber                        \\
     & - \frac{\as_j}{144}\lp\frac{1}{175}+\frac{479\,\as_j^3}{7000} +\frac{2 n_0^6}{(n_0^2-1)^3}\rp\lp\frac{2}{m}\rp^{\frac{8}{3}}
        + \Oo\lp m^{-3}\rp \Bigg]
\end{align}
and
\begin{align}\label{eq:kjTE}\hskip-2ex
    k^{\TE}_j(m) = \frac{m}{R n_0} \Bigg[ 1
     & + \frac{\as_j}{2}\lp\frac{2}{m}\rp^{\frac{2}{3}}
        - \frac{1}{2n_0(n_0^2-1)^{\frac{1}{2}}}\lp\frac{2}{m}\rp
    + \frac{3\,\as_j^2}{40}\lp\frac{2}{m}\rp^{\frac{4}{3}} \nonumber                                                                                                          \\
     & - \frac{\as_j\, (3n_0^2-2)}{12n_0^3(n_0^2-1)^{\frac{3}{2}}}\lp\frac{2}{m}\rp^{\frac{5}{3}}
    + \frac{1}{8}\lp\frac{1}{35} -\frac{\as_j^3}{350} +\frac{1}{n_0^4(n_0^2-1)^2}\rp\lp\frac{2}{m}\rp^2 \nonumber                                                             \\
     & + \frac{\as_j^2\, (3n_0^8+12n_0^6-12n_0^4-8n_0^2+8)}{80n_0^5(n_0^2-1)^{\frac{5}{2}}}\lp\frac{2}{m}\rp^{\frac{7}{3}} \nonumber                                          \\
     & - \frac{\as_j}{144}\lp\frac{1}{175}+\frac{479\,\as_j^3}{7000} +\frac{18n_0^8-45n_0^6+12n_0^4+45n_0^2-28}{n_0^6(n_0^2-1)^3}\rp\lp\frac{2}{m}\rp^{\frac{8}{3}} \nonumber \\
     & + \Oo\lp m^{-3}\rp \Bigg]
\end{align}

%= = = = = = = = = = = = = = = = = = = = = = = = = = = = = 
\subsection{Proof: General concepts}\label{subsec:genConcepts}
%= = = = = = = = = = = = = = = = = = = = = = = = = = = = = 

As explained in Sect.\ \ref{sec:3types}, the problem under consideration has the form \eqref{eq:schr} of the semi-classical Schr\"odinger equation $-h^2\, \Hc w + W w =  E w$, where $\Hc$ is a modified Laplacian, $W$ is a potential, discontinuous at the interface $r=R$, and $h = \frac{1}{m}$ is the semiclassical parameter. Recall that in both cases \textsc{(a)} and \textsc{(b)}, the potential $W$ has a local minimum at $R$, with the distinctive feature that for $r<R$, the shape of $W$ is triangular in case \textsc{(a)}, and quadratic in case \textsc{(b)}. The rationale of the quasi-resonance construction is to localize equation around the well bottom $r=R$ and to scale variables appropriately so that equation \eqref{eq:schr} can be solved by a multiscale power expansion.
In this section, we describe the general concepts of the proof, common to the two cases \textsc{(a)} and \textsc{(b)}.

\subsubsection{Localization around the interface}
The localization starts with the introduction of the dimensionless variable
$\xi = \frac{r}{R}-1 \in (-1,+\infty)$ for which the disk boundary is translated to the origin.
Accordingly, we denote by $\nx$ the optical index function in this new variable, viz
\[
    \nx : \xi \mapsto n\big(R(1+\xi)\big)
\]
and, for all $q\in\N$, we set  $\nx_q= \nx^{(q)}(0)$, the $q$-th derivative of $\nx$ at $0$. Referring to Notation~\ref{no:n}, we have $\nx_q = R^q n_q$ and ({\em cf.} \eqref{eq:kappa} and \eqref{eq:B})
\begin{align}\label{eq:kb_mb}
    \kb = 1 + \frac{\nx_1}{\nx_0} \quad\mbox{and}\quad
    \mb = 2 - \frac{\nx_2}{\nx_0}.
\end{align}
Since $\nx_0=\no_0$, we will most often use the notation $\no_0$.

The minimum of $W$ at $r=R$ is its left limit $W_0 \coloneqq \lim_{r\nearrow R}W(r)= (R\no_0)^{-2}$. Using the change of variables $r\mapsto\xi$, we set
\begin{align*}
    \Lc(\xi,\partial_\xi) = R^2 \no_0^2 \Hc(r,\partial_r), &  &
    V(\xi) = R^2 \no_0^2 (W(r) - W_0 ),                    &  &
    \text{and}                                             &  &
    \wt{ E} = R^2 \no_0^2 ( E - W_0),
\end{align*}
so that equation \eqref{eq:schr} is transformed into
\begin{subequations}
    \begin{align}\label{eq:schr_loc}
        -h^2\, \Lc v + V v = \wt{ E} v ,
    \end{align}
    with the new unknown function $v(\xi) = w(R(1+\xi))$. We have
    \begin{align*}
        \Lc(\xi,\partial_\xi) & = \frac{\no_0^2}{\nx(\xi)^2}\, \partial_\xi^2 +
        \no_0^2 \lp\frac{1}{(1+\xi)\,\nx(\xi)^2}+(p-1)\frac{\nx'(\xi)}{\nx(\xi)^3}\rp\partial_\xi \\[1ex]
        V(\xi)                & = \lp\frac{\no_0}{(1+\xi)\,\nx(\xi)}\rp^2 - 1.
    \end{align*}
    Note that the potential $V$ has $0$ as local minimum at $\xi=0$ (well bottom).

    The unknown function $v$ satisfies furthermore the following jump condition at $\xi=0$ deduced from \eqref{eq:Ppradb}
    \begin{equation}\label{eq:jumps_loc}
        \lc v\rc_{\{0\}} = 0 \quad \text{and} \quad \lc \nx^{p-1}\, \partial_\xi v\rc_{\{0\}} = 0
    \end{equation}
    and the decay conditions in Schwarz spaces when $\xi\to\pm\infty$:
    \begin{equation}
        \label{eq:schr_lo}
        v^- \coloneqq v\on{\R_-} \in \Sc(\R_-)\quad\mbox{and}\quad
        v^+ \coloneqq v\on{\R_+} \in \Sc(\R_+).
    \end{equation}
\end{subequations}
Our concern now is to construct quasi-resonances and quasi-modes localized around the well bottom $\xi=0$, solutions to \eqref{eq:schr_loc}--\eqref{eq:schr_lo} in an asymptotic sense.

\subsubsection{Principal part of the Schr\"odinger modal equation}
The structure of quasi-pairs is determined by the {\em principal part} of problem \eqref{eq:schr_loc}--\eqref{eq:schr_lo} defined as:
\begin{subequations}
    \begin{equation}
        \label{eq:schr0}
        -h^2\, \Lcu_0 v  + \Vu_0 v = \Lu v
    \end{equation}
    where
    \begin{enumerate}
        \item The operator $\Lcu_0=(\Lcu_0^-,\Lcu_0^+)=(\partial_\xi^2,\no_0^2\partial_\xi^2)$ is the principal part of $\Lc$ frozen at $\xi=0$ on the left and on the right,

        \item The associated jump conditions are
              \begin{equation}
                  \label{eq:jump0}
                  v^-(0) = v^+(0) \qquad \text{and} \qquad \no^{p-1}_0\, \partial_\xi v^-(0) = \partial_\xi v^+(0)
              \end{equation}
              and the decay condition is the same as above
              \begin{equation}
                  \label{eq:schr_lo0}
                  v^- \in \Sc(\R_-)\quad\mbox{and}\quad
                  v^+  \in \Sc(\R_+).
              \end{equation}

        \item The potential $\Vu_0=(\Vu_0^-,\Vu_0^+)$ is the first nonzero term in the left and right Taylor expansions of $V$ at $\xi=0$. In any of the cases \textsc{(a)} and \textsc{(b)}, $\Vu_0^+={\no^2_0}-1>0$, whereas
              \[
                  \Vu_0^-(\xi) = \begin{cases}
                      -2\kb\,\xi             & \mbox{in case \textsc{(a)}}  \\
                      \phantom{-2}\mb\,\xi^2 & \mbox{in case \textsc{(b)}}.
                  \end{cases}
              \]
              In order to cover both cases \textsc{(a)} and \textsc{(b)} in a unified way, we will assume more generally that
              \begin{equation}
                  \label{eq:varkappa}
                  \Vu_0^-(\xi) = \gamma |\xi|^\varkappa,\quad \xi<0,\quad\mbox{with}\quad \gamma>0,\ \varkappa>0.
              \end{equation}
    \end{enumerate}
\end{subequations}

The system \eqref{eq:schr0}--\eqref{eq:schr_lo0} can be solved by a formal series expansion according to the following procedure:
\begin{enumerate}
    \item[(i)] Scale the variable $\xi$ differently on the left and on the right of the origin, introducing
          \begin{equation}
              \label{eq:sr}
              \sigma = \xi/h^\alpha \mbox{ for }\xi<0 \quad\mbox{and}\quad
              \rho = \xi/h^{\alpha'} \mbox{ for }\xi>0
          \end{equation}
          with $\alpha$ and $\alpha'>0$ chosen in order to homogenize the operators $-h^2\Lcu_0^- + \Vu_0^-$ and $-h^2\Lcu_0^+ + \Vu_0^+$. We find that $-h^2\Lcu_0^- + \Vu_0^-$ becomes $-h^{2-2\alpha}\partial^2_\sigma + \gamma h^{\varkappa\alpha}|\sigma|^\varkappa$, and $-h^2\Lcu_0^+ + \Vu_0^+$ becomes $-h^{2-2\alpha'}\partial^2_\rho + n_0^2-1$, implying to choose
          \begin{equation}
              \label{eq:alpha}
              \alpha = \frac{2}{2+\varkappa} \quad\mbox{and}\quad \alpha' = 1.
          \end{equation}
    \item[(ii)] Expand the new functions
          \begin{equation}
              \label{eq:ansatz_sol}
              \vp(\sigma) \coloneqq v(\xi) \mbox{ for } \xi<0 \quad\mbox{and}\quad
              \psi(\rho) \coloneqq v(\xi) \mbox{ for } \xi>0,
          \end{equation}
          and $\Lu$ in series of type $\sum_{q\in\N} a_q h^{q\beta}$ for some suitable $\beta>0$. The jump condition~\eqref{eq:jumps_loc}
          being transformed into the following matching condition at $\sigma = \rho = 0$
          \begin{equation}
              \label{eq:jumpsr}
              \vp(0) = \psi(0) \quad\text{and}\quad
              \no_0^{p-1}\, h^{-\alpha}\partial_\sigma\vp(0) = h^{-\alpha'}\partial_\rho\psi(0) ,
          \end{equation}
          we find that $\alpha$, $\alpha'$, and $\alpha'-\alpha$ should be integer multiples of $\beta$.
\end{enumerate}

\subsubsection{Back to the full Schr\"odinger modal equation}
Now, we take advantage of the choices made in (i)--(ii) to treat the system \eqref{eq:schr_loc}--\eqref{eq:schr_lo} in its general form.
Hence we know that $\alpha'=1$ and leave $\alpha$ in equations for further determination.
By the change of variables~\eqref{eq:sr}--\eqref{eq:alpha} and the change of functions \eqref{eq:ansatz_sol}, the equation \eqref{eq:schr_loc} is transformed into the following two equations set on each side of the interface $\sigma = \rho = 0$
\begin{subequations}
    \begin{align}\label{eq:schr_LamTil}
        \left\{\begin{array}{rcl@{\qquad}l}
            h^{2-2\alpha}(-\Lc_h^-\vp + V_h^-\vp) & = & \wt{ E}\vp,  & \sigma\in (-\infty,0) \\[0.5ex]
            -\Lc_h^+\psi + V_h^+\psi\phantom{)}   & = & \wt{ E}\psi, & \rho\in (0,+\infty)
        \end{array}\right.
    \end{align}
    with the matching condition
    \begin{equation}
        \label{eq:jumpsra}
        \vp(0) = \psi(0) \quad\text{and}\quad
        \no_0^{p-1}\, h^{1-\alpha}\partial_\sigma\vp(0) = \partial_\rho\psi(0) ,
    \end{equation}
    the decay condition
    \begin{equation}
        \label{eq:lora}
        \vp \in \Sc(\R_-)\quad\mbox{and}\quad
        \psi  \in \Sc(\R_+),
    \end{equation}
\end{subequations}
and where the operators $\Lc_h^-$ and $\Lc_h^+$ are defined by
\begin{equation}
    \label{eq:Lh}
    \begin{aligned}
        \Lc_h^- & = \frac{\no_0^2}{\nx(h^\alpha\sigma)^2}\partial_\sigma^2
        + h^\alpha \no_0^2 \lp\frac{1}{(1+h^\alpha\sigma)\,\nx(h^\alpha\sigma)^2}
        +(p-1)\frac{\nx'(h^\alpha\sigma)}{\nx(h^\alpha\sigma)^3}\rp\partial_\sigma, \\
        \Lc_h^+ & = \no_0^2 \lp\partial_\rho^2 + \frac{h}{1+h\rho}\partial_\rho\rp,
    \end{aligned}
\end{equation}
and the potentials $V_h^-$ and $V_h^+$ are given by
\begin{equation}
    \label{eq:Vh}
    \begin{aligned}
        V_h^-(\sigma) & = h^{2\alpha-2}\lp\lp\frac{\no_0}{(1+h^\alpha\sigma)\,\nx(h^\alpha\sigma)}\rp^2 - 1\rp , \\
        V_h^+(\rho)   & = \frac{\no_0^2}{(1+h\rho)^2} - 1.
    \end{aligned}
\end{equation}

\subsubsection{Formal series of operators} The next step is to associate a formal series of operators to the system \eqref{eq:schr_LamTil}--\eqref{eq:lora}, using a Taylor expansion at $\sigma=\rho=0$ of their coefficients: For any smooth coefficient $f$, this association reads
\[
    f(h^\alpha\sigma) \sim \sum_{\ell \in\N} h^{\alpha\ell}\; \frac{f^{(\ell)}(0)}{\ell!}\, \sigma^\ell
    \quad\mbox{and}\quad
    f(h\rho) \sim \sum_{\ell \in\N} h^{\ell}\; \frac{f^{(\ell)}(0)}{\ell!}\, \rho^\ell.
\]
This defines a formal series of operators in terms of powers of $h^\beta$
\begin{equation}
    \label{eq:Apm}
    -\Lc_h^\pm + V_h^\pm \sim \sum_{q\in\N} h^{q\beta}\,\bA^\pm_q
\end{equation}
inviting to look for $\vp$, $\psi$, and $\wt{ E}$ in the form of the formal series
\begin{equation}\label{eq:1223}
    \vp(\sigma) = \sum_{q\in\N} h^{q\beta}\,\vp_q(\sigma), \quad
    \psi(\rho) = \sum_{q\in\N} h^{q\beta}\,\psi_q(\rho), \quad
    \text{and} \quad
    \wt{ E} = \sum_{q\in\N} h^{q\beta}\,\wt{ E}_q.
\end{equation}
Note that in the general framework \eqref{eq:varkappa} we have
\begin{equation}
    \label{eq:A0}
    \bA^-_0 = -\partial^2_\sigma + \gamma|\sigma|^\varkappa \quad\mbox{and}\quad
    \bA^+_0 = -\no_0^2\partial^2_\rho + \no_0^2 - 1 .
\end{equation}
Finally, since we want to construct quasi-modes with a whispering gallery structure, we give priority in \eqref{eq:schr_LamTil} to the equation in $\R_-$, which means that we look for an expansion of $\wt{ E}$ starting as $h^{2-2\alpha}$. This motivates the introduction of
\begin{equation}
    \label{eq:lam}
    \lambda = h^{2\alpha-2}\wt{ E}\sim \sum_{q\in\N} h^{q\beta}\lambda_q\,,
\end{equation}
so that equations \eqref{eq:schr_LamTil} read
\begin{align}\label{eq:schr_lam}
    \left\{\begin{array}{rcll}
        -\Lc_h^-\vp + V_h^-\vp   & = & \lambda\vp, \quad               & \sigma\in (-\infty,0) \\[0.5ex]
        -\Lc_h^+\psi + V_h^+\psi & = & h^{2-2\alpha}\lambda\psi, \quad & \rho\in (0,+\infty)
    \end{array}\right.
\end{align}
still coupled with the matching condition \eqref{eq:jumpsra} and the decay condition \eqref{eq:lora}.

%= = = = = = = = = = = = = = = = = = = = = = = = = = = = = 
\subsection{Proof: Specifics in case \textsc{(a)}}\label{sec:43}
%= = = = = = = = = = = = = = = = = = = = = = = = = = = = = 
In case \textsc{(a)}, $\kb$ is positive and the above general framework applies with the quantities
\[
    \varkappa = 1,\quad \gamma=2\kb,\quad \alpha = \tfrac{2}{3},\quad \alpha'=1,\quad \beta=\tfrac{1}{3}.
\]
This case is very close to the ``toy model'' considered in \cite[Sec.\ III]{DauRay12}.
From expressions \eqref{eq:Lh}--\eqref{eq:Apm}, we find that the first terms of the operator series $\bA_q^\pm$ are as follows
\begin{empheq}[left={\hspace*{-3ex}\empheqlbrace}]{align}
    \label{eq:A012-}
    \bA^-_0 &= -\partial^2_\sigma + 2\kb |\sigma|, \quad
    \bA^-_1 = 0, \quad
    \bA^-_2 = 2\tfrac{\nx_1}{n_0}\sigma \, \partial^2_\sigma
    - \lp 1 + (p-1)\tfrac{\nx_1}{n_0}\rp \,\partial_\sigma + c^-_2\sigma^2,\\
    \label{eq:A012+}
    \bA^+_0 &= -\no_0^2\partial^2_\rho + \no_0^2 - 1, \quad
    \bA^+_1 = 0, \quad
    \bA^+_2 = 0,\ \quad
    \bA^+_3 = -\no^2_0(\partial_\rho + 2\rho),
\end{empheq}
where $c^-_2 = 3+4\frac{\nx_1}{n_0}+3\frac{\nx_1^2}{n_0^2}-\frac{\nx_2}{n_0}$.
For a comprehensive description of the general terms $\bA^\pm_q$ we need the introduction of polynomial spaces.
\begin{notation}
    For $q\in\N$, let $\P^q$ denote the space of polynomials in one variable with degree $\le q$ and $\P^q_{\!*}$ the subspace of $\P^q$ formed by polynomials $P$ such that $P(0)=0$.
\end{notation}

A Taylor expansion at $\sigma=\rho=0$ of the coefficients of $\Lc_h^\pm$ and of $V_h^\pm$ allows to prove that
\begin{lemma}
    \label{lem:Aqpm}
    For any integer $q\ge1$, there holds
    \begin{align*}
        \bA^-_q & = A^-_q(\sigma) \,\partial^2_\sigma + B^-_q(\sigma)\,\partial_\sigma + C^-_q(\sigma)
                &                                                                                      & \text{with}\ \ A^-_q \in \P^{[\frac{q}{2}]},\ \ B^-_q \in \P^{[\frac{q}{2}]-1},\ \ C^-_q \in \P^{[\frac{q}{2}]+1} \\
        \bA^+_q & = B^+_q(\rho)\,\partial_\rho + C^+_q(\rho)
                &                                                                                      & \text{with}\ \ B^+_q \in \P^{[\frac{q}{3}]-1},\ \ C^+_q \in \P^{[\frac{q}{3}]}.
    \end{align*}
\end{lemma}

By the identifications \eqref{eq:Apm}--\eqref{eq:1223}, the system \eqref{eq:schr_lam} with jump conditions \eqref{eq:jumpsra} is associated with the formal series system of equations
\begin{equation}
    \label{eq:formal}
    \left\{\begin{array}{rcll}
        \lp \sum_{\ell\in\N} \ee h^{\frac{\ell}{3}} \ee\bA^-_\ell\rp
        \lp \sum_{\ell\in\N} \ee h^{\frac{\ell}{3}} \ee\vp_\ell\rp
                                                                      & \!=\!              &
        \lp \sum_{\ell\in\N} \ee h^{\frac{\ell}{3}} \ee\lambda_\ell\rp
        \lp \sum_{\ell\in\N} \ee h^{\frac{\ell}{3}} \ee\vp_\ell\rp \  & \mbox{in}\ \ \R_-
        \\[1ex]
        \lp \sum_{\ell\in\N} \ee h^{\frac{\ell}{3}} \ee\bA^+_\ell\rp
        \lp \sum_{\ell\in\N} \ee h^{\frac{\ell}{3}} \ee\psi_\ell\rp
                                                                      & \!=\!              &
        h^{\frac{2}{3}}
        \lp \sum_{\ell\in\N} \ee h^{\frac{\ell}{3}} \ee\lambda_\ell\rp
        \lp \sum_{\ell\in\N} \ee h^{\frac{\ell}{3}} \ee\psi_\ell\rp   & \mbox{in}\ \ \R_+
        \\[1ex]
        \lp \sum_{\ell\in\N} \ee h^{\frac{\ell}{3}} \ee\vp_\ell\rp
                                                                      & \!=\!              &
        \lp \sum_{\ell\in\N} \ee h^{\frac{\ell}{3}} \ee\psi_\ell\rp   & \mbox{at}\ \ \{0\}
        \\[1ex]
        \no_0^{p-1}\, h^{\frac{1}{3}} \lp \sum_{\ell\in\N} \ee h^{\frac{\ell}{3}} \ee\vp'_\ell\rp
                                                                      & \!=\!              &
        \lp \sum_{\ell\in\N} \ee h^{\frac{\ell}{3}} \ee\psi'_\ell\rp  & \mbox{at}\ \ \{0\}
    \end{array}\right.
\end{equation}
This system is equivalent to an infinite collection of systems obtained by equating the series coefficients: Namely, for $q$ spanning $\N$,
\begin{equation}
    \label{eq:formalq}
    \left\{\begin{array}{rcll}
        \sum_{\ell=0}^q \bA^-_\ell\,  \vp_{q-\ell}
                                                            & \!=\!             &
        \sum_{\ell=0}^q \lambda_\ell\,  \vp_{q-\ell}   \    & \mbox{in}\ \ \R_-
        \\[0.5ex]
        \sum_{\ell=0}^q \bA^+_\ell\,  \psi_{q-\ell}
                                                            & \!=\!             &
        \sum_{\ell=2}^{q} \lambda_{\ell-2}\,  \psi_{q-\ell} & \mbox{in}\ \ \R_+
        \\
        \vp_q(0)
                                                            & \!=\!             &
        \psi_q(0)
        \\
        \psi'_q(0)
                                                            & \!=\!             &
        \no_0^{p-1}\, \vp'_{q-1}(0)
    \end{array}\right.
\end{equation}
where we agree that the right hand side of the second line is $0$ when $q=0$ or $q=1$, and, likewise, the right hand side of the fourth line is $0$ when $q=0$.

\subsubsection{Initialization stage}
For $q=0$, the system \eqref{eq:formalq} reads
\begin{equation}
    \label{eq:formal0}
    \left\{\begin{array}{rcll}
        \bA^-_0\,  \vp_{0}
                               & \!=\!             &
        \lambda_0 \vp_{0}   \  & \mbox{in}\ \ \R_-
        \\
        \bA^+_0\,  \psi_{0}
                               & \!=\!             &
        0                      & \mbox{in}\ \ \R_+
        \\
        \vp_0(0)
                               & \!=\!             &
        \psi_0(0)
        \\
        \psi'_0(0)
                               & \!=\!             &
        0
    \end{array}\right.
\end{equation}
for which we look for solutions $\vp_0\in\Sc(\R_-)$ and $\psi_0\in\Sc(\R_+)$.

Since the equation $\bA^+_0\, \psi_{0} = 0$ with the Neumann condition at $0$ has no non-zero solution in $\Sc(\R_+)$, it is natural to take $\psi_0=0$ in \eqref{eq:formal0}. Then we are left with the following Airy eigen-problem on $\R_-$ for $\vp_0$
\[
    -\vp_0''(\sigma) - 2\kb \,\sigma\vp_0(\sigma) = \lambda_0\vp_0(\sigma) \ \mbox{ for }\ \sigma\in(-\infty,0),
    \qquad\mbox{and}\qquad \vp_0(0)=0
\]
whose decaying solutions can be expressed in terms of the mirror Airy function $\As$. Recall that $\as_j$ for $j\in\N$, denote the successive roots of $\As$. We obtain immediately:

\begin{lemma}\label{lem:lambda0_A}
    Let $j\in\N$. The couple of functions $(\vp_0,\psi_0)$ and the number $\lambda_0$ defined by
    \[
        \vp_0(\sigma) = \As\big(\as_j+(2\kb)^{\frac{1}{3}}\sigma\big), \quad
        \psi_0(\rho) = 0,\quad\mbox{and}\quad
        \lambda_0 = \as_j (2\kb)^{\frac{2}{3}}
    \]
    solve \eqref{eq:formal0} in $\Sc(\R_-)\times\Sc(\R_+)$.
\end{lemma}

\begin{remark}
    \label{rem:expdec}
    The quasi-mode construction requires a cut-off at infinity at some stage. Such a cut-off will be harmless to the satisfied equations if the functions $(\vp_0,\psi_0)$, and more generally $(\vp_q,\psi_q)$, are exponentially decreasing when $\sigma\to-\infty$ and $\rho\to+\infty$. It is easy to see that any such solution of \eqref{eq:formal0} is proportional to one of the solutions given in Lemma \ref{lem:lambda0_A}.
\end{remark}

\subsubsection{Sequence of nested problems and recurrence}

Reordering the terms in the system~\eqref{eq:formalq} of rank $q$, we can write it in the following form
\begin{subequations}    \label{eq:sysQ_A}
    \begin{empheq}[left={ \hspace*{-10mm}(\mathcal{R}^{\text{\textsc{(a)}}}_q)\quad     \empheqlbrace}]{align}
        \hspace*{6mm}   -\vp_q''(\sigma) - (2\kb\sigma +\lambda_0)\vp_q(\sigma) &= \lambda_q\,\vp_0(\sigma) + S_q^\vp(\sigma)
        & \sigma \in \R_-
        \label{eq:sysQ_Aa} \\
        \hspace*{6mm}   -\no_0^{2}\psi_q''(\rho) + \lp \no_0^{2}-1\rp\psi_q(\rho) &= S_q^\psi(\rho) & \rho \in \R_+
        \label{eq:sysQ_Ab} \\
        \vp_q(0) &= \psi_q(0) &
        \label{eq:sysQ_Ac} \\
        \psi_q'(0) &= \no_0^{p-1}\vp_{q-1}'(0) &
        \label{eq:sysQ_Ad}
    \end{empheq}
\end{subequations}
with right hand terms $S_q^\vp$ and $S_q^\psi$ defined as (recall that $\bA^+_1=0$)
\begin{equation}
    \label{eq:Sq_A}
    S_q^\vp = - \bA^-_q \,\vp_{0} + \sum_{\ell=1}^{q-1} (\lambda_\ell - \bA^-_\ell) \,  \vp_{q-\ell}
    \quad\mbox{and}\quad
    S_q^\psi = \sum_{\ell=2}^{q} (\lambda_{\ell-2} - \bA^+_\ell) \, \psi_{q-\ell} .
\end{equation}

\begin{proposition}\label{pro:solQ_A}
    Choose $j\in\N$ and define $(\vp_0,\psi_0,\lambda_0)$ according to Lemma \ref{lem:lambda0_A}.
    Then there exist, for any $q\ge1$,
    \begin{itemize}
        \item a unique $\lambda_q\in\R$
        \item unique polynomials $P_q^\vp\in\P^q_{\!*}$, $Q_q^\vp\in\P^{q-1}$, and $P_q^\psi\in\P^{q-1}$
    \end{itemize}
    such that setting
    \begin{equation}
        \label{eq:PQ_A}
        \begin{aligned}
            \vp_q(\sigma) & = P_q^\vp(\sigma)\As\big(\as_j+(2\kb)^{\frac{1}{3}}\sigma\big)
            + Q_q^\vp(\sigma)\As'\big(\as_j+(2\kb)^{\frac{1}{3}}\sigma\big)
                          & \forall\sigma\in\R_-                                           \\[0.5ex]
            \psi_q(\rho)  & = \te P_q^\psi(\rho) \exp\big(-\rho\sqrt{1-\no_0^{-2}}\,\big)
                          & \forall\rho\in\R_+\end{aligned}
    \end{equation}
    the collection $(\vp_0,\ldots,\vp_q,\psi_0,\ldots,\psi_q,\lambda_0,\ldots,\lambda_q)$ solves the sequence of problems $(\mathcal{R}^{\text{\textsc{(a)}}}_\ell)$ introduced in \eqref{eq:sysQ_A} for $\ell = 0,\ldots,q$.
\end{proposition}

\begin{proof}
    We proceed by induction on $q$. For $q=0$,  lemma \ref{lem:lambda0_A} provides $\lambda_0$, $\vp_0$, and $\psi_0$  solutions to  $(\mathcal{R}^{\text{\textsc{(a)}}}_0)$ and we readily obtain the polynomials $P_0^\vp = 1$, $Q_0^\vp = 0$, and $P_0^\psi = 0$.
    Let $q\ge 1$ and suppose that $(\lambda_\ell)_{0\le\ell\le q-1}$, $(\vp_\ell)_{0\le\ell\le q-1}$, and $(\psi_\ell)_{0\le\ell\le q-1}$  are solutions to problems $(\mathcal{R}^{\text{\textsc{(a)}}}_\ell)$ for $\ell=0,\ldots,q-1$, and satisfy \eqref{eq:PQ_A}.

    Using the expression \eqref{eq:Sq_A} of $S_q^\psi$ combined with Lemma \ref{lem:Aqpm}, we deduce from the induction assumption that
    there exists a polynomial $E_q^\psi\in\P^{q-2}$ such that
    \begin{align*}
        S_q^\psi(\rho) = E_q^\psi(\rho) \te \exp\big(-\rho\sqrt{1-\no_0^{-2}}\,\big) .
    \end{align*}
    From  Lemma \ref{lem:exp}  in Appendix, there exists a unique polynomial $\wt{P}_q^\psi\in\P^{q-1}_{\!*}$
    such that the function $\wt{\psi}$ defined by  $\wt{\psi}_q(\rho) = \wt{P}_q^\psi(\rho) \exp(-\rho\sqrt{1-\no_0^{-2}}\ee)$
    is solution to \eqref{eq:sysQ_Ab}. It follows that the sought function $\psi_q$ is given by
    \begin{align*}
        \psi_q(\rho) = \big( a_0 + \wt{P}_q^\psi(\rho)\big) \te\exp\big( -\rho\sqrt{1-\no_0^{-2}}\,\big)
    \end{align*}
    where the constant $a_0$ is determined from Neumann condition \eqref{eq:sysQ_Ad}. This defines the polynomial $P_q^\psi$ as $a_0 + \wt{P}_q^\psi$ and hence $P_q^\psi\in \P^{q-1}$ as desired.

    Let us now consider equation \eqref{eq:sysQ_Aa}. Using Lemma \ref{lem:Aqpm} combined with the relation $\As''(z) = -z\As(z)$, we deduce from the expression \eqref{eq:Sq_A} of $S_q^\vp$
    and the induction assumption that there exist polynomials $R_q^\psi\in\P^{q}$ and $T_{q}^\vp\in\P^{q-1}$ such that
    \begin{align*}
        S_q^\vp(\sigma) = R_{q}^\vp(\sigma) \As\big(\as_j+(2\kb)^{\frac{1}{3}}\sigma\big)
        + T_{q}^\vp(\sigma)\As'\big(\as_j + (2\kb)^{\frac{1}{3}}\sigma\big) .
    \end{align*}
    From Lemma \ref{lem:airy} there exist unique polynomials $P_q^\vp\in\P^{q}_{\!*}$ and $\wt{Q}_q^\vp\in\P^{q-1}$
    such that the function given by
    \[
        \wt{\vp}_q(\sigma) = P_q^\vp(\sigma) \As\big(\as_j+(2\kb)^{\frac{1}{3}}\sigma\big)
        + \wt{Q}_q^\vp(\sigma) \As'\big(\as_j+(2\kb)^{\frac{1}{3}}\sigma\big)
    \]
    is a solution to the ODE (compare it with \eqref{eq:sysQ_Aa})
    \[
        -\wt{\vp}_q''(\sigma) - (2\kb\sigma +\lambda_0)\wt{\vp}_q(\sigma) = S_q^\vp(\sigma) \,.
    \]
    If we define $\vp_q$ as follows:
    \begin{align*}
        \vp_q(\sigma) =
        \frac{\lambda_q}{(2\kb)^{\frac{2}{3}}}\As'\big( \as_j+(2\kb)^{\frac{1}{3}}\sigma\big) + \wt{\vp}_q(\sigma)
        \quad\mbox{with}\quad
        \lambda_q = \frac{(2\kb)^{\frac{2}{3}}}{\As'(\as_j)} \lp\psi_q(0) - \wt{\vp}_q(0)\rp ,
    \end{align*}
    then $\vp_q$  solves  \eqref{eq:sysQ_Aa} and the continuity condition \eqref{eq:sysQ_Ac}. (Note that we have $\As'(\as_j) \neq 0$  for all $j\in\N$, see \cite[Sect.\ 9.9(ii)]{Nist}.)
    Finally, we set $Q_q^\vp = {\lambda_q}{(2\kb)^{-\frac{2}{3}}} + \wt{Q}_q^\vp \in\P^{q-1}$.
\end{proof}
%--------------------------------------------------------------------

Calculating for $q=1$ in Proposition \ref{pro:solQ_A} provides $P^\vp_1=0$ and explicit values for $Q^\vp_1$, $P^\psi_1$ and $\lambda_1$, from which we deduce
\begin{lemma}
    We have $\lambda_1 = \dfrac{-n_0^p\, 2\kb}{\sqrt{n_0^2-1}}$ and, for $(\sigma,\rho) \in \R_- \times \R_+$,
    \begin{align*}
        \vp_1(\sigma) = -\frac{n_0^p\, (2\kb)^{\frac{1}{3}}}{\sqrt{n_0^2-1}}\, \As'(\as_j+(2\kb)^{\frac{1}{3}}\sigma), &  &
        \psi_1(\rho) = -\frac{n_0^p\, (2\kb)^{\frac{1}{3}}\, \As'(\as_j)}{\sqrt{n_0^2-1}}\, \exp\lp -\rho\sqrt{1-n_0^{-2}}\rp.
    \end{align*}
\end{lemma}

\begin{remark}\label{rem:Algo_A}
    From the proof of Proposition \ref{pro:solQ_A}, we see that, for each chosen $j$ and $q\ge1$, the four operators
    \[
        \begin{aligned}
            \{\lambda_0,\ldots,\lambda_{q-1},P^\psi_0,\ldots,P^\psi_{q-1}\} & \longmapsto E^\psi_q                         \\
            \{E^\psi_q,P^\vp_{q-1},Q^\vp_{q-1}\}                            & \longmapsto P^\psi_q                         \\
            \{\lambda_0,\ldots,\lambda_{q-1},P^\vp_0,\ldots,P^\vp_{q-1},Q^\vp_0,\ldots,Q^\vp_{q-1}\}
                                                                            & \longmapsto \{R^\vp_q, T^\vp_q\}             \\
            \{R^\vp_q,T^\vp_q,P^\psi_q\}                                    & \longmapsto \{ P^\vp_q, Q^\vp_q,\lambda_q \}
        \end{aligned}
    \]
    act between finite dimensional spaces and can be identified to matrices.
    They result into an algorithm that can be derived and implemented in a computer algebra system to obtain the expression of $\lambda_q, \vp_q, \psi_q$ for $q\ge 2$, see \cite[Annexe D]{MoitierPhD}.
    The coefficients of the polynomials $P_q^\vp, Q_q^\vp, P_q^\vp$ are \emph{rational functions} of the quantities $(2\kb)^{\frac{1}{3}}$, $\sqrt{n_0^2-1}$, $\as_j$, $\As'(\as_j)$, $n_0^p$, and $n_\ell$ for all $\ell \in \{0,\ldots,q\}$.
\end{remark}

\subsubsection{Convergence}\label{sec:ccls_A}
Choose $p\in\{\pm1\}$ and a natural integer $j$.
In a last stage, we have to prove that the formal series
\begin{equation}
    \label{eq:formals}
    \sum_{q\in\N} \lambda_q h^{\frac{q}{3}},\qquad
    \sum_{q\in\N} \vp_q h^{\frac{q}{3}},  \qquad\mbox{and}\qquad
    \sum_{q\in\N} \psi_q h^{\frac{q}{3}},
\end{equation}
obtained from Proposition \ref{pro:solQ_A} give rise to a family of resonance quasi-pairs in the sense of Definition \ref{def:quasi}.
Note that, by construction, the functions $\vp_q$ and $\psi_q$ are exponentially decreasing at infinity, thus belong to  $\Sc(\R_-)$ and $\Sc(\R_+)$, respectively.
Relying on Borel's theorem \cite[Thm.\ 1.2.6]{HorI} and its variant given in Lemma \ref{th:Borel} in Appendix, we obtain the existence of smooth functions having $(\lambda_q)_q$, $(\vp_q)_q$ and $(\psi_q)_q$ as Taylor terms at $0$.
Combined with Lemma \ref{lem:Taylor}, this yields the following results for the remainders of truncated series expansions of formal series \eqref{eq:formals}.

\begin{lemma}\label{lem:Borel_A}
    Let $(\lambda_q)_{q\in\N}$, $(\vp_q)_{q\in\N}$ and $(\psi_q)_{q\in\N}$ given by Proposition \ref{pro:solQ_A}.
    There exist smooth functions $\lu_{\pj}\in\Cc^\infty([0,1])$, $\Phi_{\pj}\in\Cc^\infty([0,1],\Sc(\R_-))$ and
    $\Psi_{\pj}\in\Cc^\infty([0,1],\Sc(\R_+))$ such that for all $(h,\sigma,\rho)\in [0,1]\times\R_-\times\R_+$ and for all integer $N\ge0$, we have the following finite expansions with remainders
    \begin{subequations}
        \begin{align}
            \label{eq:Bora}
            \lu_{\pj}(h^{\frac{1}{3}})         & =
            \sum_{q=0}^{N-1} h^{\frac{q}{3}}\, \lambda_q + h^{\frac{N}{3}} R_N^\lambda(h^{\frac{1}{3}}),
                                               &   & \mbox{with}\quad R_N^\lambda\in\Cc^\infty([0,1])        \\
            \label{eq:Borb}
            \Phi_{\pj}(h^{\frac{1}{3}};\sigma) & =
            \sum_{q=0}^{N-1} h^{\frac{q}{3}}\, \vp_q(\sigma) + h^{\frac{N}{3}} R_N^\vp(h^{\frac{1}{3}};\sigma),
                                               &   & \mbox{with}\quad R_N^\vp\in\Cc^\infty([0,1],\Sc(\R_-))  \\
            \label{eq:Borc}
            \Psi_{\pj}(h^{\frac{1}{3}};\rho)   & =
            \sum_{q=0}^{N-1} h^{\frac{q}{3}}\, \psi_q(\rho) + h^{\frac{N}{3}} R_N^\psi(h^{\frac{1}{3}};\rho)
                                               &   & \mbox{with}\quad R_N^\psi\in\Cc^\infty([0,1],\Sc(\R_+))
        \end{align}
    \end{subequations}
\end{lemma}

Note that the remainders at rank $N=0$ simply coincide with the original function.
%--------------------------------------------------------------------

\begin{definition}\label{def:uu_A}
    Choose a real number $\delta\in (0,\frac{1}{2})$ and a smooth cut-off function $\chi$, $0\le \chi\le1$, such that $\chi(\xi) = 1$ for $|\xi| \le \delta$ and $\chi(\xi) = 0$ for $|\xi| \ge 2\delta$.
    We define for any integer $m \ge 1$ with the notation $h=m^{-1}$, the quantities:
    \begin{align*}
        \ku_{\pj}(m)          & = \frac{m}{R\no_0}
        \sqrt{1+h^{\frac{2}{3}}\ \lu_{\pj}(h^{\frac{1}{3}})},                                                                                                  \\
        \vu_{\pj}(m;\xi)      & = \chi(\xi)\left\lbrace\begin{array}{ll}
            \Phi_{\pj}(h^{\frac{1}{3}};h^{-\frac{2}{3}}\xi) , & \xi \le 0 \\
            \Psi_{\pj}(h^{\frac{1}{3}};h^{-1}\xi),            & \xi > 0
        \end{array}\right. &                                              & \xi\in (-1,+\infty) \\
        \uu_{\pj}(m;r,\theta) & = \vu_{\pj}\lp m;\tfrac{r}{R}-1\rp\, \e^{\ic m \theta}
                              &                                                           & (r,\theta) \in (0,+\infty)\times\R / 2\pi\Z.
    \end{align*}
\end{definition}

We now show that the sequence $(\ku_{\pj}(m),\uu_{\pj}(m))_{m\ge 1}$  is a family of ``almost'' quasi-pairs in the sense of the following lemma. A further correction will have to be made to transform this family into a true family of resonance quasi-pairs in the sense of Definition \ref{def:quasi}.

\begin{lemma}\label{lem:QM_A}
    The sequence $(\ku_{\pj}(m),\uu_{\pj}(m))_{m\ge 1}$ defined above has the following properties:
    \begin{itemize}
        \item[(i)] For all $m$, the function $\uu_{\pj}(m)$ is supported in an annulus around the interface $r=R$
              \begin{align*}
                  \supp(\uu_{\pj}(m)) \subset B(0,R(1+2\delta)) \setminus B(0,R(1-2\delta)) .
              \end{align*}
        \item[(ii)] For all $m$, the function $\uu_{\pj}(m)$ is piece-wise smooth up to the interface $r=R$:
              \begin{align*}
                  \uu_{\pj}(m)\on{\wb{\Omega}} \in \Cc^\infty\lp\wb{\Omega}\rp
                  \quad\mbox{and}\quad
                  \uu_{\pj}(m)\on{\R^2\setminus\Omega} \in \Cc^\infty\lp\R^2\setminus\Omega\rp .
              \end{align*}
        \item[(iii)] We have the following estimates for the jumps across the interface when $m \to +\infty$
              \begin{align*}
                  \lc \uu_{\pj}(m) \rc_{\partial\Omega} = \Oo\lp m^{-\infty}\rp
                  \quad\text{and}\quad
                  \lc n^{p-1}\, \partial_\nu \uu_{\pj}(m) \rc_{\partial\Omega} = \Oo\lp m^{-\infty}\rp .
              \end{align*}
        \item[(iv)] Defining the residuals
              \begin{equation}\label{eq:1723}
                  \resid(m) \coloneqq \div\lp n^{p-1}\nabla \uu_{\pj}(m)\rp + \ku_{\pj}^2(m)\, n^{p+1}\, \uu_{\pj}(m)
              \end{equation}
              we have the following estimates in $\Omega$ and $\R^2\setminus\overline\Omega$ when $m \to +\infty$
              \begin{align*}
                  \frac{\lo\resid(m) \ro_{\Lr^2(\Omega)}
                      + \lo\resid(m) \ro_{\Lr^2(\R^2\setminus\overline\Omega)} }
                  {\lo\uu_{\pj}(m)\ro_{\Lr^2(\R^2)}}
                  = \Oo\lp m^{-\infty}\rp .
              \end{align*}
    \end{itemize}
\end{lemma}

\begin{proof}
    (i) and (ii) are obvious consequences of the definition of $\uu_{\pj}(m)$.
    \medskip\par\noindent
    (iii) From Definition \ref{def:uu_A}, we have, for all $\theta\in\R / 2\pi\Z$, and with $h=\frac{1}{m}$
    \begin{align*}
        \lc \uu_{\pj}(m)\rc_{\partial\Omega}(\theta) = \lc \vu_{\pj}(m;\xi)\rc_{\{\xi=0\}}\e^{\ic m\theta}
        = \lp\Psi_{\pj}(h^{\frac{1}{3}};0) - \Phi_{\pj}(h^{\frac{1}{3}};0)\rp\e^{\ic m\theta}.
    \end{align*}
    Let $N\ge1$. From \eqref{eq:Borb}--\eqref{eq:Borc}, we deduce that
    \begin{equation}\label{eq:1612}
        \Psi_{\pj}(h^{\frac{1}{3}};0) - \Phi_{\pj}(h^{\frac{1}{3}};0) =
        \sum_{q=0}^{N-1} \big( \psi_q(0) - \vp_q(0) \big) h^{\frac{q}{3}} +
        h^{\frac{N}{3}} \lp R_N^\psi(h^{\frac{1}{3}};0) - R_N^\vp(h^{\frac{1}{3}};0)\rp .
    \end{equation}
    Since by construction $\psi_q(0) - \vp_q(0) = 0$ for any $q\in\N$, {\em cf.} \eqref{eq:sysQ_A}, we deduce from \eqref{eq:1612} that
    $\Psi_{\pj}(h^{\frac{1}{3}};0) - \Phi_{\pj}(h^{\frac{1}{3}};0)= \Oc(h^{\frac{N}{3}})$.
    This statement is true for any $N\ge1$, hence $\lc \uu_{\pj}(m) \rc_{\partial\Omega} =\Oc(m^{-\infty})$.

    We proceed in a similar way for the second jump condition. We have
    \begin{align*}
        \lc n^{p-1} \partial_\nu \uu_{\pj}(m)\rc_{\partial\Omega}(\theta) =
        R^{-1}\lc \no^{p-1} \partial_\xi v_{\pj}(m,\xi)\rc_{\{\xi=0\}}\e^{\ic m\theta}
    \end{align*}
    and
    \begin{align*}
        \lc \no^{p-1} \partial_\xi v_{\pj}(m,\xi)\rc_{\{\xi= 0\}}
         & = h^{-1} \lp\partial_\rho \Psi_{\pj}(h^{\frac{1}{3}};0)
        - \no_0^{p-1}h^{\frac{1}{3}}\partial_\sigma \Phi_{\pj}(h^{\frac{1}{3}};0)\rp
        \\
         & = \sum_{q=1}^{N-1}
        \big( \psi'_q(0) - \no_0^{p-1}\vp'_{q-1}(0) \big) h^{\frac{q}{3}-1} \\
         & + \
        h^{\frac{N}{3}-1} \lp - \no_0^{p-1}\vp'_{N-1}(0)
        + \partial_\rho R_N^\psi(h^{\frac{1}{3}};0)
        - \no_0^{p-1} h^{\frac{1}{3}}\partial_\sigma R_N^\vp(h^{\frac{1}{3}};0) \rp.
    \end{align*}
    From the jump relation \eqref{eq:sysQ_Ad} $\psi'_q(0) - \no_0^{p-1}\vp'_{q-1}(0)$  for any $q\geq 1$,  we deduce that the above quantity is a $\Oc(h^{\frac{N}{3}-1})$ for any $N\ge1$, hence a $\Oc(m^{-\infty})$.

    \bigskip\noindent
    (iv) In order to prove the estimates on the residuals, it is enough to prove that the $\Lr^2$ norm of the residual $\resid(m)$ on $\Omega$ and on $\R^2\setminus\overline\Omega$ is $\Oc(m^{-\infty})$ and that
    \begin{align}
        \label{eq:norm_A}
        \lo\uu_{\pj}(m)\ro_{\Lr^2(\R^2)} =
        \gamma\, m^{-\frac{1}{3}} + \Oo( m^{-\frac{2}{3}})
    \end{align}
    for some positive constant $\gamma$.
    Given a parameter $t>0$, we introduce the following weighted $\Lr^2$ (semi) norm on any interval $I\subset\R$:
    \begin{align*}
        \lo w \ro_{\Lr^2[t](I)}^2 =
        \int_{I\,\cap\,(-2\delta/t\ee,\ee\infty)} |w(\tau)|^2 \,t (1+\tau t) \dd\tau.
    \end{align*}

    \medskip\noindent
    Let us first prove \eqref{eq:norm_A}.
    We readily obtain from Definition \ref{def:uu_A}, having set $L \coloneqq 2\pi R$,
    \begin{equation}
        \label{eq:L2}
        \lo\uu_{\pj}(m)\ro_{\Lr^2(\R^2)}^2 =
        L \lp \lo \chi(\cdot\ee h^{\frac{2}{3}})\Phi_{\pj}(h^{\frac{1}{3}};\cdot)
        \ro_{\Lr^2[h^{\frac{2}{3}}](\R_-)}^2 +
        \lo \chi(\cdot\ee h)\Psi_{\pj}(h^{\frac{1}{3}};\cdot) \ro_{\Lr^2[h](\R_+)}^2\rp.
    \end{equation}
    From \eqref{eq:Borb} and \eqref{eq:Borc} considered with $N=1$,
    we have
    \begin{align*}
        \Phi_{\pj}(h^{\frac{1}{3}};\sigma) & = \As(\as_j+(2\kb)^{\frac{1}{3}}\sigma)
        + h^{\frac{1}{3}} R_1^\vp(h^{\frac{1}{3}};\sigma)                                     \\
        \Psi_{\pj}(h^{\frac{1}{3}};\rho)   & = h^{\frac{1}{3}} R_1^\psi(h^{\frac{1}{3}};\rho)
    \end{align*}
    where $R_1^\vp \in \Cc^\infty([0,1],\Sc(\R_-))$ and $R_1^\psi \in \Cc^\infty([0,1],\Sc(\R_+))$. We deduce that
    \[
        \Big| \lo\uu_{\pj}(m)\ro_{\Lr^2(\R^2)} - \sqrt{L}\,
        \big\| \chi(\cdot\ee h^{\frac{2}{3}})\As(\as_j+(2\kb)^{\frac{1}{3}}\cdot) \big\|_{\Lr^2[h^{\frac{2}{3}}](\R_-)}
        \Big| \le C_1h^{\frac{1}{3}}(h^{\frac{1}{3}} + h^{\frac{1}{2}})
        \le C'_1h^{\frac{2}{3}}
    \]
    for some constants $C_1$ and $C'_1$.
    We now have to estimate the quantity
    \[
        \big\| \chi(\cdot\ee h^{\frac{2}{3}})
        \As(\as_j+(2\kb)^{\frac{1}{3}}\cdot) \big\|_{\Lr^2[h^{\frac{2}{3}}](\R_-)}^2 =
        \int_{\R_-} \lv \chi(\sigma h^{\frac{2}{3}})\As(\as_j+(2\kb)^{\frac{1}{3}}\sigma) |\rv^2\,
        h^{\frac{2}{3}} (1+\sigma h^{\frac{2}{3}}) \dd\sigma\,.
    \]
    We split the integral according to $I_1-I_2-I_3$ with the three positive integrals
    \begin{align*}
        I_1 & = h^{\frac{2}{3}} \int_{\R_-}
        \lv\As(\as_j+(2\kb)^{\frac{1}{3}}\sigma)\rv^2\,\dd\sigma
        \\
        I_2 & = h^{\frac{2}{3}} \int_{\R_-} \lp 1-\chi\big(\sigma h^{\frac{2}{3}}\big)^2\rp
        \lv\As(\as_j+(2\kb)^{\frac{1}{3}}\sigma) \rv^2\,\dd\sigma
        \\
        I_3 & = h^{\frac{4}{3}} \int_{\R_-}
        \lv \chi(\sigma h^{\frac{2}{3}})\As(\as_j+(2\kb)^{\frac{1}{3}}\sigma) \rv^2\,
        |\sigma| \dd\sigma\,.
    \end{align*}
    Since a primitive function of $\As^2$ is $x \mapsto \As'(x)^2 + x\As(x)^2$, we find that
    \[
        I_1 = h^{\frac{2}{3}} \,(2\kb)^{-\frac{1}{3}}\, \As'(\as_j)^2.
    \]
    Moreover, since $\As$ is exponentially decreasing over $\R_-$, we find that $I_3\le C_3 h^{\frac{4}{3}}$.
    Finally, Lemma \ref{lem:dec} in Appendix shows that $I_2=\Oc(h^{\infty})$.

    \medskip\noindent
    Let us now show that the $\Lr^2$ norm on $\Omega$ and on $\R^2\setminus\overline\Omega$  of the residual $\resid(m) $ defined in \eqref{eq:1723}  is $\Oc(m^{-\infty})$.
    Revisiting all variable changes and problem reformulations, we find
    \begin{subequations}
        \begin{align}
             & \big\lVert \resid(m)  \big\rVert_{\Lr^2(\Omega)}
            =
            m^2\sqrt{L}\, \lo h^{\frac{2}{3}} \lp -\Lc_h^- + V_h^- - \lu(h^{\frac{1}{3}}) \rp
            \lp \chi(\cdot\ee h^{\frac{2}{3}})\Phi_{\pj}(h^{\frac{1}{3}};\cdot) \rp \ro_{\Lr^2[h^{\frac{2}{3}}](\R_-)}  \label{eq:res-}
            \\
             & \big\lVert \resid(m) \big\rVert_{\Lr^2(\R^2\setminus\overline\Omega)}
            =
            m^2\sqrt{L}\, \lo \lp -\Lc_h^+ + V_h^+ - h^{\frac{2}{3}} \lu(h^{\frac{1}{3}}) \rp
            \lp \chi(\cdot\ee h)\Psi_{\pj}(h^{\frac{1}{3}};\cdot) \rp \ro_{\Lr^2[h](\R_+)} \label{eq:res+}
        \end{align}
    \end{subequations}
    Introducing the commutators $\lc\Lc_h^-,\chi(\cdot\ee h^{\frac{2}{3}})\rc$ and $\lc\Lc_h^+,\chi(\cdot\ee h)\rc$ of the differential operators $\Lc_h^\pm$ with scaled cut-off functions, we deduce from \eqref{eq:res-}--\eqref{eq:res+} the inequalities
    \begin{subequations}
        \begin{align}
            \label{eq:res-b}
            \big\lVert \resid(m)  \big\rVert_{\Lr^2(\Omega)}
             & \le
            \sqrt{L} \; h^{-\frac{4}{3}} (\Nc_\vp + \Nc'_\vp)
            \\
            \label{eq:res+b}
            \big\lVert \resid(m) \big\rVert_{\Lr^2(\R^2\setminus\overline\Omega)}
             & \le
            \sqrt{L}\; h^{-2}  (\Nc_\psi + \Nc'_\psi)
        \end{align}
    \end{subequations}
    where
    \begin{subequations}\label{eq:1519}
        \begin{align}
            \Nc_\vp   & = \lo\chi(h^{\frac{2}{3}}\ee\cdot) \lp -\Lc_h^-
            + V_h^- - \lu_{\pj} \rp \Phi_{\pj}(h^{\frac{1}{3}};\cdot)\ro_{\Lr^2[h^{\frac{2}{3}}](\R_-)}
            \label{eq:1519a}                                                    \\
            \Nc_\vp'  & = \lo\lc \Lc_h^-,\chi(h^{\frac{2}{3}}\ee\cdot)\rc
            \Phi_{\pj}(h^{\frac{1}{3}};\cdot)\ro_{\Lr^2[h^{\frac{2}{3}}](\R_-)} \\
            \Nc_\psi  & = \lo\chi(h\ee\cdot) \lp -\Lc_h^+
            + V_h^+ - h^{\frac{2}{3}}\lu_{\pj} \rp \Psi_{\pj}(h^{\frac{1}{3}};\cdot) \ro_{\Lr^2[h](\R_+)}
            \label{eq:1519c}                                                    \\
            \Nc_\psi' & = \lo\lc \Lc_h^+,\chi(h\ee\cdot)\rc
            \Psi_{\pj}(h^{\frac{1}{3}};\cdot)\ro_{\Lr^2[h](\R_+)}. \label{eq:1519d}
        \end{align}
    \end{subequations}
    Both operators $\chi(h^{\frac{2}{3}}\ee\cdot) \lp -\Lc_h^- + V_h^-\rp$ and $\chi(h\ee\cdot) \lp -\Lc_h^+ + V_h^+\rp$ are differential operators in the form $a^\pm_2\partial^2+a^\pm_1\partial+a^\pm_0$ with coefficients $a^\pm_i(h^{\frac{1}{3}},\cdot)$ belonging to $\Cc^\infty_{\sf bounded}([0,1]\times\R_\pm)$, see~\hbox {\eqref{eq:Lh}--\eqref{eq:Vh}.} Hence, the formal series \eqref{eq:Apm} gives rise to the following sequences of finite expansions with remainders: For any $N\ge1$,
    \begin{equation}
        \label{eq:Apmf}
        -\Lc_h^\pm + V_h^\pm = \sum_{q=0}^{N-1} h^{\frac{q}{3}}\, \bA^\pm_q
        + h^{\frac{N}{3}} \bR^\pm_N(h^{\frac{1}{3}};\cdot)
    \end{equation}
    where the remainders $\bR^\pm_N$ are differential operators of order $2$ such that $\chi(h^{\frac{2}{3}}\ee\cdot) \bR^-_N(h^{\frac{1}{3}};\cdot)$ and $\chi(h\ee\cdot) \bR^+_N(h^{\frac{1}{3}};\cdot)$ have coefficients belonging to $\Cc^\infty_{\sf bounded}([0,1]\times\R_\pm)$.
    It follows that for any given $N,N'\in \N$
    \begin{align}
         & \lp -\Lc_h^- + V_h^- - \lu_{\pj} \rp \Phi_{\pj}(h^{\frac{1}{3}};\cdot)  \nonumber \\
         & = \lp \sum_{q=0}^{N-1} h^{\frac{q}{3}}\, (\bA^-_q - \lambda_q)
        + h^{\frac{N}{3}} \lp \bR^-_N(h^{\frac{1}{3}};\cdot) - R^\lambda_N(h^{\frac{1}{3}}) \rp \rp
        \lp \sum_{q=0}^{N'-1} h^{\frac{q}{3}}\, \vp_q + h^{\frac{N'}{3}} R^\vp_{N'}(h^{\frac{1}{3}};\cdot) \rp
        \nonumber                                                                            \\
         & =
        h^{\frac{N}{3}} \lp \sum_{q=0}^{N-1} (\bA^-_q - \lambda_q) R^\vp_{N-q}(h^{\frac{1}{3}};\cdot)
        +  \lp \bR^-_N(h^{\frac{1}{3}};\cdot) - R^\lambda_N(h^{\frac{1}{3}}) \rp R^\vp_{0}(h^{\frac{1}{3}};\cdot)\rp
        \label{eq:1749}
    \end{align}
    where the second equality is obtained from the relation
    $
        \sum_{\ell=0}^q  (\bA^-_\ell - \lambda_\ell) \vp_{q-\ell} = 0
    $
    for all $q\in\N$ deduced from  \eqref{eq:formalq}.
    In a similar way, we show that
    \begin{align}
        \big( -\Lc_h^+ + V_h^+ - h^{\frac{2}{3}}\lu_{\pj} \big) \Psi_{\pj}(h^{\frac{1}{3}};\cdot)
        =
        h^{\frac{N}{3}} \Bigg( & \sum_{q=0}^{N-1} (\bA^+_q - \lambda_{q-2})
        R^\psi_{N-q}(h^{\frac{1}{3}};\cdot)
        \nonumber                                                           \\ &
        + \lp \bR^+_N(h^{\frac{1}{3}};\cdot) - R^\lambda_{N-2}(h^{\frac{1}{3}}) \rp R^\psi_{0}(h^{\frac{1}{3}};\cdot)\Bigg)
        \label{eq:1750}
    \end{align}
    where $\lambda_{-1} = \lambda_{-2} = 0$.
    We deduce from from \eqref{eq:1749} and \eqref{eq:1750} that
    \[
        \Nc_\vp + \Nc_\psi \le h^{\frac{N}{3}} \sum_{q=0}^N
        \lp \lo F^\vp_q(h^{\frac{1}{3}};\cdot) \ro_{\Lr^2[h^{\frac{2}{3}}](\R_-)}
        + \lo F^\psi_q(h^{\frac{1}{3}};\cdot) \ro_{\Lr^2[h](\R_+)} \rp
    \]
    with  $F^\vp_q\in\Cc^\infty([0,1],\Sc(\R_-))$ and
    $F^\psi_q\in\Cc^\infty([0,1],\Sc(\R_+))$,
    and finally that
    \begin{equation}\label{eq:1756}
        \Nc_\vp + \Nc_\psi \le C_N h^{\frac{N}{3}}
    \end{equation}
    for some constant $C_N$ independent of $h$.

    Let us consider now the two commutators norms $\Nc'_\vp$ and $\Nc'_\psi$.
    We observe that the coefficients of the operators $\big[ \Lc_h^-,\chi(h^{\frac{2}{3}}\ee\cdot)\big]$ and $\lc \Lc_h^+,\chi(h\ee\cdot)\rc$ are zero in the regions defined by $-\delta h^{-\frac{2}{3}} \le \sigma \le 0$ and $0 \le \rho \le \delta h^{-1}$, respectively. This allows us to deduce that
    \[
        \Nc'_\vp + \Nc'_\psi \le
        \lp \lo G^\vp(h^{\frac{1}{3}};\cdot) \ro_{\Lr^2[h^{\frac{2}{3}}](-\infty,-\delta h^{-\frac{2}{3}})}
        + \lo G^\psi(h^{\frac{1}{3}};\cdot) \ro_{\Lr^2[h](\delta h^{-1},+\infty)} \rp
    \]
    with functions $G^\vp\in\Cc^\infty([0,1],\Sc(\R_-))$ and
    $G^\psi\in\Cc^\infty([0,1],\Sc(\R_+))$.  Lemma \ref{lem:dec} shows that $ \Nc'_\vp + \Nc'_\psi  = \Oc(h^\infty)$. Combined with \eqref{eq:1756} true for all $N\in\N$, this complete
    the proof of part (iv) of the Lemma.
\end{proof}

%--------------------------------------------------------------------
\subsubsection{Proof of Theorem \ref{th:(a)}}\label{sec:ProofA}
Choose $p\in\{\pm 1\}$ and $j\in\N$.
In order to meet all the requirements listed in Definition \ref{def:quasi}, we modify the sequence of functions $(\uu_{\pj}(m))_{m\ge 1}$ constructed in Definition \ref{def:uu_A}, so that each such function satisfies the jump conditions in \eqref{eq:dom}. To lift the jumps of $\uu_{\pj}(m)$, we define the ``radial'' function
\begin{align}\label{eq:1555}
    v^*(m;\xi) = \tfrac{1}{2}\chi(\xi) \left\lbrace\begin{array}{ll}
        -\lc\vu_{\pj}(m)\rc_{\xi=0}
        - n_0^{1-p}\xi\lc n^{p-1}\, \partial_r \vu_{\pj}(m) \rc_{\xi=0}                                & \xi\le 0 \\[1ex]
        \phantom{-}\lc\vu_{\pj}(m)\rc_{\xi=0} + \xi\lc n^{p-1}\, \partial_\nu \vu_{\pj}(m) \rc_{\xi=0} & \xi > 0
    \end{array}\right.
\end{align}
where $\chi$ can be taken as the same cut-off function used in Definition \ref{def:uu_A}.
We set
\[
    \uu_{p;j}(m;r,\theta) \coloneqq \lp\vu\lp m;\tfrac{r}{R}-1\rp-v^*\lp m;\tfrac{r}{R}-1\rp\rp \ \e^{\ic m \theta}
\]
and $\ku_{p;j}(m) \coloneqq \ku(m)$.
Using \eqref{eq:norm_A}, we can normalize the function $\uu_{p;j}(m)$ in the $\Lr^2$ norm.
Relying on Lemmas \ref{lem:Borel_A} and \ref{lem:QM_A} it is easy to check that the family $(\Kg_{p;j},\Ug_{p;j})$ where $\Kg_{p;j} = (\ku_{p;j}(m))_{m\ge 1}$ and $\Ug_{p;j} = (\uu_{p;j}(m))_{m\ge 1}$ satisfies the four conditions of Definition \ref{def:quasi}.

%============================
\section{Case \textsc{(b)} \ Half-quadratic potential well}
\label{sec:(b)}
%============================
Our concern is now the case when $\kb=0$ and $\mb >0$.
According to the same plan as before, we start with the complete description of the quasi-pairs that are constructed in the rest of the section.
The corresponding statement has to be combined with Theorem \ref{th:QuasiRes} to imply Theorem \ref{th:B}.
As mentioned earlier,  case  \textsc{(a)} and case \textsc{(b)} share general concepts in the way the asymptotic expansion of quasi-pairs is obtained.
Hence,  we do not provide a comprehensive proof of Theorem \ref{th:(b)}
but instead highlight the differences with the proof of  Theorem \ref{th:(a)}.

%= = = = = = = = = = = = = = = = = = = = = = = = = = = = = 
\subsection{Statements}
\label{ss:statB}
%= = = = = = = = = = = = = = = = = = = = = = = = = = = = = 

\begin{theorem}
    \label{th:(b)}
    Choose $p\in\{\pm1\}$.
    Let Assumptions \ref{as:n} be verified and, using Notation \ref{no:n}, assume that $\kb = 0$ and $\mb >0$.
    Then there exists for each $j\in\N$, a family of resonance quasi-pairs $\Fg_{p\pv j}=(\Kg_{p\pv j},\Ug_{p\pv j})$ of whispering gallery type ({\em cf.}  Definition \emph{\ref{def:quasi}})  with $\Kg_{p\pv j} = (\ku_{p\pv j}(m))_{m\ge1}$ and $\Ug_{p\pv j} = (\uu_{p\pv j}(m))_{m\ge1}$.

    \medskip
    (i) The regularity properties \eqref{eq:contk}--\eqref{eq:contkN} with respect to $m$ holds with $\beta = \frac{1}{2}$: There exist coefficients $\Kr_{p\pv j}^{\,\ell}$ for any $\ell\in\N$, and constants $C_N$ so that
    \begin{equation}
        \label{eq:kNb}
        \forall N\in\N,\qquad
        \left|\frac{\ku_{p\pv j}(m)}{m} -
        \sum_{\ell=0}^{N-1} \,\Kr_{p\pv j}^{\,\ell}\, m^{-\ell/2} \right|
        \le C_N\, m^{-N/2}.
    \end{equation}
    The coefficients $\Kr_{p\pv j}^{\,0}$ are all equal to $\lp R\,n(R)\rp^{-1}$, the coefficients of degree $1$ are zero, and the coefficients of degree $2$ are all distinct with $j$, see \eqref{eq:0915b}.

    \medskip
    (ii) The functions $\uu_{p\pv j}(m)$ still have the form \eqref{eq:ua}
    with radial functions $\wu_{p\pv j}(m)$ that have a boundary layer structure around $r=R$ with the different scaled variables $\sigma$  as $r<R$ and $\rho$ as $r>R$:
    \[
        \sigma=m^{\frac{1}{2}}\lp\tfrac{r}{R}-1\rp \ \ \texte{if}\ \ r<R
        \qquad\mbox{and}\qquad
        \rho=m\lp\tfrac{r}{R}-1\rp \ \ \texte{if}\ \ r>R.
    \]
    There exist smooth functions $\Phi_{p\pv j} \in \Cc^\infty([0,1],\Sc(\R_-)) : (t,\sigma) \mapsto \Phi_{p\pv j}(t,\sigma)$ and $\Psi_{p\pv j} \in \Cc^\infty([0,1],\Sc(\R_+)) : (t,\rho) \mapsto \Psi_{p\pv j}(t,\rho)$ such that
    \begin{equation}\label{eq:vb}
        \wu_{p\pv j}(m\pv r) =
        \Xg(r)\Big( \II_{r<R}(r)\,\Phi_{p\pv j}(m^{-\frac{1}{2}},\sigma) +
        \II_{r>R}(r)\, \Psi_{p\pv j}(m^{-\frac{1}{2}},\rho) \Big)
    \end{equation}
    where $\Xg\in\Cc^\infty_0(\R_+)$,  $\Xg\equiv 1$ in a neighborhood of $R$.
\end{theorem}

The first terms of the expansions of the quasi-resonances $\ku_{p\pv j}(m)$ in powers of $m^{-{1}/{2}}$ as $m\to\infty$ are as follows:
\begin{equation}\label{eq:0915b}
    \ku_{p\pv j}(m) = {\frac{m}{R{n_0}}} \lc 1 + \sum_{\ell=1}^3 \kr_{p\pv j}^\ell
    \lp\frac{\sqrt{\mb}}{m}\rp^{\frac{\ell}{2}} +
    \Oo\lp {m^{-2}} \rp \rc
\end{equation}
where $\kr_{p\pv j}^1 = 0$,
\begin{align*}
    \kr_{p\pv j}^2 = 2j+\frac{3}{2} &  &
    \text{and}                       &  &
    \kr_{p\pv j}^3 = {\Psi_{2j+1}^{\GH}}'(0)^2\, \lc\frac{-n_0^p}{\sqrt{n_0^2-1}} + \frac{4j+3}{9\, \mb^{\frac{3}{2}}}\lp 6+\frac{R^3 n_3}{n_0} - 6\mb\rp\rc.
\end{align*}
Here, $\Psi^{\GH}_\ell$ denotes the Gauss-Hermite function of order $\ell$, see \cite{AbrSte64,Olv97}.

The asymptotic expansions of the radial part of the quasi-modes $\wu_{p\pv j}(m)$ in  \eqref{eq:vb} starts as
\begin{equation}\label{eq:B_va01}
    \wu_{p\pv j}(m\pv r) =
    \Xg(r) \lp \Vs_{p\pv j}^{\,0}(m;r) +
    \lp\frac{1}{m}\rp^{\frac{1}{2}} \Vs_{p\pv j}^{\,1}(m;r)\rp
    + \Oo\lp m^{-1}\rp ,
\end{equation}
where, using the scaled variables $\sigma=m^{\frac{1}{2}}(\frac{r}{R}-1)$ and $\rho=m(\frac{r}{R}-1)$, we have
\begin{equation}\label{eq:B_V0}
    \Vs_{p\pv j}^{\,0}(m;r) = \left\{\begin{array}{ll}
        \Psi^{\GH}_{2j+1} \big(\mb^{\frac{1}{4}}\sigma\big) & \texte{if } r < R,   \\
        0                                                   & \texte{if } r \ge R,
    \end{array}\right.
\end{equation}
and
\begin{equation}\label{eq:B_V1}
    \Vs_{p\pv j}^{\,1}(m;r) = \frac{-n_0^p\, \mb^{\frac{1}{4}}\, ( {\Psi_{2j+1}^{\GH}})'(0)}{\sqrt{n_0^2-1}}
    \left\{\begin{array}{ll}
        \frac{1}{\Psi_0^{\GH}(0)} \Psi_0^{\GH}\big(\mb^{\frac{1}{4}}\sigma\big) + \wt{\vp}_1(\sigma) & \texte{if }r < R,    \\[1ex]
        \exp\big(-\frac{\sqrt{n_0^2-1}}{n_0}\,\rho\big)                                              & \texte{if } r \ge R,
    \end{array}\right.
\end{equation}
where  $\wt{\vp}_1 \in \Hr_0^1(\R_-) \cap \Sc(\R_-)$
is the unique solution to a Dirichlet problem posed on $\R_-$, see~\eqref{eq:ode_vp}. In contrast with case \textsc{(a)}, the expression of $\Vs_{p\pv j}^{\,1}$, though determined, is not explicit. This is the reason why we gave only two terms in the asymptotic expansions \eqref{eq:0915b}.

\begin{remark}\label{rem:gapB}
    When $m\to\infty$ and $j$ stays bounded, the joint quasi-spectrum $(m,\ku_{p\pv j}(m))$ tends to the rectangular lattice
\[
   (m,j) \longmapsto \lp m, \tfrac{1}{R{n_0}} \big( m + (2j+\tfrac{3}{2})\sqrt{\mb} \big) \rp.
\]
    Indeed, the gaps between two resonances with consecutive indices satisfy
\[
\begin{aligned}
    \ku_{p\pv j}(m+1) - \ku_{p\pv j}(m) &=
    \frac{1}{Rn_0}  + \Oo\lp {m^{-\frac{3}{2}}} \rp \\
    \ku_{p\pv j+1}(m) - \ku_{p\pv j}(m) &=
    \frac{2\sqrt{\mb}}{R n_0}  + \Oo\lp {m^{-\frac{1}{2}}} \rp .
\end{aligned}
\]
\end{remark}

%= = = = = = = = = = = = = = = = = = = = = = = = = = = = =
\subsection{Proof}
%= = = = = = = = = = = = = = = = = = = = = = = = = = = = =

The general outline of the proof of Theorem \ref{th:(b)} is similar to the one of Theorem~\ref{th:(a)} in case \textsc{(a)}.
In particular, the general framework of section \ref{subsec:genConcepts} applies with
\[
    \varkappa = 2, \quad \gamma=\mb, \quad \alpha = \tfrac{1}{2}, \quad \alpha'=1, \quad \beta=\tfrac{1}{2}.
\]
We provide now the part of the proof of Theorem \ref{th:(b)} specific to  case \textsc{(b)}.

From general expressions \eqref{eq:Lh}--\eqref{eq:Apm}, we find that
\begin{empheq}[left={\hspace*{-3ex}\empheqlbrace\ }]{align}
    \label{eq:B012-}
    \bA^-_0 &= -\partial^2_\sigma + \mb \sigma^2, \quad
    \bA^-_1 = -2\partial_\sigma^2 + (p-2)\partial_\sigma - \sigma^3 \lp\frac{\nx_2}{n_0} + \frac{\nx_3}{3n_0}\rp,\\
    \label{eq:B012+}
    \bA^+_0 &= -\no_0^2\partial^2_\rho + \no_0^2 - 1, \quad
    \bA^+_1 = 0, \quad
    \bA^+_2 = -\no^2_0(\partial_\rho + 2\rho).
\end{empheq}

The analog of Lemma \ref{lem:Aqpm} describing the coefficients of the formal series of operator~ \eqref{eq:Apm}  in terms of powers of $h^\beta = h^\frac{1}{2}$ reads as follows.
\begin{lemma}
    \label{lem:Bqpm}
    For any integer $q\ge1$,
    \begin{align*}
        \bA^-_q & = A^-_q(\sigma) \,\partial^2_\sigma + B^-_q(\sigma)\,\partial_\sigma + C^-_q(\sigma)
                &                                                                                      & \text{with}\ \ A^-_q \in \P^{q},\ \ B^-_q \in \P^{q-1},\ \ C^-_q \in \P^{q+1}   \\
        \bA^+_q & = B^+_q(\rho)\,\partial_\rho + C^+_q(\rho)
                &                                                                                      & \text{with}\ \ B^+_q \in \P^{[\frac{q}{2}]-1},\ \ C^+_q \in \P^{[\frac{q}{2}]}.
    \end{align*}
\end{lemma}

We proceed as in Section \ref{sec:43}, associating to the system \eqref{eq:schr_lam} a formal series system of equations like \eqref{eq:formal}, in which the powers of $h$ are modified according to the values of $\alpha$, $\alpha'$, $\beta$, and $\varkappa$.
As a matter of fact, equating the series coefficients, we obtain in case \textsc{(b)} exactly the same infinite collection of systems \eqref{eq:formalq} as in case \textsc{(a)}, but with the new expressions of operators $\bA^\pm_q$.
The coefficients of the formal series expansions \eqref{eq:1223} are obtained by solving~\eqref{eq:formalq} for $q$ spanning $\N$.

\subsubsection{Initialization stage}
For $q=0$,  the couple of functions $(\vp_0,\psi_0)$ and the number $\lambda_0$ are obtained by solving \eqref{eq:formal0} with   $\bA^-_0$ and $  \bA^+_0$   given in \eqref{eq:B012-}--\eqref{eq:B012+}.
Since the equation $\bA^+_0\, \psi_{0} = 0$ with the Neumann condition at $0$ has no non-zero solution in $\Sc(\R_+)$, it is natural to take $\psi_0=0$. Then,  we are left with the following harmonic oscillator problem on $\R_-$  
\[
    -\vp_0''(\sigma) - \mb \,\sigma^2\vp_0(\sigma) = \lambda_0\vp_0(\sigma) \ \mbox{ for }\ \sigma\in(-\infty,0),
    \qquad\mbox{and}\qquad \vp_0(0)=0
\]
whose bounded solutions are generated by the odd Gauss-Hermite functions $\big\{\Psi^{\GH}_{2j+1}\big\}_{j\in\N}$.

\begin{lemma}\label{lem:lambda0_B}
    Let $j\in\N$. The couple of functions $(\vp_0,\psi_0)$ and the number $\lambda_0$ defined by
    \[
        \vp_0(\sigma) = \Psi^{\GH}_{2j+1}\lp\mb^{\frac{1}{4}} \sigma\rp, \quad
        \psi_0(\rho) = 0,\quad\mbox{and}\quad
        \lambda_0 = (4j+3) \sqrt{\mb}
    \]
    solve \eqref{eq:formal0} for $\bA^-_0$ and $\bA^+_0$ given in \eqref{eq:B012-}--\eqref{eq:B012+}.
\end{lemma}

\subsubsection{Sequence of nested problems and recurrence}
As in case \textsc{(a)},
reordering the terms in the system \eqref{eq:formalq} taking into account Lemma \ref{lem:Bqpm}, we obtain that
the couple of functions $(\vp_q,\psi_q)$ and the number $\lambda_q$ for $q\geq 1$ are solutions to
\begin{subequations}    \label{eq:sysQ_B}
    \begin{empheq}[left={ \hspace*{-10mm}(\mathcal{R}^{\text{\textsc{(b)}}}_q)\quad     \empheqlbrace}]{align}
        \hspace*{6mm}   -\vp_q''(\sigma) + (\mb\sigma^2 - \lambda_0)\vp_q(\sigma) &= \lambda_q\,\vp_0(\sigma) + S_q^\vp(\sigma)
        & \sigma \in \R_-
        \label{eq:sysQ_Ba} \\
        \hspace*{6mm}   -\no_0^{2}\psi_q''(\rho) + \lp \no_0^{2}-1\rp\psi_q(\rho) &= S_q^\psi(\rho) & \rho \in \R_+
        \label{eq:sysQ_Bb} \\
        \vp_q(0) &= \psi_q(0) &
        \label{eq:sysQ_Bc} \\
        \psi_q'(0) &= \no_0^{p-1}\vp_{q-1}'(0) &
        \label{eq:sysQ_Bd}
    \end{empheq}
\end{subequations}
with right hand side terms $S_q^\vp$ and $S_q^\psi$ defined as (recall that $\bA^+_1=0$)
\begin{equation}
    \label{eq:Sq_B}
    S_q^\vp = - \bA^-_q \,\vp_{0} + \sum_{\ell=1}^{q-1} (\lambda_\ell - \bA^-_\ell) \,  \vp_{q-\ell}
    \quad\mbox{and}\quad
    S_q^\psi = \sum_{\ell=2}^{q} (\lambda_{\ell-2} - \bA^+_\ell) \, \psi_{q-\ell} .
\end{equation}
\begin{notation}\label{no:weightedSobolev}
    For a real number $t$ let $\omega_t : \sigma \mapsto \exp\lp 2^t |\sigma|\rp$.
    We denote by $\Lr^2(\R_-,\omega_t)$ and $\Hr^\ell(\R_-,\omega_t)$, the weighted Lebesgue and Sobolev spaces with measure $\omega_t(\sigma) \dd{\sigma}$.
\end{notation}

\begin{proposition}\label{pro:solQ_B}
    Choose $j\in\N$ and take $\vp_0,\psi_0,\lambda_0 $ as given in  Lemma \ref{lem:lambda0_B}.
    For any $q\ge1$, there exist
    \begin{itemize}
        \item a unique $\lambda_q\in\R$
        \item a unique real number $c_q$, a unique real sequence $m \in \N \mapsto b_q^i$ such that $b_q^j = 0$, and a unique polynomial $P_q^\psi\in\P^{q-1}$
    \end{itemize}
    such that setting
    \begin{equation}
        \label{eq:PQ_B}
        \begin{aligned}
            \vp_q(\sigma) & = c_q\Psi_0^{\GH}\big(\mb^{\frac{1}{4}}\sigma\big)
            + \wt{\vp}_q(\sigma) \quad\mbox{with}\quad
            \wt{\vp}_q(\sigma) = \sum_{i\in\N}
            b_q^i \Psi_{2i+1}^{\GH}\big(\mb^{\frac{1}{4}}\sigma\big)
                          & \forall\sigma\in\R_-                                          \\[-1ex]
            \psi_q(\rho)  & = \te P_q^\psi(\rho) \exp\big(-\rho\sqrt{1-\no_0^{-2}}\,\big)
                          & \forall\rho\in\R_+\end{aligned}
    \end{equation}
    the collection $(\vp_0,\ldots,\vp_q,\psi_0,\ldots,\psi_q,\lambda_0,\ldots,\lambda_q)$ solves the sequence of problems $(\mathcal{R}^{\text{\textsc{(b)}}}_\ell)$ for $\ell = 0,\ldots,q$. Moreover $\wt{\vp}_q \in \Hr_0^1(\R_-) \cap \Hr^2(\R_-,\omega_{-q})$
\end{proposition}
%---
\begin{proof}
    The proof is quite similar to the one of Proposition \ref{pro:solQ_A} and we will focus on the main differences.
    We proceed by induction on $q$.
    For $q=0$, Lemma \ref{lem:lambda0_B} provides $\lambda_0$, $\vp_0$, and $\psi_0$ solutions to $(\mathcal{R}^{\text{\textsc{(b)}}}_0)$ and we readily obtain $c_0 = 0$, $\wt{\vp}_0 = \Psi_{2j+1}^{\GH}(\mb^{\frac{1}{4}}\,\cdot) \in \Hr_0^1(\R_-)$, and $P_0 = 0$.
    Moreover, $\vp_0$ belongs to  $\Hr^2(\R_-,\omega_0)$ because $\Psi_{2j+1}^{\GH}$ is defined as the product of the $(2j+1)$-th order Hermite polynomial of degree $2j+1$ by $x \mapsto \exp(-\frac{x^2}{2})$.

    Let $q \ge 1$ and suppose that $(\lambda_\ell)_{0\le\ell\le q-1}$, $(\vp_\ell)_{0\le\ell\le q-1}$, and $(\psi_\ell)_{0\le\ell\le q-1}$ are solutions to problems $(\mathcal{R}^{\text{\textsc{(b)}}}_\ell)$ for $\ell=0,\ldots,q-1$, and satisfy \eqref{eq:PQ_B}.
    Solving equation \eqref{eq:sysQ_Bb} for $\psi_q$ proceed in a way very similar to \eqref{eq:sysQ_Ab} in the proof of Proposition \ref{pro:solQ_A} to show that
    there exists $P_q^\psi \in \P^{q-1}$ such that
    \begin{align*}
        \psi_q(\rho) = P_q^\psi(\rho) \te\exp\big( -\rho\sqrt{1-\no_0^{-2}}\,\big).
    \end{align*}

    Let us now consider equation \eqref{eq:sysQ_Bb} for   $\vp_q$.
    First of all, we obtain by induction that $\vp_\ell \in \Hr^2(\R_-,\omega_{1-q})$ for all $\ell \in \{0,\ldots,q-1\}$.
    Then, using Lemma \ref{lem:Bqpm} and \eqref{eq:Sq_B}, it follows that  $S_q^\vp \in \Lr^2(\R_-,\omega_{1/2-q})$.
    Note that the value of the constant $q-1$  in the exponential weight is reduced by $\frac{1}{2}$  to $\frac{1}{2}-q$ in order to absorb the polynomials behavior.
    To solve equation~\eqref{eq:sysQ_Ba} with the non-homogeneous boundary condition  \eqref{eq:sysQ_Bc} we introduce as new unknown $\wt{\vp}_q = \vp_q - c_q\Psi_0^{\GH}(\mb^{\frac{1}{4}}\cdot)$ where $c_q = \frac{\psi_q(0)}{\Psi_0^{\GH}(0)}$.
    It belongs to $ \Hr^2(\R_-,\omega_{1-q})$ and the Dirichlet problem \eqref{eq:sysQ_Ba}, \eqref{eq:sysQ_Bc} becomes
    \begin{subequations}   \label{eq:ode_vp}
        \begin{empheq}[left={ \empheqlbrace}]{align}
            -\wt{\vp}_q'' + (\mb\sigma^2 - \lambda_0) \wt{\vp}_q &=
            \lambda_q\Psi_{2j+1}^{\GH}(\mb^{\frac{1}{4}}\cdot) + \wt{S}_q^\vp
            \quad \forall \sigma\in (-\infty,0) \\
            \wt{\vp}_q(0) &= 0
        \end{empheq}
    \end{subequations}
    where $\wt{S}_q^\vp = S_q^\vp + 2(2j+1)\sqrt{\mb}\, c_q\Psi_0^{\GH}(\mb^{\frac{1}{4}}\cdot)$.
    Problem \eqref{eq:ode_vp} has a solution only when the left hand side function $\lambda_q\Psi_{2j+1}^{\GH}(\mb^{\frac{1}{4}}\cdot) + \wt{S}_q^\vp$ is orthogonal to $\Psi_{2j+1}^{\GH}(\mb^{\frac{1}{4}}\cdot)$.
    It follows that we must have
    \begin{align}
        \label{eq:lamq}
        \lambda_q = -2\mb^{\frac{1}{4}}\, \int_{-\infty}^0
        \Psi_{2j+1}^{\GH}(\mb^{\frac{1}{4}}\sigma)\, \wt{S}_q^\vp(\sigma) \dd{\sigma} .
    \end{align}
    Finally, from Lemma \ref{lem:GHsolPoids}, there exists a unique $\wt{\vp}_q$ in $\Hr_0^1(\R_-) \cap \Hr^2(\R_-,\omega_{-q})$ and orthogonal to $\Psi_{2j+1}^{\GH}$ solution to \eqref{eq:ode_vp}.
    The formula giving $\wt{\vp}_q$ is
    \begin{align}\label{eq:1155}
        \wt{\vp}_q = \sum_{i=0,\ i\neq j}^{+\infty}
        \frac{1}{4(i-j)}\,
        \Big( \wt{S}_q^\vp,\, \varsigma_i\Psi_{2i+1}^{\GH}\Big)_{\Lr^2(\R_-)}
        \, \varsigma_i\Psi_{2i+1}^{\GH}
    \end{align}
    where $\varsigma_i = \Vert\Psi_{2i+1}^{\GH}\Vert_{\Lr^2(\R_-)}^{-1}$.
\end{proof}
%--------------------------------------------------------------------

\begin{remark}\label{rem:Algo_B}
    In contrast to case \textsc{(a)}, we cannot deduce from Proposition \ref{pro:solQ_B} a finite algorithm to compute the terms of the sequence $(\vp_q)_{q\in\N}$.
    The reason is that for $q\geq 1$, the sum of the series \eqref{eq:1155} cannot be computed explicitly.
    However, a few terms are explicit: We know $(\psi_0,\vp_0,\lambda_0)$ so we can compute, first  $P_1^\psi$, then $c_1$, and, after this, $S_1^\vp$.
    With these latter quantities, we can deduce an explicit expression of $\wt{S}_1^\vp$ as a finite sum of polynomials times Gauss-Hermite functions.
    Now, from the definition of the Gauss-Hermite functions \cite[eq.\ 1.3.8]{Hel13} and recurrence relations on  Hermite polynomials \cite[Sect.\ 18.9(i)]{Nist}, we deduce the following recurrence relations
    for $i \ge 0$ and $z \in \R$,
    \begin{subequations}\label{eq:GHrelation}
        \begin{align}
            \partial_z\Psi_i^{\GH}(z) & = \left(\tfrac{i}{2}\right)^\frac{1}{2} \Psi_{i-1}^{\GH}(z) - \left(\tfrac{i+1}{2}\right)^\frac{1}{2} \Psi_{i+1}^{\GH}(z),
            \label{eq:GHrelationa}                                                                                                                                                        \\
            z\Psi_i^{\GH}(z)          & = \left(\tfrac{i}{2}\right)^\frac{1}{2} \Psi_{i-1}^{\GH}(z) + \left(\tfrac{i+1}{2}\right)^\frac{1}{2} \Psi_{i+1}^{\GH}(z). \label{eq:GHrelationb}
        \end{align}
    \end{subequations}
    Hence we can rewrite $\wt{S}_1^\vp$ as a finite sum of Gauss-Hermite functions and with this we can compute explicitly $\lambda_1$ given by \eqref{eq:lamq}. Nevertheless $\wt{\vp}_1$ will be an infinite sum of Gauss-Hermite functions so, for $q \ge 2$, $\lambda_q$ does not have a closed form.
\end{remark}

\begin{lemma}\label{cor:deriveePoids}
    For all $q \in \N$, we have $\vp_q \in \Sc(\R_-)$ and $\psi_q\in\Sc(\R_+)$.
\end{lemma}
%---
\begin{proof}
    From the expression \eqref{eq:PQ_B} of $\psi_q$, it is obvious that it belongs to $\Sc(\R_+)$.

    From Proposition \ref{pro:solQ_B}, we know that $\vp_q$ and its derivatives of order $\le2$ are exponentially decaying as $\sigma\to\infty$.
    Concerning higher order derivatives $\vp_q^{(i)}$, from the identity $\vp_q'' = (\mb\sigma^2 -\lambda_0)\vp_q - \lambda_0 \vp_0 -  S_q^\vp$ deduced from \eqref{eq:sysQ_Ba}, from \eqref{eq:Sq_B} and Lemma \ref{lem:Bqpm}, we find that there exist families of polynomials $P_{q,i}^\ell, Q_{q,i}^\ell$ such that
    \begin{equation}\label{eq:1505}
        \vp_q^{(i)} = \sum_{\ell = 0}^q \left( P_{q,i}^\ell\, \vp_\ell + Q_{q,i}^\ell\, \vp_\ell' \right).
    \end{equation}
    Hence $\vp_q^{(i)}$ is exponentially decaying too, and we have proved that $\vp_q$ belongs to $\Sc(\R_-)$.
\end{proof}
%--------------------------------------------------------------------

\subsubsection{Convergence}\label{sec:ccls_B}

The proof that the formal series
\begin{equation}
    \label{eq:formalsb}
    \sum_{q\in\N} \lambda_q h^{\frac{q}{2}},\qquad
    \sum_{q\in\N} \vp_q h^{\frac{q}{2}},  \qquad\mbox{and}\qquad
    \sum_{q\in\N} \psi_q h^{\frac{q}{2}},
\end{equation}
obtained from Proposition \ref{pro:solQ_B} give rise to a family of resonance quasi-pairs in the sense of Definition \ref{def:quasi} can be achieved exactly as in Section \ref{sec:ccls_A} for case \textsc{(a)}.
Namely, Lemma \ref{lem:Borel_A} and Definition \ref{def:uu_A} are respectively replaced by the following Lemma \ref{lem:Borel_B}  and Definition \ref{def:uu_B}.
\begin{lemma}\label{lem:Borel_B}
    Let $(\lambda_q)_{q\in\N}$, $(\vp_q)_{q\in\N}$ and $(\psi_q)_{q\in\N}$ given by Proposition \ref{pro:solQ_B}.
    There exist smooth functions $\lu_{\pj}\in\Cc^\infty([0,1])$, $\Phi_{\pj}\in\Cc^\infty([0,1],\Sc(\R_-))$ and
    $\Psi_{\pj}\in\Cc^\infty([0,1],\Sc(\R_+))$ such that for all $(h,\sigma,\rho)\in [0,1]\times\R_-\times\R_+$ and for all integer $N\ge0$, we have the following finite expansions with remainders
    \begin{subequations}
        \begin{align}
            \lu_{\pj}(h^{\frac{1}{2}})         & =
            \sum_{q=0}^{N-1} h^{\frac{q}{2}}\, \lambda_q + h^{\frac{N}{2}} R_N^\lambda(h^{\frac{1}{2}}),
                                               &   & \mbox{with}\quad R_N^\lambda\in\Cc^\infty([0,1])        \\
            \Phi_{\pj}(h^{\frac{1}{2}};\sigma) & =
            \sum_{q=0}^{N-1} h^{\frac{q}{2}}\, \vp_q(\sigma) + h^{\frac{N}{2}} R_N^\vp(h^{\frac{1}{2}};\sigma),
                                               &   & \mbox{with}\quad R_N^\vp\in\Cc^\infty([0,1],\Sc(\R_-))  \\
            \Psi_{\pj}(h^{\frac{1}{2}};\rho)   & =
            \sum_{q=0}^{N-1} h^{\frac{q}{2}}\, \psi_q(\rho) + h^{\frac{N}{2}} R_N^\psi(h^{\frac{1}{2}};\rho)
                                               &   & \mbox{with}\quad R_N^\psi\in\Cc^\infty([0,1],\Sc(\R_+))
        \end{align}
    \end{subequations}
\end{lemma}

\begin{definition}\label{def:uu_B}
    Choose a real number $\delta\in (0,\frac{1}{2})$ and a smooth cut-off function $\chi$, $0\le \chi\le1$, such that $\chi(\xi) = 1$ for $|\xi| \le \delta$ and $\chi(\xi) = 0$ for $|\xi| \ge 2\delta$.
    We define for any integer $m \ge 1$ with the notation $h=m^{-1}$, the quantities:
    \begin{align*}
        \ku_{\pj}(m)          & = \frac{m}{R\no_0}
        \sqrt{1+h\ \lu_{\pj}(h^{\frac{1}{2}})},                                                                                                                \\
        \vu_{\pj}(m;\xi)      & = \chi(\xi)\left\lbrace\begin{array}{ll}
            \Phi_{\pj}(h^{\frac{1}{2}};h^{-1}\xi) , & \xi \le 0 \\
            \Psi_{\pj}(h^{\frac{1}{2}};h^{-1}\xi),  & \xi > 0
        \end{array}\right. &                                              & \xi\in (-1,+\infty) \\
        \uu_{\pj}(m;r,\theta) & = \vu_{\pj}\lp m;\tfrac{r}{R}-1\rp\, \e^{\ic m \theta}
                              &                                                           & (r,\theta) \in (0,+\infty)\times\R / 2\pi\Z.
    \end{align*}
\end{definition}

One can show that the sequence $(\ku_{\pj}(m),\uu_{\pj}(m))_{m\ge 1}$  is a family of ``almost'' quasi-pairs in the sense of Lemma \ref{lem:QM_A}.
The main difference with case \textsc{(a)} in proving Lemma \ref{lem:QM_A} for the sequence $(\ku_{\pj}(m),\uu_{\pj}(m))_{m\ge 1}$ introduced in Definition \ref{def:uu_B} is that we do not have anymore an explicit expression for $\vp_q$ but this does not prevent to obtain the same estimates as in case~\textsc{(a)}.
We refer to \cite{MoitierPhD} for details.

\subsubsection{Proof of Theorem \ref{th:(b)}}\label{sec:ProofB}

A further correction will have to be made to transform
the sequence of functions $(\uu_{\pj}(m))_{m\ge 1}$ constructed in Definition \ref{def:uu_B} into a true family of resonance quasi-modes in the sense of Definition \ref{def:quasi}.
We set
\[
    \uu_{p;j}(m;r,\theta) \coloneqq \lp\vu\lp m;\tfrac{r}{R}-1\rp-v^*\lp m;\tfrac{r}{R}-1\rp\rp \ \e^{\ic m \theta}
\]
where $v^*$ is defined as in \eqref{eq:1555}
and $\ku_{p;j}(m) \coloneqq \ku(m)$.
Relying on Lemmas \ref{lem:Borel_B} and the analogous of \ref{lem:QM_A} for case \textsc{(b)}, one can check that the family $(\Kg_{p;j},\Ug_{p;j})$ where $\Kg_{p;j} = (\ku_{p;j}(m))_{m\ge 1}$ and $\Ug_{p;j} = (\uu_{p;j}(m))_{m\ge 1}$ satisfies the four conditions of Definition \ref{def:quasi}.

%============================
\section{Case \textsc{(c)} \ Quadratic potential well}
\label{sec:(c)}
%============================
We are now under Assumption \eqref{eq:(c)}.
We recall that case \textsc{(c)} corresponds to a situation where $\kb<0$ and
the potential $W$ has no local minimum at $R$ but has at least one local inner minimum $R_0$ over $(0,R)$.
Here, we construct explicit families of resonance quasi-pairs $\Fg_{p\pv j}$ localized around the circle $r=R_0$  inside the cavity $\Omega$.
Note that strictly speaking, these families of resonance quasi-pairs are not of whispering gallery type.

%= = = = = = = = = = = = = = = = = = = = = = = = = = = = = 
\subsection{Statements}
\label{ss:statC}
%= = = = = = = = = = = = = = = = = = = = = = = = = = = = = 

\begin{theorem}
    \label{th:(c)}
    Choose $p\in\{\pm1\}$.
    Let Assumptions \ref{as:n} be satisfied and
    assume $\kb < 0$.
    Let $R_0 \in (0,R)$ such that $1+\frac{R_0\, n'(R_0)}{n(R_0)} = 0$ and $\mb_0 \coloneqq 2-\frac{R_0^2\, n''(R_0)}{n(R_0)} > 0$, {\em cf.} \eqref{eq:C}.
    Then,  for each $j\in\N$, there exists a family of resonance quasi-pairs $\Fg_{p\pv j}=(\Kg_{p\pv j},\Ug_{p\pv j})$ with $\Kg_{p\pv j} = (\ku_{p\pv j}(m))_{m\ge1}$ and $\Ug_{p\pv j} = (\uu_{p\pv j}(m))_{m\ge1}$.

    \medskip
    (i) The regularity property \eqref{eq:contk}--\eqref{eq:contkN} with respect to $m$ holds with $\beta = \frac{1}{2}$, see \eqref{eq:kNb}.
    The coefficients $\Kr_{p\pv j}^{\,0}$ are all equal to $\lp R_0\,n(R_0)\rp^{-1}$, the coefficients of degree $1$ are zero, and the coefficients of degree $2$ are all distinct with $j$, see \eqref{eq:0915c}.

    \medskip
    (ii)  The functions $\uu_{p\pv j}(m)$ still have the form \eqref{eq:ua}
    with radial functions $\wu_{p\pv j}(m)$ that are smooth in the scaled variables $\sigma=m^{\frac{1}{2}}(r / R_0-1)$.
    There exists a smooth function $\Phi_{p\pv j} \in \Cc^\infty([0,1],\Sc(\R)) : (t,\sigma) \mapsto \Phi_{p\pv j}(t,\sigma)$ such that
    \begin{equation}\label{eq:vc}
        \wu_{p\pv j}(m\pv r) =
        \Xg(r)\, \Phi_{p\pv j}(m^{-\frac{1}{2}},\sigma)
    \end{equation}
    where $\Xg\in\Cc^\infty_0(\R_+)$,  $\Xg\equiv 1$ in a neighborhood of $R_0$.
\end{theorem}

These families of resonance quasi-pairs are \emph{not} of whispering gallery type:
The quasi-modes are strictly localized inside the cavity.
The asymptotic expansion of $\ku_{p\pv j}$ starts as:
\begin{equation}\label{eq:0915c}
    \ku_{p\pv j}(m) = \frac{m}{R_0{n(R_0)}}
    \lc 1 + \sum_{\ell'=1}^2 \kr_{p\pv j}^{2\ell'}
    \lp\frac{\sqrt{\mb_0}}{m}\rp^{\ell'} +
    \Oo\lp {m^{-\frac{5}{2}}} \rp \rc
\end{equation}
with $\kr_{p\pv j}^2 = j+\frac{1}{2}$ and
\begin{multline*}
    \kr_{p\pv j}^4 = \frac{1}{64}\biggl[ 13-16p + \frac{8p^2-16p-5}{\mb_0} - \frac{2\eta_3-3\eta_4}{3\mb_0^2} -\frac{7\eta_3^2}{9\mb_0^3}\\
        +(2j+1)^2 \lp 5 - \frac{35}{\mb_0} + \frac{10\eta_3+\eta_4}{\mb_0^2} - \frac{5\eta_3^2}{3\mb_0^3}\rp\biggr] 
\end{multline*}
where
\begin{align*}
    \eta_3 = 6 + \frac{R_0^3\, n^{(3)}(R_0)}{n(R_0)} \quad
    \text{and} \quad
    \eta_4 = 24 - \frac{R_0^4\, n^{(4)}(R_0)}{n(R_0)} .
\end{align*}
Note that the coefficients of odd order $\kr_{p\pv j}^1$ and $\kr_{p\pv j}^3$ are zero.

The asymptotic expansion of the quasi-modes starts with
\begin{equation}
    \uu_{p\pv j}(\pm m\pv x,y) =
    \Xg(r) \Psi^{\GH}_j\lp\mb_0^{\frac{1}{4}} m^{\frac{1}{2}} \lp\tfrac{r}{R_0}-1\rp \rp\e^{\pm\ic m\theta} + \Oo\lp m^{-\frac{1}{2}}\rp .
\end{equation}

\begin{remark}
    As in case  \textsc{(b)},
    the quasi-resonances tend to a rectangular lattice and 
    the gaps between two resonances with consecutive indices satisfy
\[
\begin{aligned}
    \ku_{p\pv j}(m+1) - \ku_{p\pv j}(m) &= \frac{1}{R_0 n(R_0)}  + \Oo\lp {m^{-2}} \rp ,\\
    \ku_{p\pv j+1}(m) - \ku_{p\pv j}(m) &=
    \frac{\sqrt{\mb_0}}{R_0 n(R_0)}  + \Oo\lp {m^{-1}} \rp .
\end{aligned}
\]
\end{remark}

\subsection{Proof}
The proof of Theorem \ref{th:(c)} can be seen as a simpler version of the proof of  Theorem \ref{th:(b)}
since the driving operator $\bA^-_0 =  -\partial_\sigma^2 + \mb_0\sigma^2$ of the asymptotic expansion is the same quadratic oscillator on both side of the potential well location $R_0$, \textsl{i.e.}\ $ \bA^+_0 =\bA^-_0$.
Therefore, we will not detail the entire proof of Theorem \ref{th:(c)} but we will focus on an interesting byproduct of our approach compared to the results of \cite{HelSjo86}, viz a finite algorithm for computing the terms of the asymptotic expansion of the resonance quasi-pairs.

In the framework of the Schr\"odinger analogy introduced in Section \ref{sec:3types},
we start this time by introducing the dimensionless variable $\xi = \frac{r}{R_0}-1$
(instead of $\xi = \frac{r}{R}-1$ as in the two previous cases)
and the unknown $v$ such that   $v(\xi) = w(R_0(1+\xi))$.
This leads to the same equation \eqref{eq:schr_loc} where $\wt{ E} = R_0^2\, \nx(0)^2\, ( E - W_0)$ with $\nx(\xi) =n(R_0(1+\xi))$.
Compared to the general framework introduced in Section \ref{subsec:genConcepts}
for cases \textsc{(a)} and  \textsc{(b)}, the potential $V$ is smooth at its local minimum at $\xi=0$.
As a consequence, it is not anymore necessary to introduce a different scaling on both side of $\xi=0$.
Moreover, it is still possible to take advantage of the framework of Section \ref{subsec:genConcepts}, but taking into account the fact the variable $\sigma = h^{-\alpha}\,\xi $  must be considered over $\R$ and not only over $\R_-$.
This framework applies with the same relevant quantities as in case \textsc{(b)} (the ones affecting $\Lc$ on $\R_-$).
Denoting by $\vp$ the new unknown such that $\vp(\sigma)=v(\xi)$, equation \eqref{eq:schr_loc} becomes $-\Lc_h\vp + V_h\vp = \lambda\vp$, $\sigma \in \R$, where the operator $\Lc_h$ and the potentials $V_h$ have the same expressions than $\Lc_h^-$ in \eqref{eq:Lh} and $V_h^-$ in~\eqref{eq:Vh} with $\nx(\xi) =n(R_0(1+\xi))$.
The decay condition is $\vp \in \Sc(\R)$.

We define a formal series of operators in terms of powers of $h^\frac{1}{2}$, similarly to \eqref{eq:Apm}, as $-\Lc_h + V_h \sim \sum_{q\in\N} h^{\frac{q}{2}}\,\bA_q$ and we look for a function $\vp$ and a scalar $\lambda$ in the form of the formal series $\vp = \sum_{q\in\N} h^{\frac{q}{2}}\, \vp_q$ and $\lambda = \sum_{q\in\N} h^{\frac{q}{2}}\, \lambda_q$.
One can show that the coefficients $\bA_q$, $q\in\N$, satisfy Lemma \ref{lem:Bqpm} (the statement on $\bA_q^-$).
Then, by the same arguments as in cases \textsc{(a)} and \textsc{(b)} that can equally apply here, we obtain that $ (\vp_0,\lambda_0)$ is solutions to the full harmonic oscillator equation (in opposition to the half harmonic oscillator of  case \textsc{(b)})
\begin{equation}\label{eq:1134}
    -\vp_0''(\sigma) + \mb_0\, \sigma^2\vp_0(\sigma) = \lambda_0\vp_0, \qquad
    \sigma \in \R, \qquad \vp_0 \in \Sc(\R),
\end{equation}
and that for $q\geq 1$, $ (\vp_q,\lambda_q)$ are solutions to the sequence of problems
\begin{subequations}
    \begin{empheq}[left={ \hspace*{-10mm}(\mathcal{R}^{\text{\textsc{(c)}}}_q)\quad     \empheqlbrace}]{align}
        \hspace*{6mm}   -\vp_q''(\sigma) + (\mb_0\sigma^2 - \lambda_0)\vp_q(\sigma) &= \lambda_q\,\vp_0(\sigma) + S_q^\vp(\sigma)
        & \sigma \in \R \label{eq:sysQ_C}\\
        \vp_q &\in \Sc(\R) &
    \end{empheq}
\end{subequations}
with the right hand side term $S_q^\vp$ defined as $S_q^\vp = - \bA_q \,\vp_{0} + \sum_{\ell=1}^{q-1} (\lambda_\ell - \bA_\ell) \, \vp_{q-\ell}$.

Solutions to the full harmonic oscillator equation \eqref{eq:1134} are
\begin{equation}\label{eq:lambda0_C}
    \vp_0(\sigma) = \Psi^{\GH}_j\lp\mb_0^{\frac{1}{4}} \sigma\rp \quad
    \text{and} \quad
    \lambda_0 = (2j+1) \sqrt{\mb_0}\quad  (j\in\N).
\end{equation}

For $q\geq 1$, the features of the solution $(\vp_q, \lambda_q)$ to problem $(\mathcal{R}^{\text{\textsc{(c)}}}_q)$ are detailed in the following proposition.
Its proof below also provides an algorithm to compute $\vp_q$ and $\lambda_q$.

\begin{proposition}\label{pro:solQ_C}
    Let $j\in\N$ and let $(\vp_0,\lambda_0)$ given by \eqref{eq:lambda0_C}.
    Then there exist, for any $q\ge1$, a unique $\lambda_q\in\R$ and a unique $(b_q^i)_{i \in \{0, \ldots, j+3q\}} \in \R^{j+3q+1}$ with $b_q^j = 0$ such that by setting
    \begin{equation}\label{eq:PQ_C}
        \vp_q(\sigma) = \sum_{i = 0}^{j+3q} b_q^i\, \Psi_i^{\GH}\lp\mb_0^{\frac{1}{4}}\sigma\rp, \quad \forall \sigma \in \R,
    \end{equation}
    the collection $(\vp_0,\ldots,\vp_q,\lambda_0,\ldots,\lambda_q)$ solves the sequence of problems $(\mathcal{R}^{\text{\textsc{(c)}}}_\ell)_{\ell = 0,\ldots,q}$.
\end{proposition}
%---
\begin{proof}
    Relations \eqref{eq:GHrelation} combined with Lemma \ref{lem:Bqpm} show that there exists  $(d_q^i)_{i \in \{0, \ldots, j+3q\}} \in \R^{j+3q+1}$ such that
    \begin{equation}\label{eq:1214}
        S_q^\vp(\sigma) = \sum_{i = 0}^{j+3q} d_q^i\, \Psi_i^{\GH}\lp\mb_0^{\frac{1}{4}}\sigma\rp \quad \forall \sigma \in \R.
    \end{equation}
    Namely, on the one hand,  relations \eqref{eq:GHrelation} indicate that the $\ell$-th derivative of $\Psi_i^{\GH}$ and the function obtained by multiplying $\Psi_i^{\GH}$ by $z^\ell$ can be expressed as a linear combination of Gauss-Hermite functions up to order $i + \ell$.
    On the other hand, Lemma \ref{lem:Bqpm} indicates that $\bA_q^-\vp_0$ and $(\lambda_\ell - \bA_\ell^-)\vp_{q-\ell}$, for $1 \le \ell \le q - 1$ can respectively be expressed as a linear combination of Gauss-Hermite function up to order $j+q+2$ and $j+3q-2\ell+2$, these two numbers being bounded by $j+3q$.

    Equation \eqref{eq:sysQ_C} has a solution in $\Lr^2(\R)$ if, and only if, $\lambda_q\,\vp_0 + S_q^\vp$ is orthogonal to $\Psi_j^{\GH}(\mb_0^{\frac{1}{4}}\cdot)$; This implies that $\lambda_q = -d_q^j$.
    Moreover, since the operator $-\partial_\sigma^2 + \mb_0\, \sigma^2 - \hbox{$ (2j+1)$}\sqrt{\mb_0}$ is diagonalizable and invertible on $\Vect(\Psi_i^{\GH}(\mb_0\,\cdot) \mid i \ge 0, \ i \neq j)$, we get $b_q^i = \frac{d_q^i}{2(i-j)}$ for $i \in \{0,\ldots,j+3q\} \setminus \{j\}$ and $b_q^j = 0$.
\end{proof}

\begin{remark}
    From the proof of Proposition \ref{pro:solQ_C} we can deduce a finite algorithm for the computation of the terms in the asymptotic expansion of the resonance quasi-pairs because the expression of $S_q^\vp$ in \eqref{eq:1214} involves a finite sum and because the computation of the solution $(\lambda_q,\vp_q)$ is explicit form the coefficients of $S_q^\vp$.
\end{remark}

The proof that the formal series  $
    \sum_{q\in\N} \lambda_q h^{\frac{q}{2}} $
and
$
    \sum_{q\in\N} \vp_q h^{\frac{q}{2}}$
obtained from Proposition \ref{pro:solQ_C} give rise to a family of resonance quasi-pairs in the sense of Definition \ref{def:quasi} can be achieved exactly as in Section \ref{sec:ccls_B} for case \textsc{(b)}.
Note that in order to use  Borel's Theorem and to obtain the required estimates, we have to show that $\vp_q \in \Sc(\R) \cap \Hr^2(\R,\e^{|\sigma|}\dd{\sigma})$.
This properties can be deduced directly from equation \eqref{eq:PQ_C}.

Finally, we can conclude with the proof of Theorem  \ref{th:(c)} in a way very similar to the one of Theorem  \ref{th:(b)}  as detailed in Section \ref{sec:ProofB}.

%============================
\section{Proximity between quasi-resonances and true resonances}
\label{sec:proxy}
%============================

\subsection{Separation of quasi-resonances, quasi-orthogonality of quasi-modes}
For the three cases \textsc{(a)}, \textsc{(b)}, and \textsc{(c)}, {\em cf.} Theorems  \ref{th:(a)}, \ref{th:(b)}, and \ref{th:(c)}, we have exhibited families of resonance quasi-pairs in the sense of Definition \ref{def:quasi}.
Namely, for each $j \ge 0$ and $m\ge1$, we have constructed a quasi-pair $(\ku_{p\pv j}(m),\uu_{p\pv j}(m))$ where $\ku_{p\pv j}(m) \in \R_+$ is a quasi-resonance and $\uu_{p\pv j}(m) \in \Hr_p^2(\R^2,\Omega)$ is a compactly supported quasi-mode.
Actually, to each quasi-resonance $\ku_{p\pv j}(m)$, we can associate two quasi-modes: $\uu_{p\pv j}(m)$ and its conjugate.
These quasi-modes are quasi-orthogonal with respect to $j$ and $m$, as stated in the next lemma.

We consider the Hilbert space $\Lr^2(\R^2,n(x)^{p+1}\dd{x})$ and denote its scalar product by
\begin{align*}
    \big\langle f,g\big\rangle 
    = \int_{\R^2} \overline{f(x)}\, g(x)\, n(x)^{p+1}\dd{x} \qquad
    \text{for } f,g \in \Lr^2(\R^2,n(x)^{p+1}\dd{x}).
\end{align*}

\begin{lemma}\label{lem:quasiPS}
    For all the three cases \textsc{(a)}, \textsc{(b)}, and \textsc{(c)}, and for all $i,j \ge 0$ and $m,m' \ge 1$, we have $\big\langle \overline{\uu_{p\pv i}(m)}, \uu_{p\pv j}(m')\big\rangle
        = 0$ and
    \begin{align*}
        \big\langle \uu_{p\pv i}(m), \uu_{p\pv j}(m')\big\rangle 
        =
        \begin{dcases}
            1                     & \text{if } m = m' \text{ and } i = j ,   \\
            0                     & \text{if } m \neq m',                    \\
            \Oo\lp m^{-\infty}\rp & \text{if } m = m' \text{ and } i \neq j.
        \end{dcases}
    \end{align*}
    For any $m\ge1$ and $i,j \ge 0$, we have the separation property
    \begin{align}
        \label{eq:sep}
        \ku_{p\pv i}(m)^2 - \ku_{p\pv j}(m)^2 =
        \begin{dcases}
            C^{\textsc{(a)}}_{ij} m^{\frac{4}{3}} + \Oc(m) &
            \text{in case \textsc{(a)} },                    \\
            C^{\textsc{(x)}}_{ij} m + \Oc(m^{\frac{1}{2}}) &
            \text{in cases \textsc{(b)},\textsc{(c)} },      \\
        \end{dcases}
    \end{align}
    with $C^{\textsc{(x)}}_{ij}\neq0$ if $i\neq j$.
\end{lemma}
%---
\begin{proof}
    The relation $\langle \uu_{p\pv i}(m), \uu_{p\pv j}(m')\rangle = 1$, for all $j \ge 0$ and $m \ge 1$, comes from the normalization of the quasi-mode in Definition \ref{def:quasi}.
    The relations $\langle \overline{\uu_{p\pv i}(m)}, \uu_{p\pv j}(m')\rangle = 0$, for all $i,j \ge 0$ and $m,m' \ge 1$, and $\langle \uu_{p\pv i}(m), \uu_{p\pv j}(m')\rangle = 0$, for all $i,j \ge 0$ and $m \neq m'$, $m,m' \ge 1$, are deduced from the identity $\int_0^{2\pi} \e^{\ic q \theta}\dd{\theta} = 0$ for all integer $q \neq 0$.

    For the last estimate, we consider $i \neq j$, $i,j \ge 0$, and $m \ge 1$.
    By construction, there exists $R_q \in \Lr^2(\R^2)$, for $q \in \{i,j\}$, such that $\Vert R_q\Vert_{\Lr^2(\R^2)} = \Oo(m^{-\infty})$ and
    \begin{align}
        \label{eq:Rq}
        \ku_{p\pv q}(m)^2\, n^{p+1}\, \uu_{p\pv q}(m)
        = -\div\lp n^{p-1}\, \nabla\uu_{p\pv q}(m)\rp - R_q.
    \end{align}
    Using the conjugate of this identity for $q=i$, we deduce:
    \begin{multline*}
        \ku_{p\pv i}(m)^2
        \int_{\R^2} \overline{\uu_{p\pv i}(m)}\, \uu_{p\pv j}(m)\, n^{p+1}\dd{x} =\\
        -\int_{\R^2} \div\lp n^{p-1}\, \nabla\overline{\uu_{p\pv i}(m)}\rp \uu_{p\pv j}(m) \dd{x} - \int_{\R^2} \overline{R_i}\, \uu_{p\pv j}(m)\dd{x}.
    \end{multline*}
    Integrating by parts and using again \eqref{eq:Rq}, we get
    \begin{align*}
        \lp\ku_{p\pv i}(m)^2 - \ku_{p\pv j}(m)^2\rp
        \big\langle\uu_{p\pv i}(m), \uu_{p\pv j}(m)\big\rangle_{\Lr^2(\R^2)}
        = \int_{\R^2}
        \lp\overline{\uu_{p\pv i}(m)}\, R_j - \overline{R_i}\, \uu_{p\pv j}(m)\rp\dd{x}.
    \end{align*}
    Taking the modulus and using Cauchy-Schwarz inequality, we obtain
    \begin{align*}
        \lv\ku_{p\pv i}(m)^2 - \ku_{p\pv j}(m)^2\rv\
        \lv\big\langle\uu_{p\pv i}(m), \uu_{p\pv j}(m)\big\rangle_{\Lr^2(\R^2)}\rv
         & \le \Vert R_i\Vert_{\Lr^2(\R^2)} + \Vert R_j\Vert_{\Lr^2(\R^2)} \\
         & = \Oo\lp m^{-\infty}\rp .
    \end{align*}
    Then, we use the separation property \eqref{eq:sep} (which is an obvious consequence of the asymptotic formulas for $\ku_{p\pv j}(m)$ obtained in each case) and
    finally get the estimate $\langle \uu_{p\pv i}(m), \uu_{p\pv j}(m')\rangle  = \Oo(m^{-\infty})$.
\end{proof}

\subsection{Spectral-like theorems for resonances}
We have constructed well separated quasi-pairs for the operator
\[
    P \coloneqq -n^{-p-1}\div(n^{p-1}\, \nabla \cdot)
\]
with domain $\Hr_p^2(\R^2,\Omega)$ on the Hilbert space $\Lr^2(\R^2,n(x)^{p+1}\dd{x})$.
The operator $P$ is self-adjoint and its spectrum $\Sigma(P)$ reduces to its essential spectrum, equal to $[0,+\infty)$.
If we apply the spectral theorem \cite[Theorem 5.9]{HisSig96} to our quasi-resonances for the operator $P$, we get that for each quasi-resonance $\ku_{p\pv j}(m)$ there exists an interval $I$ of length $\Oo(m^{-\infty})$ such that the intersection $\Sigma(P)\cap I$ is non empty, which is useless, since we know already that $\Sigma(P)=[0,+\infty)$.

If the operator $P$ has been defined as the Dirichlet realization of $-n^{-p-1}\div(n^{p-1}\, \nabla \cdot)$ on a bounded open set containing $\overline{\Omega}$, then its spectrum would have been discrete. In such case, the application of the spectral theorem would be more significant. Nevertheless, this procedure of cut-off would not inform us about resonances.

That is why we need to use a spectral-like theorem for resonances. Two statements are available in the literature: one from \textsc{Tang} and \textsc{Zworski} \cite{TanZwo98}, and another from~\textsc{Stefanov} \cite{Ste99}.
Those theorems lie in the \emph{black box scattering} framework. We are going to present the main assumptions and results of these papers in a simplified way, convenient for our application.

In dimension $2$, the main ingredients are
\begin{itemize}
    \item A complex Hilbert space $\Hc$ with orthogonal decomposition (with positive $\varrho_0$)
          \[
              \Hc = \Hc_{\varrho_0} \oplus \Lr^2(\R^2\setminus B(0,\varrho_0))
          \]
    \item A family of unbounded selfadjoint operators $h\mapsto P(h)$ on $\Hc$ with domain independent of $h$, whose projection onto $\Lr^2(\R^2\setminus B(0,\varrho_0))$ coincides with $\Hr^2(\R^2\setminus B(0,\varrho_0))$.
\end{itemize}
We introduce the following assumptions
\begin{equation}
    \label{eq:H1}
    \II_{B(0,\varrho_0)} (P(h)-\ic)^{-1}\quad\mbox{compact}\quad
    \Hc\to\Hc
    \tag{H1}
\end{equation}
and
\begin{equation}
    \label{eq:H2}
    \II_{\R^2\setminus B(0,\varrho_0)} P(h) u
    = - h^2\Delta u\on{\R^2\setminus B(0,\varrho_0)}.
    \tag{H2}
\end{equation}
Then, we choose $\varrho\gg\varrho_0$ and periodize $P(h)$ outside $B(0,\varrho_0)$, obtaining an operator $P^\sharp(h)$ on the Hilbert space
\[
    \Hc^\sharp = \Hc_{\varrho_0} \oplus \Lr^2(M\setminus B(0,\varrho_0))
    \quad\mbox{with}\quad M = (\R/\varrho\Z)^2
\]
Denoting by $N(P^\sharp(h),I)$ the number of eigenvalues in $I$, we write the third assumption as
\begin{equation}
    \label{eq:H3}
    N(P^\sharp(h),[-\lambda,\lambda])
    =  \Oc\lp (\lambda/h^2)^{n^\sharp/2}\rp
    \quad \lambda\to\infty,\quad\mbox{for some}\quad n^\sharp\ge 2.
    \tag{H3}
\end{equation}
Let us denote by $\Zc(P(h))$ the set of poles of the resolvent $z\mapsto(P(h)-z)^{-1}$. In dimension~$2$, this set is a subset of the Riemann logarithmic surface and its elements satisfy $\arg z<0$ with our convention for the definition of resonances.

Now we can state a simplified version of the main result of \cite{TanZwo98}:
\begin{theorem}[\cite{TanZwo98}]
    \label{th:TZ}
    Let $P(h)$ satisfy hypotheses \eqref{eq:H1}, \eqref{eq:H2}, and \eqref{eq:H3}. Assume that there exists for any $h\in(0,h_0]$ a quasi-pair $(E(h),u(h))$ with $E(h)\subset[E_0-h,E_0+h]$ for some real $E_0$, and with $u(h)$ normalized in $\Hc$ and compactly supported independently of $h$. The quasi-pairs are supposed to satisfy the residue estimate
    \[
        \| \lp P(h) - E(h) \rp u(h) \|_\Hc = \Oc(h^\infty).
    \]
    Then, for any $h\in(0,h'_0]$ with a positive $h'_0$ small enough, there exists a resonance pole $z(h)\in\Zc(P(h))$ such that
    \[
        |E(h)-z(h)| = \Oc(h^\infty).
    \]
\end{theorem}

The result in \cite{Ste99} is more precise but requires one more hypothesis, according to which the number of resonance poles is not too large: For some positive integers $N$ and $N'$
\begin{equation}
    \label{eq:H4}
    \Card \left\{ z\in\Zc(P(h)),\;\; a_0\le |z|\le b_0,\;\; -\Im z < h^N \right\}
    \le C_{a_0,b_0} h^{N'}.
    \tag{H4}
\end{equation}

Our simplified version of the main result of \cite{Ste99} follows:
\begin{theorem}[\cite{Ste99}]
    \label{th:Ste}
    Let $\Hg$ be a infinite subset of $(0,1]$ with accumulation point at $0$.
    Let $P(h)$ satisfy hypotheses \eqref{eq:H1}, \eqref{eq:H2}, \eqref{eq:H3}, and \eqref{eq:H4}. Assume that, for any $h\in(0,h_0]\cap\Hg$, there exists  $d$ quasi-pair $(E_\ell(h),u_\ell(h))$ with $E_1(h)=\ldots=E_d(h)\in[a_0,b_0]$, and with $u_\ell(h)$ compactly supported independently of $h$, and almost orthonormal: $|\langle u_i(h),u_\ell(h)\rangle_{\Hc}-\delta_{i\ell}|=\Oc(h^\infty)$. The quasi-pairs are supposed to satisfy the residue estimate
    \[
        \| \lp P(h) - E_\ell(h) \rp u_\ell(h) \|_\Hc = \Oc(h^\infty),\quad \ell=1,\ldots,d.
    \]
    Then for any $h\in(0,h'_0]\cap\Hg$ with a positive $h'_0$ small enough, there exists $d$ resonance poles $z_\ell(h)\in\Zc(P(h))$ with repetition according to multiplicity, such that
    \[
        |E_\ell(h)-z_\ell(h)| = \Oc(h^\infty),\quad \ell=1,\ldots,d.
    \]
\end{theorem}

The distinction between the two above theorems is the consideration of multiplicity in Theorem \ref{th:Ste}. The multiplicity of a resonance pole $z_0$ is understood as the rank of the operator
\[
    \frac{1}{2\ic\pi} \int_{|z-z_0|=\varepsilon} (P(h)-z)^{-1} \dd z
\]
for $\varepsilon>0$ small enough to isolate the pole $z_0$ and $(P(h)-z)^{-1}$ is the meromorphic extension of the resolvent \cite[Definition 4.6]{DyaZwo19}.

\subsection{Application of the spectral-like theorems to disks with radially varying index}
We apply the above theorems to our situation. We set
\[
    P(h) = h^2P \quad\mbox{with}\quad
    P=-n^{-p-1}\div(n^{p-1}\, \nabla \cdot) \quad\mbox{on}\quad
    \Hc = \Lr^2(\R^2,n(x)^{p+1}\dd{x}).
\]
The subset $\Hg$ is $\{h=\frac{1}{m},\ m\in\N^*\}$.
Hypotheses \eqref{eq:H1} and \eqref{eq:H2} are easy to check.
Concerning \eqref{eq:H3}, by using the max--min principle for eigenvalues \cite[Theorem 11.12]{Hel13} and comparing the eigenvalues of $P^\sharp(h)$ with $(-h^2\Delta)^\sharp$ on a large torus $M = (\R / \varrho\Z)^2$ for $\varrho \gg R$, we get that the counting function $N(P^\sharp(h),[-\lambda,\lambda]) = \Oo(\lambda/h^2)$ for $\lambda \to \infty$, which yields $n^\sharp = 2$.

Concerning \eqref{eq:H4}, we simply have to use the main theorem in \cite{Vodev94} (with $\phi(t)=t^2$ and $a=1+\varepsilon$).

Then to bridge our families of resonance quasi-pairs with the formalism of \cite{Ste99}, we set, for any chosen $p\in\{\pm1\}$ and any chosen $j\in\N$:
\[
    E_\ell(h) = h^2\,\ku_{p\pv j}(\tfrac{1}{h})^2,\quad h\in\Hg,\quad\ell=1,2
\]
with
\[
    u_1(h) = \uu_{p\pv j}(\tfrac{1}{h}) \quad\mbox{and}\quad
    u_2(h) = \overline{\uu_{p\pv j}(\tfrac{1}{h})}, \quad h\in\Hg.
\]
Then, as $h$ tends to $0$, the energy $E_\ell(h)$ converges to $1/(Rn(R))^2$ in cases \textsc{(a)} and \textsc{(b)}, and to $1/(R_0n(R_0))^2$ in case \textsc{(c)}. Applying Theorem \ref{th:Ste} and coming back to resonances by the formula
\[
    k_m = m\sqrt{z_1(\tfrac{1}{m})} \quad\mbox{and}\quad
    k'_m = m\sqrt{z_2(\tfrac{1}{m})} ,\quad m\ge1
\]
we have proved:

\begin{theorem}\label{th:QuasiRes}
    For $p \in \{\pm 1\}$, $j \in \N$, and $m$ large enough, there exist two resonances $k_m$ and $k'_m$ (counted with multiplicity) such that, as $m \rightarrow +\infty$, we have
    \begin{align*}
        \max \lp \lv\ku_{p\pv j}(m) - k_m\rv, \lv\ku_{p\pv j}(m) - k'_m\rv\rp = \Oo\lp m^{-\infty}\rp .
    \end{align*}
\end{theorem}

\begin{remark}
    \emph{(i)} It is plausible that modes associated with the true resonances $k_m$ and $k'_m$ have $m$ as polar mode index.
    The proof of this would require to apply a spectral theorem to the family of one dimensional resonance problems \eqref{eq:Pprad}, which seemingly does not enter the general framework of \cite{TanZwo98} or \cite{Ste99}.
    Nevertheless, finite element computations performed with perfectly matched layers displayed numerical modes complying with the structure of quasi-modes (see \cite[Chapter 7]{MoitierPhD}).

    \emph{(ii)} Throughout the paper we have assumed that $p \in \{\pm 1\}$, because of the physical motivation, but without any change, everything is true for $p \in \R$.
    
    \emph{(iii)} The assumption that the optical index $n$ of the cavity is real is necessary to deduce Theorem \ref{th:QuasiRes} from the black box scattering theory since the latter requires that the operator $u \mapsto -n^{-p-1}\, \div(n^{p-1}\, \nabla u)$ is self-adjoint.
    However, the statements Theorem \ref{th:(a)}, \ref{th:(b)}, and \ref{th:(c)} still hold if we allow complex optical indices according to the following relaxed assumptions:
    For Theorem \ref{th:(a)}, we can weaken the hypotheses on the optical index to $\Re\sqrt{1 - n_0^{-2}} > 0$ and $\Re\sqrt{\kb} > 0$, and for Theorem \ref{th:(b)}, and \ref{th:(c)}, we can replace the condition on effective  hessian by $\Re\sqrt{\mb} > 0$.
\end{remark}

\appendix

\makeatletter
\renewcommand\thetheorem{\thesection.\@arabic\c@theorem}
\makeatother

\section{Technical lemmas}\label{App:A}

\subsection{Explicit solutions to some differential equations}

\begin{lemma}\label{lem:exp}
    For all $\ell\in\N$, let denote by $\gamma_\ell$ the mapping $ z\in\R \mapsto z^\ell\e^{-z}$
    and for  $d\ge 1$ by $\Ec_d$ the set $\{\gamma_\ell ; \ell=0,\ldots,d\}$.
    The operator $-\partial_z^2 + 1$ is a bijection from the vector-space $\Vect(\Ec_d \setminus \{\gamma_0\})$ to the vector-space $\Vect(\Ec_{d-1})$.
\end{lemma}
%---
\begin{proof}
    For all $\ell\ge 0$, we readily obtain
    $
        (-\partial_z^2 + 1)\gamma_\ell = \ell(\ell-1)\, \gamma_{\ell-2} - 2\ell\, \gamma_{\ell-1}.
    $
    Considering the vector basis $(\gamma_1,\ldots,\gamma_d)$ and $(\gamma_0,\ldots,\gamma_{d-1})$  of $\Vect(\Ec_d \setminus \{\gamma_0\})$ and $\Vect(\Ec_{d-1})$ respectively,
    the matrix of $-\partial_z^2 + 1$ from $\Vect(\Ec_d \setminus \{\gamma_0\})$ to $\Vect(\Ec_{d-1})$   is an upper triangular matrix with a determinant equal to $(-2)^d d! \neq 0$.
\end{proof}
%--------------------------------------------------------------------

\begin{lemma}\label{lem:airy}
    For all $\ell\in\N$, let denote by $\alpha_\ell$ the mapping $ z\in\R \mapsto z^\ell\As(z)$ and by $\beta_\ell$ the mapping $ z\in\R \mapsto z^\ell\As'(z)$ where $\As$ is the mirror Airy's function.
    For all $d\in\N$, let $\Ac_d$ be the set $\{\alpha_\ell, \beta_\ell ; \ell=0,\ldots,d\}$.
    The operator $-\partial_z^2 - z$ is a bijection from the vector-space $\Vect(\Ac_d \setminus \{\alpha_0\})$ to the vector-space $\Vect(\Ac_d \setminus \{\beta_d\})$.
\end{lemma}
%---
\begin{proof}
    From the definition of Airy's function, we have $\As'' = -z\As$ and therefore, for $\ell \ge 0$,
    \begin{align*}
         & (-\partial_z^2 - z)\alpha_\ell = -\ell(\ell-1)\, \alpha_{\ell-2}- 2\ell\, \beta_{\ell-1},
        \\
         & (-\partial_z^2 - z)\beta_\ell = -\ell(\ell-1)\, \beta_{\ell-2} + (2\ell+1)\, \alpha_\ell.
    \end{align*}
    Considering the vector basis $(\beta_0,\alpha_1,\beta_1,\ldots,\alpha_d,\beta_d)$ of $\Vect(\Ac_d \setminus \{\alpha_0\})$ and the vector basis $(\alpha_0,\beta_0,\alpha_1,\beta_1,\ldots,\alpha_d)$ of $\Vect(\Ac_d \setminus \{\beta_d\})$, the matrix of $-\partial_z^2 - z$ considered from $\Vect\lp\Ac_d \setminus \{\alpha_0\}\rp$ to $\Vect\lp\Ac_d \setminus \{\beta_d\}\rp$ is an upper triangular matrix with a determinant equal to $(-1)^d(2d+1)! \neq 0$.
\end{proof}
%--------------------------------------------------------------------

\subsection{Half harmonic oscillator}
We recall that we denote by $\Lr^2(\R_-,\omega_x)$ and $\Hr^\ell(\R_-,\omega_x)$ the weighted Sobolev spaces with measure $\omega_x(\sigma) \dd{\sigma}$ where $\omega_x : \sigma \mapsto \exp\lp 2^x |\sigma|\rp$ for $x$ real
and that $\Psi_{2j+1}^{\GH}$ refers to the Gauss-Hermite function of order $2j+1$, see \cite{AbrSte64, Olv97}.

\begin{lemma}\label{lem:GHsolPoids}
    Let $\beta \in \R$, $\theta > 0$, and $j \in \N$.
    For any $S \in \Lr^2(\R_-,\omega_\beta) \cap \Vect(\Psi_{2j+1}^{\GH})^\perp$ there exists a unique solution to the problem: Find $w \in \Hr^2(\R_-) \cap \Vect(\Psi_{2j+1}^{\GH})^\perp$ such that
    \begin{align}\label{eq:HHO}
        \begin{cases}
            -w''(x) + (x^2 - 4j - 3)w(x) = S(x) & \forall x\in (-\infty,0) \\
            w(0) = 0
        \end{cases}.
    \end{align}
    Moreover, this solution belongs to  $\Hr_0^1(\R_-,\omega_\beta) \cap \Hr^2(\R_-,\omega_{\beta-\theta})$.
\end{lemma}
%---
\begin{proof}
    Existence and unicity rely on the fact that the family $(\Psi_{2\ell+1}^{\GH})_{\ell\in\N}$ is a Hilbert basis of $\Lr^2(\R_-)$ and that the half harmonic oscillator operator is diagonalizable on $\Vect(\Psi_{2\ell+1}^{\GH} \mid \ell\in\N)$.
    The solution to problem \eqref{eq:HHO} can be written as
    \begin{align}\label{eq:GHseriesSolution}
        w = \sum_{\ell=0,\ \ell\neq j}^{+\infty} \frac{1}{4(\ell-j)}\, (S,\, \varsigma_\ell\Psi_{2\ell+1}^{\GH})_{\Lr^2(\R_-)}\, \varsigma_\ell\Psi_{2\ell+1}^{\GH}
    \end{align}
    where $\varsigma_\ell = \Vert\Psi_{2\ell+1}^{\GH}\Vert_{\Lr^2(\R_-)}^{-1}$ and we clearly have $w \in \Hr^2(\R_-) \cap \Vect(\Psi_{2j+1}^{\GH})^\perp$.

    We set $J \coloneqq \sqrt{4(j+1)+2^{2\beta-2}}$ so that, for all $x \le -J$, we have $V(x) \coloneqq x^2-4j-3-2^{2\beta-2} \ge 1$.
    Let $\phi \in \Cc^\infty(\R)$ such that $0 \le \phi \le 1$, $\phi(x) = 0$ for all $x \le 0$, and $\phi(x) = 1$ for all $x \ge 1$, let $b = (1 + 2^\beta)\, \max_{\R} |\phi'|$ and let $a > J+2b$.
    We define a cut-off function $\chi_a \in \Cc_\comp^\infty(\R_-)$ by
    \begin{align*}
        \chi_a(x) = \phi(b^{-1}(x+a)) \cdot \phi(-b^{-1}(x+J)), \quad
        \forall x \in \R_-.
    \end{align*}
    We also define $\chi(x) = \phi(-b^{-1}(x+J))$ for $x\in \R_-$.
    Note that, for all $a > J+2b$, we have $|\chi_a'| \le C$ where $C = (1 + 2^\beta)^{-1}$.
    Let also  $\wh{w}  \coloneqq w\, \omega_{\beta-1}$ and $\wh{S} \coloneqq  S\, \omega_{\beta-1}$.
    Multiplying both sides of equation \eqref{eq:HHO} by $\chi_a\, w\, \omega_\beta$ and integrating over $\R_-$, yields
    \begin{align*}
        \int_{-\infty}^0 \Big( w' \lp\chi_a\, w\, \omega_\beta\rp' + \chi_a\, \lp x^2-4j-3\rp \wh{w}^2  \Big)\dd{x} = \int_{-\infty}^0 \chi_a\, \wh{S}\, \wh{w} \dd{x}.
    \end{align*}
    Since $w'\, \omega_{\beta-1} = \wh{w}' + 2^{\beta-1}\wh{w}$ and $w'(x) (w(x)\, \omega_\beta)' = {\wh{w}'(x){}}^2 - 2^{2\beta-2}\wh{w}^2(x)$, we deduce that
    \begin{align}\label{eq:estiGH1}
        \int_{-\infty}^0  \chi_a\, {\wh{w}'{}}^2  + \chi_a\,  V\, \wh{w}^2  \dd{x}
        + \int_{-\infty}^0 \chi_a' \wh{w} \lp \wh{w}' + 2^{\beta-1}\wh{w}\rp \dd{x}
        = \int_{-\infty}^0 \chi_a\, \wh{S}\, \wh{w} \dd{x}.
    \end{align}
    For the first term on the left hand side of \eqref{eq:estiGH1}, since $\chi_a\, V \ge \chi_a$, we have
    \begin{align}\label{eq:lGHSP1}
        \int_{-\infty}^0  \chi_a\, {\wh{w}'{}}^2  + \chi_a\, V\, \wh{w}^2 \dd{x}
        \ge \int_{-\infty}^0  \chi_a \lp{\wh{w}'{}}^2  + \wh{w}^2\rp \dd{x}.
    \end{align}
    Then, for the second term on the left hand side of \eqref{eq:estiGH1}, since $\chi_a' \ge -C$, $1 \ge \chi_a$, and $b$ is such that $C (1+2^\beta) = 1$, we have
    \begin{align}
        \int_{-\infty}^0 \chi_a' \wh{w} \lp \wh{w}' + 2^{\beta-1}\wh{w}\rp \dd{x}
         & \ge -C \int_{-\infty}^0 \wh{w}\, \wh{w}' + 2^{\beta-1}\wh{w}^2 \dd{x} \nonumber                    \\
         & \ge -\frac{C}{2} \int_{-\infty}^0 {\wh{w}'{}}^2  + (1+2^\beta)\wh{w}^2 \dd{x} \nonumber            \\
         & \ge -\frac{1}{2} \int_{-\infty}^0 \chi_a \lp{\wh{w}'{}}^2  + \wh{w}^2\rp \dd{x} \label{eq:lGHSP2}.
    \end{align}
    For the last term on the right hand side of \eqref{eq:estiGH1}, since $\chi_a^2 \le \chi_a$, we have
    \begin{align}\label{eq:lGHSP3}
        \int_{-\infty}^0 \chi_a\, \wh{S}\, \wh{w} \dd{x}
        \le \Vert\wh{S}\Vert_{\Lr^2(\R_-)} \lp\int_{-\infty}^0 \chi_a \wh{w}^2 \dd{x}\rp^{\frac{1}{2}}
        \le \Vert\wh{S}\Vert_{\Lr^2(\R_-)}\, \Nc_w(a)
    \end{align}
    where $\Nc_w(a) = \sqrt{\int_{-\infty}^0 \chi_a \lp{\wh{w}'{}}^2  + \wh{w}^2\rp \dd{x}}$.
    Combining the estimates \eqref{eq:lGHSP1}, \eqref{eq:lGHSP2}, and \eqref{eq:lGHSP3} yields
    \begin{align*}
        \Nc_w(a)^2 \le 2\Vert\wh{S}\Vert_{\Lr^2(\R_-)}\, \Nc_w(a).
    \end{align*}
    The function $a \in (J+2,+\infty) \mapsto \Nc_w(a)$ is not negative and not decreasing, so the function is either always zero or positive for $a$ large enough but in any cases we have
    \begin{align*}
        \Nc_w(a) \le 2\Vert\wh{S}\Vert_{\Lr^2(\R_-)}.
    \end{align*}
    By letting $a$ tends towards $+\infty$, we obtain that
    \begin{align*}
        \int_{-\infty}^0 \chi \lp {\wh{w}'{}}^2  + \wh{w}^2\rp \dd{x} \le 4\Vert\wh{S}\Vert_{\Lr^2(\R_-)}^2
    \end{align*}
    which implies that $\wh{w}$ belongs to $\Hr_0^1(\R_-)$.
    It follows that $w$ belongs to $\Lr^2(\R_-,\omega_\beta)$.
    From the relation $w'\, \omega_{\beta-1} = -2^{\beta-1}\wh{w}-\wh{w}'$, we deduce that $w \in \Hr_0^1(\R_-,\omega_\beta)$.

    Finally, using the relation $w'' = (x^2-4j-3)w - S$, we get
    \begin{align*}
        \int_{-\infty}^0 {w''}^2\, \omega_{\beta-\theta}\dd{x} \le C \int_{-\infty}^0 \lp w^2 + S^2\rp\, \omega_\beta\dd{x}
    \end{align*}
    where $C = \max_{x\in\R_-} (x^2-4j-3)^2\, \omega_{\beta+\wt{\theta}}^{-1}(x) < +\infty$ with $\wt{\theta} = \ln(1-2^{-\theta}) / \ln(2)$.
    This shows that $w \in \Hr^2(\R_-,\omega_{\beta-\theta})$.
    Note that the constant $\beta$ is replaced by $\beta-\theta$ by the need to take into account the coefficient $x^2-4j-3$.
\end{proof}
%--------------------------------------------------------------------

\subsection{Borel's Theorem}
Our construction of quasi-modes requires to find a smooth function given its Taylor expansion.
This can be achieved using a Borel's like theorem on the spaces of Schwartz functions $\Sc(\R_\pm)$.
We denote by $p_{\alpha,\beta}(f) = \sup_{x\in\R_\pm} |x^\alpha \partial_x^\beta f(x)|$,  $\alpha, \beta \in \N$, the usual family of semi-norms over  $\Sc(\R_\pm)$.

\begin{lemma}%[Borel's Theorem]
    \label{th:Borel}
    Let $(f_q)_{q\in\N}$ be a sequence of functions where  $f_q\in\Sc(\R_\pm)$ for all $q\in\N$.
    There exists $f \in \Cc^\infty([0,1],\Sc(\R_\pm))$ such that
    \begin{align*}
        \partial_t^q f(0,x) = f_q(x), \quad \forall x\in\R_\pm .
    \end{align*}
\end{lemma}
%---
\begin{proof}
    This proof is inspired by the proof of \cite[theorem 1.2.6]{HorI} where smooth functions with compact support are replaced by Schwartz functions.

    Let $g\in\Cc_\comp^\infty(\R)$ be a smooth cut-off function such that $g(t)=1$ for all $t \in [-1,1]$.
    For each $q\in\N$  we introduce the function
    \[
        g_q :  (t,x) \in \R\times\R_\pm \longmapsto g\lp\ve_q^{-1}\, t\rp\, \frac{t^q}{q!}\, f_q(x)
    \]
    for some positive number $\ve_q$ that will be specified later on.
    For all $d, \alpha, \beta \in\N$, we have
    \begin{align}\label{eq:1645}
        x^\alpha\partial_t^d\partial_x^\beta\ee g_q(t,x) =
        \varepsilon_q^{q-d}\, G_q^{(d)}(\ve_q^{-1}\, t)\,
        x^\alpha f_q^{(\beta)}(x), \quad \forall (t,x)\in \R\times\R_\pm ,
    \end{align}
    where $G_q(s) = g(s)\, \frac{s^q}{q!}$.
    It follows that
    $\lv x^\alpha \partial_t^d \partial_x^\beta\ee g_q(t,x)\rv \le C_{q,d}^{\alpha,\beta}\, \ve_q^{q-d}$
    where
    \begin{align*}
        C_{q,d}^{\alpha,\beta} = \sup_{s\in\R} \lv G_q^{(d)}(s)\rv \ p_{\alpha,\beta}(f_q) < +\infty .
    \end{align*}
    By choosing
    $
        \ve_q = \min\lp \ve_{q-1}, \min_{d+\alpha+\beta < q} (2^q\, C_{q,d}^{\alpha,\beta})^{-\frac{1}{q-d}} \rp
    $
    , $q \ge 1$ and $\ve_0 = 1$ we obtain that  $\lv x^\alpha \partial_t^d \partial_x^\beta g_q(t,x)\rv \le 2^{-q}$ for all $d,\alpha,\beta\in\N$ and for all $q > d+\alpha+\beta$.
    Therefore, the sum $$f = \sum_{q\ge 0} g_q$$ is well defined because the series converge absolutely.  Its successive derivatives are equal to the sum of the derivatives of $g_g$; As a consequence, $f\in\Cc^\infty([0,1]\times\R_\pm)$.  Moreover,  from the estimate
    \begin{align*}
        p_{\alpha,\beta}\lp \partial_t^d f(t,\cdot)\rp \le
        \sum_{q=0}^{d+\alpha+\beta} C_{q,d}^{\alpha,\beta}\, \ve_q^{q-d} + 2^{-d-\alpha-\beta},
         &  & \forall t\in [0,1],\  \forall d, \alpha, \beta\in\N
    \end{align*}
    we obtain $f \in \Cc^\infty([0,1],\Sc(\R_\pm))$.
    From \eqref{eq:1645}, we deduce that for all $d\in\N$
    \begin{align*}
        \partial_t^d f(0,x) = \sum_{q=0}^{+\infty} \ve_q^{q-d}\, G_q^{(d)}(0)\, f_q(x), &  & \forall x\in\R_\pm .
    \end{align*}
    Since $g$ is constant equal to $1$ around $t=0$, we have $G_q^{(d)}(0) = \delta_{q,d}$ where $\delta_{q,d}$ is the Kronecker symbol.
    This implies that $\partial_t^d f(0,x) = f_d(x)$.
\end{proof}
%--------------------------------------------------------------------

By Taylor's formula with integral remainder we deduce immediately the following result.

\begin{lemma}\label{lem:Taylor}
    Let $f$ be a function belonging to $\Cc^\infty([0,1],\Sc(\R_\pm))$. For all integer $N\ge1$ there exists $R_N \in \Cc^\infty([0,1],\Sc(\R_\pm))$ such that
    \begin{align*}
        f(t,x) = \sum_{q=0}^{N-1} \frac{\partial_t^q f(0,x)}{q!} \,t^q
        + t^{N} R_N(t,x), \quad \forall (t,x) \in [0,1] \times \R_\pm .
    \end{align*}
\end{lemma}

\subsection{Additional result}

\begin{lemma}
    \label{lem:dec}
    Let $F:(t,\sigma)\mapsto F(t,\sigma)$ a function in $\Cc^\infty([0,1],\Sc(\R_-))$. Then
    \[
        \int_{-\infty}^{-\delta/t} |F(t,\tau)|^2 \,\dd\tau = \Oc(t^\infty) \quad\mbox{as}\quad t\to 0.
    \]
    The same result holds with $\R_-$ replaced by $\R_+$ and $(-\infty,-\delta/t)$ replaced by $(\delta/t,\infty)$.
\end{lemma}

\begin{proof}
    It suffices to notice that for any $N\ge 1$, there exists $C_N$ such that
    \[
        |\tau^N F(t,\tau) | \le C_N,\quad \mbox{for all }\ (t,\tau)\in [0,1]\times\R_-.
    \]
    Hence $\int_{-\infty}^{-\delta/t} |F(t,\tau)|^2 \,\dd\tau \le \lp\frac{t}{\delta}\rp^{2N-1}$, which proves the lemma.
\end{proof}

\end{document}